\documentclass{article}
\usepackage{amsmath,amsfonts,amsthm,amssymb,amscd,color,xcolor}
\usepackage{graphicx,amscd,mathrsfs,wrapfig,mathrsfs,lipsum}
\usepackage{microtype}
\usepackage{float}
\usepackage{tikz}
\usepackage{multicol}
\usepackage{caption}
\usetikzlibrary{arrows}
\usepackage{hyperref}

\colorlet{darkblue}{blue!50!black}

\hypersetup{
    colorlinks,%
    citecolor=darkblue,%
    filecolor=red,%
    linkcolor=darkblue,%
    urlcolor=blue,%
    pdfnewwindow=true,%
    pdfstartview={FitH}
}

\renewcommand{\Re}{\mathop{\rm Re}\nolimits}
\renewcommand{\Im}{\mathop{\rm Im}\nolimits}
\newcommand{\Leb}{\mathop{\mathrm{Leb}}\nolimits}

\newcommand{\p}{\partial}
\newcommand{\e}{\varepsilon}
\newcommand{\bfe}{{\vec{\varepsilon}}}

\newcommand{\D}{{\mathbb D}}

\newcommand{\C}{{\mathbb C}}

\newcommand{\R}{{\mathbb R}}
\newcommand{\Z}{{\mathbb Z}}
\newcommand{\IP}{{\mathbb P}}

\newcommand{\I}{{\mathbb I}}
\newcommand{\E}{{\mathbb E}}
\newcommand{\T}{{\mathbb T}}

\newcommand{\N}{{\mathbb N}}

\newcommand{\ty}{\infty}

\newcommand{\mmmm}{{\mathfrak m}}
 
\newcommand{\bbar}{\boldsymbol{|}}

\newcommand{\XXX}{{\boldsymbol{X}}}

\newcommand{\barr}{{\boldsymbol|}}

\newcommand{\scH}{{\mathscr H}}

\newcommand{\BB}{{\cal B}}

\newcommand{\DD}{{\cal D}}
\newcommand{\EE}{{\cal E}}
\newcommand{\FF}{{\cal F}}
\newcommand{\GG}{{\cal G}}
\newcommand{\HH}{{\cal H}}
\newcommand{\II}{{\cal I}}

\newcommand{\KK}{{\cal K}}
\newcommand{\LL}{{\cal L}}

\newcommand{\NN}{{\cal N}}
\newcommand{\OO}{{\cal O}}
\newcommand{\PP}{{\cal P}}

\newcommand{\RR}{{\cal R}}

\newcommand{\UU}{{\cal U}}
\newcommand{\VV}{{\cal V}}
\newcommand{\WW}{{\cal W}}
\newcommand{\XX}{{\cal X}}
\newcommand{\YY}{{\cal Y}}
\newcommand{\ZZ}{{\cal Z}}

\newcommand{\lag}{\langle}
\newcommand{\rag}{\rangle}

\newcommand{\dd}{{\textup d}}

\newcommand{\PPPP}{{\mathfrak P}}

\newcommand{\FFFF}{{\mathfrak F}}
\newcommand{\XXXX}{{\mathfrak X}}
\newcommand{\LLLL}{{\mathfrak L}}
\newcommand{\MMMM}{{\mathfrak M}}
\newcommand{\NNNN}{{\mathfrak N}}

\newcommand{\lspan}{\mathop{\rm span}\nolimits}

\newcommand{\supp}{\mathop{\rm supp}\nolimits}
\newcommand{\diver}{\mathop{\rm div}\nolimits}

\newcommand{\dist}{\mathop{\rm dist}\nolimits}

\newcommand{\Id}{\mathop{\rm Id}\nolimits}
\newcommand{\diam}{\mathop{\rm diam}\nolimits}
\newcommand{\Ker}{\mathop{\rm Ker}\nolimits}

\theoremstyle{plain}

\newtheorem{theorem}{Theorem}[section]
\newtheorem{lemma}[theorem]{Lemma}
\newtheorem{proposition}[theorem]{Proposition}
\newtheorem{corollary}[theorem]{Corollary}
\theoremstyle{definition}
\newtheorem{definition}[theorem]{Definition}

\theoremstyle{remark}

\newtheorem{remark}[theorem]{Remark}
\newtheorem{example}[theorem]{Example}

\numberwithin{equation}{section}

\begin{document}
\author{Sergei Kuksin\footnote{Institut de Math\'emathiques de Jussieu--Paris Rive Gauche, CNRS, Universit\'e Paris Diderot, UMR 7586, Sorbonne Paris Cit\'e, F-75013, Paris, France \& School of Mathematics, Shandong University, Jinan, PRC \& Saint Petersburg State University, Universitetskaya nab., St. Petersburg, Russia; e-mail: \href{mailto:Sergei.Kuksin@imj-prg.fr}{Sergei.Kuksin@imj-prg.fr}} \and Vahagn~Nersesyan\footnote{Laboratoire de Math\'ematiques, UMR CNRS 8100, UVSQ, Universit\'e Paris-Saclay, 45, av. des Etats-Unis, F-78035 Versailles, France;  e-mail: \href{mailto:Vahagn.Nersesyan@math.uvsq.fr}{Vahagn.Nersesyan@math.uvsq.fr}} \and
Armen Shirikyan\footnote{Department of Mathematics, University of Cergy--Pontoise, CNRS UMR 8088, 2 avenue Adolphe Chauvin, 95302 Cergy--Pontoise, France \& Centre de Recherches Math\'ematiques, CNRS UMI 3457, Universit\'e de Montr\'eal, Montr\'eal,  QC, H3C 3J7, Canada \& Department of Mathematics and Statistics, McGill University, 805 Sherbrooke Street West, Montreal, QC, H3A 2K6, Canada; e-mail: \href{mailto:Armen.Shirikyan@u-cergy.fr}{Armen.Shirikyan@u-cergy.fr}}}
\title{Exponential mixing for a class of dissipative PDEs with bounded degenerate noise}
\date{}
\maketitle

\begin{abstract}
We study a class of discrete-time random dynamical systems with compact phase space. Assuming that the deterministic counterpart of the system in question possesses a dissipation property, its linearisation is approximately controllable, and the driving noise is bounded and has a decomposable structure, we prove that the corresponding family of Markov processes has a unique stationary measure, which is exponentially mixing in the dual-Lipschitz metric. The abstract result is applicable to nonlinear dissipative PDEs perturbed by a bounded random force which affects only a few Fourier modes. We assume that the nonlinear PDE in question is well posed, its nonlinearity  is non-degenerate in the sense of the control theory, and the random force is a regular and bounded function of time which satisfies some decomposability and observability hypotheses. This class of forces includes random Haar series, where the coefficients for high Haar modes decay sufficiently fast. In particular, the result applies to the 2D Navier--Stokes system and the nonlinear complex Ginzburg--Landau equations. The proof of the abstract theorem uses the coupling method, enhanced by the Newton--Kantorovich--Kolmogorov fast convergence.

\smallskip
\noindent
{\bf AMS subject classifications:} 35Q10, 35Q56, 35R60, 37A25, 37L55, 60H15, 76M35

\smallskip
\noindent
{\bf Keywords:} Markov process, stationary measure, mixing, Navier--Stokes system, Ginzburg--Landau equations, Newton--Kantorovich--Kolmogorov fast convergence, approximate controllability, Haar series
\end{abstract}

\newpage
\tableofcontents

\setcounter{section}{-1}

\section{Introduction}
\label{s0} 
The problem of uniqueness of a stationary measure for randomly forced dissipative PDEs attracted a lot of attention in the last twenty years. It is by now well understood that when all determining modes of the unforced PDE are directly affected by the noise, the problem has a unique stationary distribution, which is exponentially stable as $t\to\infty$ in an appropriate metric. We refer the reader to the papers~\cite{FM-1995,KS-cmp2000,EMS-2001,BKL-2002} for the first achievements and to the book~\cite{KS-book} for a detailed description of the results in this setting. The case when the random perturbation does not act directly on the determining modes of the flow is much less understood (see the literature review below), and it is the subject of the present article.   

To describe our results, in this introduction we confine ourselves to discussing two examples when our theory is applicable (other equations will be considered in the main text). Namely, we consider the 2D Navier--Stokes system and the Ginzburg--Landau equations, both perturbed by a Haar coloured noise. Thus, we study the following problems:
\begin{description}
\item[2D Navier--Stokes system:]
	\begin{equation} \label{0.1}
	\p_t u+\langle u,\nabla\rangle u-\nu\Delta u+\nabla p=\eta(t,x), \quad \diver u=0, 
\end{equation}
where $x\in \T^2 = \R^2/2\pi\Z^2$, $u=(u_1,u_2)$ and~$p$ are unknown velocity field and pressure, $\nu>0$ is a parameter, $\Delta$~is the Laplace operator, and $\langle u,\nabla\rangle =u_1\p_1+u_2\p_2$. 
\item[3D Ginzburg--Landau equation:]
\begin{equation} \label{cgl-intro}
	\p_tu-(\nu+i)\Delta u+\gamma u+ic|u|^{2m}u=\eta(t,x),
\end{equation}
where $x\in \T^3 = \R^3/2\pi\Z^3$, $m=1\text{ or }2$, $u=u(t,x)$ is a complex-valued unknown function, and $\nu, \gamma,c>0$ are some parameters.
\end{description}
In both cases, $\eta$ is an external (random) force. Furthermore, when dealing with the Navier--Stokes system, we assume that all the functions have zero mean value in~$x$. 

Let us introduce the trigonometric basis~$\{\varphi_i\}$ in the space of $2\pi$-periodic (vector) functions of~$x$.\footnote{In the case of the Navier--Stokes system, the basis $\{\varphi_i\}$ is composed of divergence-free $\R^2$-valued functions on~$\T^2$ with zero mean value, while in the case of the Ginzburg--Landau equation, they are complex-valued functions on~$\T^3$.} We write the external force~$\eta$ in the form
\begin{equation} \label{0.2}
	\eta(t,x)=\sum_{i=1}^Nb_i\eta^i(t)\varphi_i(x), 
\end{equation}
where~$b_i$ are non-zero numbers, and~$\eta^i$ are independent bounded real-valued random processes that are distributed as a fixed process~$\tilde\eta(t)$ constructed as follows. Let $\{h_0,h_{jl}\}$ be the $L^\infty$-normalised Haar system defined by relations~\eqref{6.8} and~\eqref{6.9}; cf.~Section~22 in~\cite{lamperti1996}. We set
\begin{equation} \label{0.3}
	\tilde\eta(t)=\sum_{k=0}^\infty \xi_k h_0(t-k)
	+\sum_{j=1}^\infty c_j\sum_{l=0}^\infty \xi_{jl}h_{jl}(t),
\end{equation}
where $\{c_j\}$ is a sequence of non-zero numbers going to zero\footnote{\label{FN1}In the case when $c_j=2^{j/2}$ for $j\ge1$ and $\{\xi_k,\xi_{jl}\}$ are independent random variables with centred normal law of unit dispersion, the series~\eqref{0.3} converges to the white noise; see Theorem~1 in~\cite[Section~22]{lamperti1996}. Moreover, by Donsker's invariance principle (see \cite[Section~8]{billingsley1999}), the integral of~\eqref{0.3} converges to the Brownian motion on large time scales.} sufficiently fast, and $\{\xi_k,\xi_{jl}\}$ is a family of independent identically distributed (i.i.d.) real-valued random variables. Processes of the form~\eqref{0.2} are called {\it coloured noises\/} and are widely used in engineering sciences; e.g., see~\cite[Section~6.1]{vanetten2006}. We thus consider the dynamics of Equations~\eqref{0.1} and~\eqref{cgl-intro} subject to the coloured noise~\eqref{0.2}. Because of the time correlations of finite depth, the corresponding trajectories do not form a Markov process. However, their restrictions to integer times do, and our aim is to describe the large-time behaviour of the corresponding discrete-time processes. The following theorem is a consequence of the main result of this paper (see Theorem~\ref{t1.1}) and is valid for both the Navier--Stokes system and the complex Ginzburg--Landau equation. 

\begin{theorem} \label{MT}
	In addition to the above hypotheses, assume that $c_j=C j^{-q}$ for all $j\ge1$ and some $C>0$ and $q>1$, and the law of the random variables $\{\xi_k,\xi_{jl}\}$ has a Lipschitz-continuous density~$\rho$ such that $0\in\supp\rho\subset[-1,1]$. Then there is an integer $N_0\ge1$ such that, for any $N\ge N_0$ and $\nu>0$, the Markov process obtained by restricting the trajectories $u(t)$ to integer times has a unique stationary measure~$\mu_\nu$. Moreover, for any solution~$u(t)$, we have 
	\begin{equation} \label{0.4}
		\DD(u(k))\rightharpoonup\mu_\nu\quad\mbox{as $k\to\infty$, $k\in\N$},
	\end{equation}
	where $\DD(\zeta)$ stands for the law of a random variable~$\zeta$, and the weak convergence in~\eqref{0.4} is understood in the sense of measures on the space of square-integrable (vector) functions on the torus. Finally, convergence~\eqref{0.4} holds exponentially fast in the dual-Lipschitz metric (defined below in {\rm Notation and conventions}).
\end{theorem}

If $u_\nu(t), t\ge 0$, is a solution of~\eqref{0.1} or~\eqref{cgl-intro} such that $\DD(u_\nu(0))=\mu_\nu$, then $u_\nu(t)$ is a statistically periodic process, $\DD(u_\nu(t))= \DD(u_\nu(t+1))$ for each $t\ge0$, and by Theorem~\ref{MT}, for any solution~$u(t)$ the dual-Lipschitz distance between $\DD(u(t))$  and $\DD(u_\nu(t))$ goes to zero as $t\to\infty$.

To prove Theorem~\ref{MT}, we shall establish a general result on exponential mixing for discrete-time Markov processes and show that it applies to the problems~\eqref{0.1} and~\eqref{cgl-intro}. Moreover, our general result is applicable to other dissipative PDEs, such as reaction-diffusion equations with polynomial nonlinearities, viscous primitive equations of large scale ocean dynamics, and multidimensional Burgers system in non-potential\footnote{The one-dimensional Burgers equation and multi-dimensional Burgers system in the potential setting can be treated by softer tools; see~\cite{sinai-1991,boritchev-2016,shirikyan-cup2017}.} setting (these examples are not treated in this paper for reasons of space). As was mentioned above, there are only a few works dealing with highly degenerate noise not acting directly on the determining modes of the unperturbed dynamics. Namely, the existence of densities for finite-dimensional projections of stationary measures for solutions of the Navier--Stokes system  was studied in~\cite{MP-2006,AKSS-aihp2007}. Hairer and Mattingly~\cite{HM-2006,HM-2011} investigated the Navier--Stokes system perturbed by a finite-dimensional white noise and established the uniqueness of stationary measure and its exponential stability in the dual-Lipschitz metric (see also~\cite{HMS-2011}). F\"oldes at~al.~\cite{FGRT-2015} proved a similar result for the Boussinesq system, assuming that the random noise acts only on the equation for temperature. Finally, the case in which the random perturbation is localised in the physical space and time was studied in~\cite{shirikyan-asens2015} (see also~\cite{shirikyan-2018} for the boundary-driven Navier--Stokes system). 

\smallskip
This work is a continuation of the research, started in  \cite{AKSS-aihp2007} in order to prove the mixing for systems with degenerate noise by the basic Markov techniques and methods of control theory, instead of the Malliavin calculus used in~\cite{HM-2006,HM-2011, FGRT-2015}. Namely, in~\cite{AKSS-aihp2007}  we evoked those techniques to establish that the finite-dimensional projections of distributions of solutions for  systems with various  degenerate noises have densities against the Lebesque measure. The results of~\cite{AKSS-aihp2007} do not provide any information about the regularity of densities, in contrast to the infinite-dimensional version of the Malliavin calculus developed in~\cite{MP-2006}. Its further development was a key ingredient of the proof of ergodicity in~\cite{HM-2006, HM-2011}. In particular, the question of almost sure invertibility of Malliavin's matrix plays a crucial role, and Tikhonov's regularisation is used to construct its approximate inverse. It is well known that, in the control theory, Malliavin's matrix corresponds to the Gramian, responsible for the controllability of the linearised system. This observation led us to the idea to use the linearised controllability in Doeblin's scheme to prove the mixing. Carrying out this idea in the form of a quadratic convergence procedure, we arrived at a ``soft" proof of mixing in systems with various random forces, in difference with the ``hard" proof, based on the exact formulas from the Malliavin calculus. The latter requires the forces to be Gaussian. In this paper, we work with systems stirred by bounded random forces (see below), but we are certain that our approach also applies to systems  perturbed  by Gaussian forces.

\subsubsection*{Bounded random forces versus white in time forces} 
In our work we study nonlinear PDEs perturbed by random forces that are, as a function of time, bounded processes of the type of {\it random Haar series\/} (see~\eqref{0.3}), while it is more traditional in mathematical physics to use the forces that are random processes {\it white in time\/}. What are the advantages and disadvantages of the former class of forces compared to the latter? The first disadvantage is the tradition: one hundred years ago, in Langevin's era, the white in time forces were successfully used to model  systems from statistical physics, and since then they  were exploited  in other problems, usually without serious discussion of their adequacy. Secondly, white-forced equations have useful algebraical features coming from Ito's formula. On the  other hand, bounded random forces, exactly due to their boundedness, serve better to build models for some specific physical problems, where unbounded forces make no sense (e.g., they are being used in modern meteorology). Secondly, they have a number of serious analytical advantages. Namely, the corresponding stochastic equations are always well posed if so are the deterministic equations, while the equations, perturbed by white noise, may be not (or their solutions may satisfy only weak a priori estimates, cf. Example~\ref{e-primitive} below). What  is more important is that for systems with the random forces which we advocate, the mixing property can be established for significantly broader class of PDEs. Indeed, if the nonlinearity of the unperturbed deterministic equation is Hamiltonian and polynomial, and the random perturbation is white in time, then the existing techniques apply to establish the mixing only if the nonlinearity is at most cubic.\footnote{See the paper~\cite{KN-2013}, where the mixing is proved  for the white-forced cubic CGL equations, and the difficulty coming from  nonlinearities of higher degree is explained.} At the same time, the main theorem of our work can be used to prove the mixing property for equations with nonlinearity of any degree, cf.\ Eq.~\eqref{cgl-intro}. For the stochastic 2D Navier--Stokes equation~\eqref{0.1} (where the nonlinearity is quadratic) with degenerate white in time force~$\eta$,  the exponential mixing is proved in the  papers~\cite{HM-2006, HM-2011}  based on an infinite-dimensional version of the Malliavin calculus developed in~\cite{MP-2006}. If the random force~$\eta$ is bounded and degenerate, then the proof of the exponential mixing, presented in our work, is significantly shorter and, we believe, conceptually clearer. A subclass of the random forces which we consider---the random Haar series~\eqref{0.3}---has a number of similarities with the white forces (and the latter may be obtained as a limiting case of the former, see footnote~\ref{FN1}). In particular, the forces~\eqref{0.3} have independent components with arbitrarily short time scales, which simplifies the verification for them of various non-degeneracy properties (e.g., see Section~\ref{s6.2}, where we show that these forces are Lipschitz-observable).

\medskip
The paper is organised as follows. In Section~\ref{s1}, we formulate our main result on the uniqueness and exponential mixing of a stationary measure for discrete-time Markov processes possessing some controllability properties and describe briefly its applications. To simplify the reading of the paper, in Section~\ref{s1.3} we describe the general philosophy of the proof  and discuss some analogies between the coupling scheme for PDEs with noise, used in our proof, and the Newton--Kantorovich--Kolomogorov fast convergence. Section~\ref{s3} is devoted to the proof of the main result, and Section~\ref{s-auxiliary-results} gathers some auxiliary assertions.  In Section~\ref{s4}, we apply our result to the 2D Navier--Stokes system and the complex Ginzburg--Landau equations perturbed by a random noise. Sections~\ref{s6} and~\ref{s-satset} describe some classes of random noises that are allowed in our approach. Finally, the Appendix gathers some auxiliary results used in the main text.  

\subsubsection*{Acknowledgements} 
We thank V.\,I.~Bogachev for discussion on the measurable version of gluing lemma, which was established in the papers~\cite{BM-doklady2019,BM-2019} on our request. We also thank the {\it Institut Henri Poincar\'e\/} in Paris for hosting our working group {\it Control Theory and Stochastic Analysis\/}, at which this article was initiated. This research was supported by the {\it Agence Nationale de la Recherche\/} through  the grants ANR-10-BLAN~0102 and ANR-17-CE40-0006-02. SK  thanks the {\it Russian Science Foundation\/} for support through the grant  18-11-00032. VN and AS were supported by the CNRS PICS {\it Fluctuation theorems in stochastic systems\/}. The research of AS was carried out within the MME-DII Center of Excellence (ANR-11-LABX-0023-01) and supported by {\it Initiative d'excellence Paris-Seine\/}.

\subsubsection*{Notation and conventions}
Let~$X$ be a Polish (i.e., complete separable metric) space with a distance $d_X(u,v)$ and the Borel $\sigma$-algebra $\BB(X)$. We denote by~$B_X(a,R)$ the closed ball of radius $R > 0$ centred at $a\in X$ and write $\dot B_X(a,R)$ for the corresponding open ball. If $X$ is a Banach space and $a=0$, we write $B_X(R)$ instead of $B_X(0,R)$. We use the following notation in which all abstract Banach and metric spaces are assumed to be separable:

\smallskip
\noindent
$\LL(E,F)$ is the space of bounded linear operators between two Banach spaces~$E$ and~$F$. It is endowed with the operator norm. 

\smallskip
\noindent
$L^p(J,E)$ is the space of Borel-measurable functions~$f$ on an interval $J\subset\R$ with range in a separable Banach space~$E$ such that 
$$
\|f\|_{L^p(J,E)}=\biggl(\int_J \|f(t)\|_E^p\dd t\biggr)^{1/p}<\infty,
$$
with an obvious modification for $p=\infty$. 

\smallskip
\noindent
$C_b(X)$ is the space of bounded continuous functions $f:X\to\R$ endowed with the norm $\|f\|_\infty=\sup_X|f|$. 

\smallskip
\noindent
$L_b(X)$ is the space of functions $f\in C_b(X)$ such that
$$
\|f\|_L:=\|f\|_\infty+\sup_{0<d_X(u,v)\le 1}\frac{|f(u)-f(v)|}{d_X(u,v)}<\infty.
$$

\smallskip
\noindent
$\PP(X)$ denotes the set of  Borel probability measures on $X$. For any $\mu\in \PP(X)$ and $\mu$-integrable function $f:X\to\R$, we set
$$
\lag f,\mu\rag=\int_Xf(u)\,\mu(\dd u). 
$$
The {\it total variation metric\/} on $\PP(X)$ is defined by
\begin{equation} \label{TVN}
	\|\mu_1-\mu_2\|_{\mathrm{var}}:=\sup_{\Gamma\in\BB(X)}|\mu_1(\Gamma)-\mu_2(\Gamma)|
=\frac12\sup_{\|f\|_\infty\le1}
\left|\lag f,\mu_1\rag-\lag f,\mu_2\rag\right|.
\end{equation}
We shall also use the  {\it dual-Lipschitz\/}  metric
$$
\|\mu_1-\mu_2\|_L^*
:=\sup_{\|f\|_L\le1}\left|\lag f,\mu_1\rag-\lag f,\mu_2\rag\right|.
$$
Note that $\|\mu_1-\mu_2\|_L^*\le 2\|\mu_1-\mu_2\|_{\mathrm{var}}$. 

\smallskip
\noindent
We denote by~$C$, $C_1$, etc.\ unessential positive constants. 

\smallskip
\noindent
If $B_1$ and~$B_2$ are {\it real\/} Banach spaces and $\OO\subset B_1$ is an open domain, then analyticity of a map $F:\OO\to B_2$ is understood in the sense of Fr\'echet. In addition, for an analytic map~$F$, we always assume that 
\begin{equation} \label{analytic}
\mbox{\it the norms of all the derivatives $D^kF$ are bounded on bounded subsets of~$\OO$}. 	
\end{equation}
Moreover, if~$F$ depends on a parameter~$u$ varying in a compact metric space~$X$, then we assume that {\it all the derivatives~$D_\eta^kF(u,\eta)$ are bounded on bounded subsets, uniformly in~$u\in X$, and are continuous functions of~$(u,\eta)$}.

\section{Mixing for Markovian random dynamical systems}
\label{s1} 
\subsection{Setting of the problem}
\label{s1.1}
Let~$H$ and~$E$ be separable Hilbert spaces and let $S:H\times E\to H$ be a continuous mapping. We consider the random dynamical system (RDS) given by 
\begin{equation} \label{1.1}
u_k=S(u_{k-1},\eta_k), \quad k\ge1,
\end{equation}
where $\{\eta_k\}$ is a sequence of i.i.d.\ random variables in~$E$. Let us denote by~$\ell$ the law of~$\eta_k$ and assume that it has a compact support~$\KK\subset E$. Suppose there is a compact set $X\subset H$ such that $S(X\times\KK)\subset X$, so that one can consider the restriction of the RDS~\eqref{1.1} to~$X$. The hypotheses imposed on~$\eta_k$ imply that the trajectories of~\eqref{1.1} form a discrete-time Markov process in~$X$; we shall denote it by~$(u_k,\IP_u)$, where~$\IP_u$ is the probability measure corresponding to the trajectories issued from~$u$ (e.g., see Section~2.5.B in~\cite{KS1991} or  Section~1.3.1 in~\cite{KS-book}). We write  $P_k(u,\Gamma)$ for the transition function and denote by~$\PPPP_k:C_b(X)\to C_b(X)$ and $\PPPP_k^*:\PP(X)\to\PP(X)$ the corresponding Markov operators. Recall that  a measure $\mu \in \PP(H)$ is  said to be {\it stationary\/} for $(u_k,\IP_u)$ if~$\PPPP^*_1\mu=\mu$. Since~$X$ is compact, the Bogolyubov--Krylov argument implies that there is at least one stationary measure. Our goal is to study its uniqueness and long-time stability under the dynamics. In what follows, we assume that the four hypotheses below are satisfied.

\smallskip
\begin{description}
\item[\hypertarget{H1}{(H$_1$)} Regularity.] 
{\sl There is a Banach space~$V$ compactly embedded into~$H$ such that the image of~$S$ is contained in~$V$, the mapping $S:H\times E\to V$ is twice continuously differentiable, and its derivatives up to the second order are bounded on bounded subsets. Moreover, for any $u\in H$, the mapping $\eta\mapsto S(u,\eta)$ is analytic from~$E$ to~$H$, and the derivatives $(D_\eta^jS)(u,\eta)$ are continuous functions of~$(u,\eta)$ that are bounded on bounded subsets of~$H\times E$; cf.~\eqref{analytic}.}

\item[\hypertarget{H2}{(H$_2$)} Dissipativity.] 
{\sl There is a number $a\in(0,1)$ and vectors $\hat \eta\in\KK$ and $\hat u\in X$  such that 
\begin{equation} \label{1.2}
\|S(u,\hat \eta)-\hat u\|\le a\|u-\hat u\|\quad\mbox{for any $u\in X$}. 
\end{equation}}
\end{description}

To formulate the third hypothesis, for any point $u\in X$, we denote by~$\KK^u$ the set of those $\eta\in\KK$ for which the image of $(D_\eta S)(u,\eta):E\to H$ is dense in~$H$. It is easy to see that~$\KK^u$ is a section of a Borel subset in the product space~$X\times E$ and, hence, is a Borel subset of~$E$.\footnote{To see this, consider a measurable space $(Y,\YY)$ and a measurable map $A:Y\to\LL(E,H)$ and denote by~$\GG_A$  the set of points $y\in Y$ for which the image of~$A(y)$ is dense in~$H$. Choosing countable dense subsets~$\{f_i\}\subset E$ and~$\{h_j\}\subset H$, it is easy to see that $\GG_A=\bigl\{y\in Y:\inf_{i\ge1}\|A(y)f_i-h_j\|_H=0\mbox{ for any }j\ge1\bigr\}$. In the case under study, we have $Y=H\times E$, $A=(D_\eta S)(u,\eta)$, and $\KK^u=\KK\cap \GG^u$, where $\GG^u=\{\eta\in E:(u,\eta)\in \GG_A\}$.} 

\begin{description}
\item[\hypertarget{H3}{(H$_3$)} Approximate controllability of linearisation.] 
{\sl For any $u\in X$, we have $\ell(\KK^u)=1$.}

\item[\hypertarget{H4}{(H$_4$)} Decomposability of the noise.]
{\sl There is an orthonormal basis~$\{e_j\}$ in~$E$ such that 
\begin{equation} \label{1.3}
\eta_k=\sum_{j=1}^\infty b_j\xi_{jk}e_j,
\end{equation}
where $\xi_{jk}$ are independent random variables such that $|\xi_{jk}|\le1$ a.s. and~$b_j$ are non-zero numbers  satisfying 
\begin{equation} \label{1.4}
\sum_{j=1}^\infty b_j^2<\infty.
\end{equation}
Moreover, there are Lipschitz-continuous functions $\rho_j:\R\to\R$ such that 
\begin{equation} \label{1.5}
\DD(\xi_{jk})=\rho_j(r)\dd r\quad\mbox{for all $j\ge1$}.
\end{equation}}
\end{description}

Before turning to formulating our main result on the discrete-time random dynamical system~\eqref{1.1}  satisfying \hyperlink{H1}{(H$_1$)}--\hyperlink{H4}{(H$_4$)}, we describe how such systems arise in the analysis of nonlinear PDEs of parabolic type perturbed by a random force and explain why, in this situation,  Hypotheses~\hyperlink{H1}{(H$_1$)}, \hyperlink{H2}{(H$_2$)} and~\hyperlink{H4}{(H$_4$)} hold trivially, whereas~\hyperlink{H3}{(H$_3$)} imposes some restrictions on both the {\it nonlinearity\/} and the {\it random force\/}. Namely, let us consider the following equation in a functional space~$H$ which, as before, is assumed to be a separable Hilbert space: 
\begin{equation} \label{npde}
	\dot u=-Lu+F(u)+\eta(t), 
\end{equation}
where $L=L^*>0$ is an unbounded linear operator in~$H$, $F$ is an analytic nonlinearity, and~$\eta$ is a random force. In the examples we deal with, $L$ is a positive degree of~$-\Delta$, and we do not assume that the linear part of~$F$ at the origin (i.e., $(DF)(0)$) vanishes. Let~$\HH$ be a separable Hilbert space of finite or infinite dimension, continuously embedded into~$H$, and let $\{\varphi_i\}_{i\in\II}$ be an orthonormal basis in~$\HH$. We assume that the random force~$\eta$ has the form
$$
\eta(t)=\sum_{k=1}^\infty \I_{[k-1,k)}(t)\eta_k(t-k+1),
$$
where $\I_{[k-1,k)}$ is the indicator function of the interval $[k-1,k)$ and~$\{\eta_k\}$ is a sequence of i.i.d.\ random variables in $E:=L^2(J,\HH)$ with $J=[0,1]$ such that 
\begin{equation} \label{comp}
	\mbox{$\KK=\supp\DD(\eta_k)$ is a compact subset of~$E$};
\end{equation}
this implies, in particular, that the random variables~$\eta_k$ are uniformly bounded in~$E$. 

Denote by $S:H\times E\to H$ the mapping which takes a pair $(v,\eta|_{[0,1]})$ to~$u(1)$, where $u(t)$ is a solution of~\eqref{npde} satisfying $u(0)=v$. With this notation, a solution of~\eqref{npde} evaluated at integer times $t\ge0$ is a trajectory of~\eqref{1.1}. The existence of a compact set $X\subset H$ such that the mapping $v\mapsto S(v,\eta_1)$ preserves~$X$ a.s.\ often follows from the parabolic regularity and the boundedness of~$\eta$. 

Assumption~\hyperlink{H1}{(H$_1$)} is another usual consequence of the parabolic regularity, while~\hyperlink{H2}{(H$_2$)} holds, for example, if the origin of~$H$ is an exponentially stable equilibrium for Eq.~\eqref{npde} with $\eta\equiv0$, and  $0\in\KK$. (This condition can be relaxed, requiring only the existence of a point that is in the set of accessibility for an arbitrary initial state; see~\cite{KNS-2019} for details.) Hypothesis~\hyperlink{H4}{(H$_4$)} is not restrictive either. In particular, it holds if $\{\eta_k(\cdot)\}$ are independent realisations of a process~$\eta(\cdot)$ defined by the relations 
$$
	\eta(t)=\sum_{i\in\II}\eta^i(t)\varphi_i, \quad 
	\eta^i (t)=\sum_{l=1}^\infty b_{il}\xi_l^i\psi_l^i(t),
$$
where~$\{\psi_l^i\}_{l\ge1}$ are orthonormal bases in~$L^2(J)$ for each~$i$, $\{\xi_l^i\}$ are independent random variables with Lipschitz densities against the Lebesgue measure such that  $|\xi_l^i|\le 1$ almost surely, and $\{b_{il}\}$ decay to zero sufficiently fast as $i,l\to\infty$ (the latter property implies~\eqref{comp}); see Section~\ref{s6.2} for specific examples. 

Finally, Hypothesis~\hyperlink{H3}{(H$_3$)}~means certain non-degeneracy of the function~$F$, and it holds if the latter satisfies a H\"ormander-type condition and if~$\eta$ is observable in a suitable sense; see Sections~\ref{s4} and~\ref{s6} for more details. As we show in the proof, these properties, together with some analyticity argument, imply that, with high probability, the trajectories of~\eqref{npde} (and those of~\eqref{1.1}) can be locally stabilised by a finite-dimensional modification of the driving force. 

In conclusion of this subsection, let us also mention that the controllability of the linearised operator (or, equivalently, the existence of its right inverse) is well known to be important when studying mixing properties for random dynamical systems. In particular, it arises in the Malliavin calculus and plays an important role when proving the absolute continuity of the laws of solutions of SDE with respect to the Lebesgue measure; see Chapter~2 in~\cite{nualart2006}. The approximate controllability of the linearised equation was used by Hairer and Mattingly~\cite{HM-2006,HM-2011} in their proof of exponential mixing of the 2D Navier--Stokes system perturbed by a degenerate noise, white in time and finite-dimensional in~$x$.

\subsection{Main result and examples}
\label{s1.2}

Let us denote by~$\|\cdot\|_L^*$ the dual-Lipschitz metric in the space of probability measures on~$X$ (cf.\ Notation and conventions). The following theorem is the main result of this paper. 

\begin{theorem} \label{t1.1}
Suppose that Hypotheses {\rm\hyperlink{H1}{(H$_1$)}--\hyperlink{H4}{(H$_4$)}} are satisfied. Then the Markov process $(u_k,\IP_u)$ has a unique stationary measure $\mu\in \PP(X)$, and there are positive numbers~$C$ and~$\gamma$ such that
\begin{equation} \label{1.6}
\|\PPPP_k^*\lambda-\mu\|_L^*\le Ce^{-\gamma k}\quad\mbox{for all $k\ge0$ and $\lambda\in\PP(X)$}. 
\end{equation}
\end{theorem}
The proof of this  result is based on an application of the Kantorovich functional method, described in Section~3.1.1 of~\cite{KS-book} and repeated here as Theorem~\ref{expodecay}.\footnote{In a different form, this method as a tool to prove the exponential mixing for nonlinear PDEs with stochastic forcing was suggested in~\cite{KPS-cmp2002}. In the final form (used in our work) the method was presented in~\cite{kuksin-ams2002}.} By that result, to prove the theorem, it suffices to check a contraction property in the space of probability measures on~$X$. This  requires subtle analysis based on ideas  from the optimal control, measure theory, and theory of analytic functions. A proof of Theorem~\ref{t1.1} is given in Section~\ref{s3}. Here we discuss briefly some examples to which our main result is applicable; the details will be given in Section~\ref{s4}.

\begin{example}[Navier--Stokes system]\label{example-NS}
Let $\T_a^2=\R^2/(2\pi a_1)\Z\oplus (2\pi a_2)\Z$ be a rectangular torus, where~$a=(a_1,a_2)$ is a vector with positive coordinates, and let
\begin{equation} \label{spaceH}
H=\biggl\{u\in L^2(\T_a^2,\R^2): \diver u=0
\mbox{ in $\T_a^2$}, \int_{\T_a^2} u(x)\,\dd x=0\biggr\}.
\end{equation}
We consider the Navier--Stokes system~\eqref{0.1} in~$\T_a^2$ perturbed by a random process. Applying the Leray projection~$\Pi:L^2(\T_a^2,\R^2)\to H$ to the equations, we reduce the system to the following nonlocal PDE:
\begin{equation} \label{1.7}
\p_tu+\nu Lu+B(u)=\eta(t,x).
\end{equation}
Here $\nu>0$ is the viscosity coefficient, $L=-\Pi\Delta$ is the Stokes operator, $B(u)=\Pi(\langle u,\nabla\rangle u)$, and~$\eta$ is a random process which is assumed to be of the form
\begin{equation} \label{1.08}
\eta(t,x)=\sum_{k=1}^\infty \I_{[k-1,k)}(t)\eta_k(t-k+1,x),
\end{equation}
where $\I_{[k-1,k)}$ is the indicator function of the interval $[k-1,k)$, and~$\{\eta_k\}$ is a sequence of i.i.d.\ random variables in $L^2(J,H)$ with $J=[0,1]$ whose law is decomposable in the sense of Hypothesis~\hyperlink{H4}{(H$_4$)}. We shall prove that if, in addition, the law of~$\eta_k$ is {\it observable\/} (see Section~\ref{s4.1}), then the Markov process obtained by restricting solutions of~\eqref{1.7} to integer times has a unique stationary distribution, which is exponentially mixing in the dual-Lipschitz metric. Examples of observable processes~$\eta_k$ are given in Section~\ref{s6}. 
\end{example}

\begin{example}[Ginzburg--Landau equation]\label{example-GL}
Let us fix a vector $a=(a_1,a_2,a_3)$ with positive coordinates and denote by~$\Z_a^3$ the integer lattice generated by the vectors~$a_k\iota_k$, where $\{\iota_1,\iota_2,\iota_3\}$ is the standard basis in~$\R^3$. We set $\T_a^3=\R^3/2\pi\Z_a^3$ and consider Eq.~\eqref{cgl-intro}, in which $u=u(t,x)$ is an unknown complex-valued function, $\nu$, $\gamma$, and~$c$ are positive parameters, and $m\in\{1,2\}$. Equation~\eqref{cgl-intro} is well posed in the Sobolev space $H^1(\T_a^3,\C)$. Assuming again that the random force~$\eta$ has the form~\eqref{1.08} and  satisfies the decomposability and observability hypotheses,  we shall prove uniqueness and exponential mixing of stationary measure for~\eqref{cgl-intro}. 
\end{example}

\begin{example}[3D stochastic primitive equations]\label{e-primitive}
Let us consider the following equations for a vector function $(u,p)=(u_1,u_2,u_3,p)$ in the horizontally periodic box $D=\T^2\times (0,1)$: 
\begin{equation} \label{primitive-equations}
\p_t v+\langle u, \nabla \rangle v-\nu \Delta v+\nabla p =\eta(t,x), \quad \diver u=0, 
\end{equation}
where $v=(u_1,u_2)$, the pressure $p$ is assumed to be independent of~$x_3$, and the buoyancy is taken to be zero for simplicity. This is a popular model in the theory of climate, and its mathematical  treatment attracted a lot of attention in the last two decades (e.g., see the paper~\cite{CT-2007} and the references therein). In the stochastic setting with a spatially regular white noise~$\eta$, the well-posedness of~\eqref{primitive-equations} and the existence of a stationary measure were established in~\cite{DGTZ-2012,GKVZ-2014}. On the other hand, the uniqueness of a stationary solution in that setting seems to be out of reach because of rather weak tail estimates for strong solutions (e.g., their moments in Sobolev spaces are not know to be bounded); cf.\ Theorems~1.7 and~1.9 in~\cite{GKVZ-2014}. In the case of a bounded observable noise, the problem of mixing can be treated with the help of Theorem~\ref{t1.1} and will be carried out in a subsequent  publication. Thus, even in the case of quadratic nonlinearities, the class of noises we deal with provides a new framework for proving the mixing behaviour of dynamics. 
\end{example}

\subsection{Coupling in infinite dimension and Newton--Kantoro\-vich--Kolmogorov fast convergence}
\label{s1.3}

In this subsection, we describe the general scheme of the proof of Theorem~\ref{t1.1}, based on a suitable coupling construction, and discuss its relation to the technique of fast convergence due to Newton--Kantorovich--Kolmogorov. To prove convergence~\eqref{1.6}, it suffices to verify that any two trajectories~$\{u_k\}$ and~$\{u'_k\}$ with random initial data~$u$ and~$u'$ converge in distribution  when $k\to\infty$. Often this property is proved  with the help of the coupling argument, originated in 1930's in the work of Doeblin~\cite{doeblin-1940} (see also~\cite{doeblin-2000}), and we recall now the main idea. 

\smallskip
Let us consider the following dynamics in the space $X\times X$:
\begin{align*}
(u_0,v_0)&=(u,u'),\\ 
(u_k,v_k)&= (S(u_{k-1},\eta_k),S(v_{k-1},\eta_k')), \quad k\ge1,	
\end{align*}
where $\{(\eta_k,\eta_k')\}$ is a sequence of independent random variables such that 
\begin{equation}
\label{d3}
\DD (\eta_k') = \DD(\eta_k)=\ell\quad \mbox{for all $k\ge1$}. 	
\end{equation}
It is clear that $\{u_k\}$ is a trajectory of~\eqref{1.1} starting from~$u$, while $\{v_k\}$ coincides with that issued from~$u'$ in the sense of law: $\DD(v_k) = \PPPP_k^* \DD(u')=\DD(u'_k)$. Our goal is to choose a sequence~$\{\eta_k'\}$ satisfying~\eqref{d3} such  that, with probability~$1$, $\{v_k\}$ is asymptotically close to $\{u_k\}$ as $k\to\infty$. To construct~$\{\eta_k'\}_{k\ge1}$, we fix a small parameter $d>0$ and distinguish between the following two cases:
\begin{itemize}
	\item[\bf\hypertarget{(a)}{(a)}] 
	Setting $\delta:= \| u_{k-1} - v_{k-1} \|$, we choose for $\eta'_k$ an independent copy of $\eta_k$ as long as $\delta >  d$. Due to the dissipativity, at some random time~$\tau$, we shall have $\|v_\tau-u_\tau\|\le d$. The Markov time~$\tau$ is no bigger than the first instance when both trajectories are in the $\frac{ d}{2}$-neighbourhood of~$\hat u$, and the latter can be controlled due to Hypothesis~\hyperlink{H2}{(H$_2$)}. 
	 
	\item[\bf\hypertarget{(b)}{(b)}] 
	The case $\delta= \| u_{k-1} - v_{k-1} \| \le d$ contains the main difficulty. To simplify notation, assume that $k=1$; so $u_0=u$ and $v_0=u'$. We seek~$\eta'_1$ in the form $\eta_1'=\varPsi^{u,u'}(\eta_1)$, where the transformation $\varPsi^{u,u'}$ satisfying 
\begin{equation}\label{d4}
\varPsi_*^{u,u'}(\ell) = \ell
\end{equation}
has to be constructed. Relation~\eqref{d4} implies~\eqref{d3} (with $k=1$), and the goal is to find $\varPsi^{u,u'}$ such that the inequality $\|u_1-v_1\| \ll \delta$ holds with high probability. \end{itemize}
In the best case, we may have $v_1=u_1$ almost surely, that is
\begin{equation}\label{d5}
S(u,\eta_1)- S(u',\eta_1')=0, \quad \text{a.s.}
\end{equation}
This is not likely to be possible because, for deterministic initial states~$u$ and~$u'$, it would imply that $\DD(u_1) =\DD(v_1)= \DD(u'_1)$, which is not necessarily the case. 

\medskip
The situation is reminiscent of that treated by Kolmogorov's celebrated  theorem on nearly-integrable Hamiltonians 
\begin{equation}\label{hamilt}
H_\delta (p,q) =h_1(p) + \delta f_1(p,q),\qquad (p,q) \in B^n \times \T^n,
\end{equation}
where $B^n\subset\R^n$ is a ball and $\delta>0$ is a small parameter. Since the (Hamiltonian) dynamics for a Hamiltonian depending only on~$p$ is integrable, a naive idea to study the one corresponding to~$H_\delta$ is to find a symplectic transformation~$S_\delta$ such that
\begin{equation}\label{int}
H_\delta\circ S_\delta(p,q)=h_\delta(p).
\end{equation}
However, it was shown by Poincar\'e that, in general, such a transformation~$S_\delta$ does not exist. Kolmogorov's well-known idea to bypass this obstruction is to achieve~\eqref{int} only up to a higher order term:
\begin{equation}\label{int1}
H_\delta\circ S_1(p,q) =h_2(p)  + \delta^2 f_2(p,q).
\end{equation}
This relation holds if the symplectic transformation $S_1$ is a time-1 flow of a Hamiltonian $\delta  g_1(p,q)$, with some function~$g_1$ satisfying the linear homological equation 
\begin{equation}\label{he}
\{ h_1(p), g_1(p,g)\} =f_1(p,q) -\langle f_1\rangle (p),
\end{equation}
where $\langle\,\cdot\,\rangle$ stands for the averaging in $q\in \T^n$.  If $h_1$ meets a mild non-degeneracy condition, then~\eqref{he} can be solved with a disparity of order~$\delta$, for all~$q$ and for~$p$ outside a small-measure set~$B_1\subset B^n$. The corresponding transformation~$S_1$ reduces~$H_\delta$ to a Hamiltonian~\eqref{int1} which is much closer to being integrable than~$H_\delta$. Iterating this argument for $j=1,2,\dots$, Kolmogorov constructed symplectomorphisms $\{S_1\circ\dots\circ S_j\}_{j\ge1}$ transforming~$H_\delta$ to a  Hamiltonian~$H_{\delta,j}$ that is $\delta^{2^j}$-close to being integrable, for all~$q$ and for~$p$ outside a set $B_1\cup\dots \cup B_j\subset B$. When $j\to\infty$, the transformations $S_1\circ\dots\circ S_j$ converge, super-exponentially fast, to a limiting transformation~$S_\delta$ which satisfies~\eqref{int} for all~$p$ outside the  set~$\cup_jB_j$, which turns out to be  of a small  measure. This implies the assertions made by Kolmogorov in  his seminal paper~\cite{kolmogorov-1954}. (That paper contained only a scheme of the argument, and a complete proof appeared only 10 years later in the works by Arnold and Moser, resulting in creation of the KAM theory; see~\cite{arnold-1963,moser-1966-1}.)

\medskip
Going back to our problem, we write $u'=u+\delta v$, where $\|v\|=1$, and seek a transformation of the form $\varPsi^{u,u'}=\Id{}+\delta \varPhi$. Similar to Kolmogorov's approach, let us rewrite relation~\eqref{d5} as
$$
\delta \bigl(D_\eta S(u,\eta_1)\varPhi(\eta_1) - S'(u,u',\eta_1)\bigr) + O(\delta^2)=0, 
$$
where $S'=\delta^{-1}(S(u+\delta v,\eta_1)-S(u,\eta_1))$ is of order~1. 
Neglecting the term~$O(\delta^2)$, consider the equation 
 \begin{equation}\label{d6}
D_\eta S(u,\eta_1)\varPhi(\eta_1)=S'(u,u',\eta_1).
\end{equation}
This is the homological equation of our proof, analogous to Eq.~\eqref{he}  from the KAM theory. Equation~\eqref{d6}, as well as \eqref{he}, cannot be solved exactly or approximatively for a.a.~$\eta_1$. However, in view of Hypothesis~\hyperlink{H3}{(H$_3$)}, it can be solved approximately for $\eta_1\notin \NN^u$, where~$\NN^u$ is a suitable ``bad'' set of small $\ell$-measure in the support~$\KK$ of~$\ell$. It is proved in Section~\ref{s2.3} that, for any $\e>0$, an approximate solution $\varPhi (\eta_1)$ solving \eqref{d6}  up to a term of order~$\e$ can be found in the form 
$$
\varPhi(\eta_1)= R_\e(u,\eta_1)S', 
$$
where $R_\e(u, \eta_1)$ is a finite-dimensional linear operator satisfying the inequality
\begin{equation}\label{d7}
\| R_\e(u,\eta_1) \| \le C(\e, \ell(\NN^u)) <\infty
\end{equation}
with some function $C(\e,r)$ going to~$+\infty$ as $\e\to0$ or $r\to0$. Setting 
$$
\varPsi^{u,u'}(\eta_1) 
= \eta_1 +\delta\varPhi(\eta_1)
= \eta_1 +\delta R_\e(u,\eta_1)S'(u,u',\eta_1),
$$ 
we make the left-hand side of~\eqref{d5} of order 
\begin{equation}\label{d8}
\delta_1=O(\delta\e)+\delta^2 C(\e,\ell(\NN^u))^2
\quad \text{for}\quad \eta_1\notin \NN^u.
\end{equation}
In the KAM theory, the usual strategy is to choose $\e=\delta^{\gamma_1}$ with some $\gamma_1>0$. Then, if we knew that $\delta^2 C(\delta^{\gamma_1},\ell(\NN^u))^2\le\delta^{1+\gamma_2}$ with some $\gamma_2>0$, this would lead to a super-exponential convergence $\| v_k-u_k\|\to0$, typical for the theory. However, such a choice is now impossible, since the constant $C(\e,\ell(\NN^u))$ is practically out of control. Instead, we derive from~\eqref{d8} that 
$$
\delta_1\le C_1\delta\bigl(\e+ dC(\e,\ell(\NN^u))^2\bigr)
$$
and choose $\e=(4C_1)^{-1}$ and $ d=\bigl(4C_1 C(\e,\ell(\NN^u))^2\bigr)^{-1}$. This implies that $\delta_1\le\frac{1}{2}\delta$ and leads to an exponential convergence $\| v_k-u_k\|\to0$, which is sufficient for our purposes. 

\smallskip
Construction of vectors $\{\eta_k'\}_{k\ge1}$ encounters two\footnote{The first  difficulty does not have an analogue in the KAM theory, whereas the second is usually present and manifests itself in the fact that the homological equation~\eqref{he} cannot be solved for all actions $p\in B$.} difficulties, which we describe for the first step:
\begin{itemize}
\item[\bf\hypertarget{(1)}{(1)}]
The transformation $\eta_1\mapsto \varPsi^{u,u'}(\eta_1)$ does not preserve the measure~$\ell$, so~\eqref{d4} does not necessarily hold.
\item[\bf\hypertarget{(2)}{(2)}]	
The transformation $\varPsi^{u,u'}(\eta_1)$ is not defined for $\eta_1\in  \NN^u$.
\end{itemize}

The first difficulty is overcome due to the observation that, in our situation, the distance between the measures~$\ell$ and $\varPsi_*^{u,u'}\ell$ is of order $\bigl(\delta\|R_\e(u,\eta_1)\|\bigr)^\varkappa$ with some $\varkappa>0$, which is small by~\eqref{d7}. Thus, even though the laws of~$v_1$ and~$u_1'$ are not the same, the two are close, which allows one  to bound the distance  between the laws of~$u_1$ and~$u'_1$ by the triangle inequality, provided that~$v_1$ is close to~$u_1$. In reality, the corresponding argument is a bit more complicated and evokes the gluing lemma; see  Section~\ref{s5.4}.

To handle the second difficulty, we extend the definition of~$\varPsi^{u,u'}$ to~$\NN^u$ as follows:
\begin{itemize}
	\item[\bf\hypertarget{(c)}{(c)}] If $\eta_1\in \NN^u$, then $\varPsi^{u,u'}(\eta_1)=\eta_1$. 
\end{itemize}
Since the mapping~$S$ is Lipschitz on the compact set~$X\times\KK$, for $\eta_1\in\NN^u$ we have $\|u_1-v_1\| \le C\delta$. We then iterate Steps~\hyperlink{(a)}{(a)}--\hyperlink{(c)}{(c)} depending on the value of the difference $\|u_k-v_k\|$.

The growth of the constant $C(\e,r)$ as $r\to0$ is difficult to control, and we cannot make the $\ell$-measure set~$\NN^u$ very small. As a consequence, case~\hyperlink{(c)}{(c)} happens rather often. This slows down the fast convergence, usually associated with the quadratic scheme. However, combining this construction with some techniques based on the study of the behaviour of Kantorovich functionals on a pair of trajectories enables one to prove that the Markov operator defines a contraction on the space of measures. This proves Theorem~\ref{t1.1}. 

\smallskip
To summarise, the proof of Theorem~\ref{t1.1} is based on the classical coupling scheme, enhanced with the quadratic convergence \`a la Kolmogorov to cope with difficulties~\hyperlink{(1)}{(1)} and~\hyperlink{(2)}{(2)} described above. The realisation of this scheme meets serious analytic difficulties. Their detailed discussion is presented in Section~\ref{s3}, which is devoted to the proof of Theorem~\ref{t1.1}. In conclusion, let us mention that these ideas apply also to the case in which the dissipativity hypothesis~\hyperlink{H2}{(H$_2$)} is replaced by a weaker condition of global approximate controllability to a point; see~\cite{KNS-2019}. 

\section{Proof of the main result}
\label{s3} 

\subsection{General scheme}
\label{ss-general-scheme}
Let us describe the scheme of the proof of Theorem~\ref{t1.1}. It is based on a sufficient condition for exponential mixing from~\cite[Section~3.1.1]{KS-book}, stated below as Theorem~\ref{expodecay}. Namely, we shall define a majorant of the metric on the space~$\PP(X)$ for which the Markov operator~$\PPPP_1^*$ is a contraction. The construction of that majorant uses an auxiliary result about a suitable coupling for the pair of measures $(P_1(u,\cdot),P_1(u',\cdot))$; see Theorem~\ref{t-coupling}. The latter reduces, with the help of a proposition about transformation of measures under smooth maps, to an approximate solution of a homological equation. The construction of solution is based on the existence of an approximate right inverse for a family of operators depending on a parameter that act between two Hilbert spaces and have a dense image, for almost all values of the parameter. Before turning to the accurate proof, we present the above scheme in more details. A reader not interested in an informal discussion may safely skip the text coming after Theorem~\ref{expodecay} and go to Subsection~\ref{s3.1}.

\subsubsection*{Sufficient condition for exponential mixing}
For reader's convenience, we formulate the result we need in an abstract setting. Let~$X$ be a compact metric space and let $(u_k,\IP_u)$ be a discrete-time Markov process in~$X$ possessing the Feller property. We denote by $P_k(u,\Gamma)$ the corresponding transition function, and by~$\PPPP_k$ and~$\PPPP_k^*$ the Markov operators. Let us consider a bounded Borel-measurable symmetric function $F:X\times X\to\R_+$ such that 
\begin{equation} \label{lowerbound}
	F(u_1,u_2)\ge c\,d_X(u_1,u_2)^\beta\quad\mbox{for any $u_1,u_2\in X$},
\end{equation}
where $\beta\le 1$ and~$c$ are positive numbers not depending on~$u_1$ and~$u_2$. Recall that the {\it Kantorovich functional\/} associated with~$F$ is defined by 
\begin{equation} \label{Kantorovich}
	\KK_F(\mu_1,\mu_2)=\inf_{\xi_1,\xi_2}\E\,F(\xi_1,\xi_2),
\end{equation}
where the infimum is taken over all $X$-valued random variables $\xi_1,\xi_2$ such that $\DD(\xi_1)=\mu_1$ and  $\DD(\xi_2)=\mu_2$. It follows from inequality~\eqref{lowerbound} and the definition of the dual-Lipschitz distance that 
\begin{equation} \label{dL-Kantorovich}
	\|\mu_1-\mu_2\|_L^*\le C\KK_F(\mu_1,\mu_2)
	\quad\mbox{for any $\mu_1,\mu_2\in \PP(X)$},
\end{equation}
where $C=c^{-1}\diam(X)^{1-\beta}$. 
A proof of the following theorem\footnote{In~\cite{KS-book}, the result is proved with $\beta=1$. However, using~\eqref{dL-Kantorovich}, it is straightforward to check that the proof remains valid for any $\beta\in(0,1)$.} can be found in Section~3.1.1 of~\cite{KS-book}. 

\begin{theorem} \label{expodecay}
	Suppose there is  a number $\varkappa\in (0,1)$ and a bounded Borel-measurable symmetric function~$F:X\times X\to\R_+$ satisfying~\eqref{lowerbound} such that 
	\begin{equation} \label{squeezing}
		\KK_F(\PPPP_1^*\mu_1,\PPPP_1^*\mu_2)\le \varkappa\,\KK_F(\mu_1,\mu_2)
		\quad\mbox{for any $\mu_1,\mu_2\in \PP(X)$}.
	\end{equation}
	Then the Markov process $(u_k,\IP_u)$ has a unique stationary measure $\mu\in\PP(X)$, which satisfies inequality~\eqref{1.6} with some positive numbers~$C$ and~$\gamma$. 
\end{theorem}

Thus, to prove Theorem~\ref{t1.1}, it suffices to construct a function~$F$ satisfying the hypotheses of the above result. 

\subsubsection*{Reduction to a coupling}

To construct the function~$F$, we fix a small parameter $d>0$ as in Section~\ref{s1.3} (see there Steps~\hyperlink{(a)}{(a)} and~\hyperlink{(b)}{(b)}) and divide the product space $\XXX = X\times X$ into disjoint subsets $\XXX_n$, $-N\le n\le +\infty$, where $N \sim -\ln d$. Roughly speaking, for $n\ge0$ the set $ \XXX_n$ consists of those pairs $(u,u') \in  \XXX$ for which $\| u-u'\|\sim  q^n d$, where~$q^{-1}$ is the maximum between~2 and the Lipschitz constant of~$S$ on~$X$, while for $k<0$, $ \XXX_k$ is the collection of the pairs  $(u,u') \in  \XXX$ for which $\| u-u'\| >d$ and $\| u\| \vee \| u'\|\sim R b^{N+k}$ for some suitable $R>0$ and $b\in (0,1)$. Then, for $(u,u') \in  \XXX_n$, $n\ge 0$, we should argue as at Step~\hyperlink{(a)}{(a)} of Section~\ref{s1.3}, while for $(u, u')\in  \XXX_k$ with $-N\le k<0$ the argument follows  Step~\hyperlink{(b)}{(b)}. We define the  function $F$ to be constant on the sets $ \XXX_n$: if $n\ge0$, then $F\le 1$ is $\sim q^{\gamma n/2}$ for some $\gamma \in(0,1)$, while if $n<0$, then $F=M_n$ for a suitably chosen decreasing sequence $\{ M_k, -N\le k\le-1\}$. Next, to prove~\eqref{squeezing} for any pair of measures $(P_1(u,\cdot),P_1(u',\cdot))$ with $(u,u')\in  \XXX_n$, we construct a coupling $(V(u,u'), V'(u,u'))$ such that, no matter if $n\ge0$ or $n<0$, we have 
\begin{equation} \label{coupling-inequalities}
\IP\biggl\{(V,V')\in\bigcup_{m\le n-2}\XXX_m\biggr\}\lesssim\|u-u'\|^{\gamma}, \quad 
\IP\biggl\{(V,V')\in\bigcup_{m\ge n+1}\XXX_m\biggr\}\ge 1-\nu,	
\end{equation}
where $\nu\in(0,1)$ is sufficiently small. In other words, with a probability close to~$1$ the distance between the trajectories issued from~$u$ and~$u'$ squeezes by a factor of~$q$ at time $t=1$, whereas the probability that distance increases by a factor of at least~$q^{-2}$ is estimated by~$\|u-u'\|^{\gamma}$. Once these two inequalities are established, the validity of~\eqref{squeezing} follows from a simple computation; see Case~1 in Section~\ref{s3.1}. 

The actual construction is more complicated since we need to take into account the points $(u,u')$ that are far from each other (so that the first inequality in~\eqref{coupling-inequalities} does not give any information). The proof of~\eqref{squeezing}, based on the existence of a coupling as above, is presented in Section~\ref{s3.1}. 

\subsubsection*{Reduction to a homological equation}
We now describe how to construct a coupling $(V,V')$ satisfying~\eqref{coupling-inequalities}, for any $(u,u') \in  \XXX$. We know that $(u,u') \in  \XXX_n$ for some  $n\ge -N$. If $n\le -1$ or $n =\infty$, then $\| u-u'\|>d$ or $u=u'$, and we choose $V= S(u, \zeta)$, $V'= S(u', \zeta)$, where $\DD(\zeta) = \ell$. If $n\ge0$, then $\|u-u'\|\le d$, and to construct the coupling we use a general result about the transformation of probability measure under smooth finite-dimensional perturbations of the identical map.  Roughly speaking, it says that if~$\ell\in\PP(E)$ is a decomposable measure with compact support, and $\varPsi:E\to E$ is a smooth map of the form $\varPsi(\zeta)=\zeta+\varPhi(\zeta)$, where~$\varPhi$ has a finite-dimensional range and satisfies the inequalities
\begin{equation} \label{estimates-for-Phi}
	\|\varPhi(\zeta)\|\le\varkappa, \quad 
	\|\varPhi(\zeta)-\varPhi(\zeta')\|\le\varkappa\|\zeta-\zeta'\|\quad\mbox{for $\ell$-a.e.\ $\zeta,\zeta'\in E$}, 
\end{equation}
then the total variation distance between~$\ell$ and~$\varPsi_*(\ell)$ can be estimated by~$C\varkappa^\beta$, with some $\beta\in(0,1)$. (We refer the reader to Theorem~\ref{T:4.1} for an exact statement.) Suppose now we have constructed finite-dimensional maps $\varPhi^{u,u'}:E\to E$ that are perturbations of identity (in the sense that~\eqref{estimates-for-Phi} holds for them with $\varkappa\simeq\|u-u'\|$) and satisfy the inequality 
\begin{equation} \label{squeezing-Phiuueta}
	\|S(u,\eta)-S(u',\eta+\varPhi^{u,u'}(\eta))\|\le q\|u-u'\| 
\end{equation}  
for $u,u'\in \XXX$ and $\ell$-a.e.~$\eta\in E$. Then the application of the above result to $\varPsi^{u,u'}=\Id+\varPhi^{u,u'}$ enables one to conclude that 
$$
\|\ell-\varPsi_*^{u,u}(\ell)\|_{\mathrm{var}}\lesssim\|u-u'\|^\beta. 
$$
Using now the gluing lemma (see Theorem~\ref{t-gluing}), we can construct a coupling $(V,V')$ satisfying~\eqref{coupling-inequalities}. 

We thus need to construct a small smooth map $\varPhi^{u,u'}$ satisfying~\eqref{squeezing-Phiuueta}. To this end, we argue as at Step~\hyperlink{(b)}{(b)} in Section~\ref{s1.3}. Namely, setting  $\zeta=\varPhi^{u,u'}(\eta)$, using Taylor's expansion, and ignoring the second-order terms, we can rewrite~\eqref{squeezing-Phiuueta} as 
\begin{equation} \label{squeezing-inequality}
	\|(D_uS)(u,\eta)(u'-u)+(D_\eta S)(u,\eta)\zeta\|\lesssim q\|u-u'\|. 
\end{equation}
If the map~$D_\eta S(u,\eta)$ had a full image, we could annihilate the left-hand side of this inequality. However, this is not the case, and we arrive at the problem of solving (approximately) the {\it homological equation\/} (cf.~\eqref{d6})
\begin{equation} \label{homological-equation}
(D_\eta S)(u,\eta)\zeta=-(D_uS)(u,\eta)(u'-u). 	
\end{equation}

\subsubsection*{Solving the homological equation}
It is exactly here where Hypothesis~\hyperlink{H3}{(H$_3$)} comes into play. Namely, we know that the image of $(D_\eta S)(u,\eta)$ is dense in~$E$ for any $u\in X$ and $\ell$-a.e.~$\eta\in E$. Hence, when $u$ is fixed, we can find an approximate solution of~\eqref{homological-equation} for almost every~$\eta$. However, this is not sufficient, since the resulting function $\zeta=\varPhi^{u,u'}(\eta)$ must satisfy a number of properties. More precisely, $\varPhi^{u,u'}$ should have a finite-dimensional range, be a function of order $\|u-u'\|$, and solve approximately~\eqref{homological-equation} for all close pairs $(u,u')\in\XXX$ and $\ell$-almost all $\eta\in E$. This will be achieved with the help of a construction of approximate right inverses for a family of linear operators with dense images. The corresponding result is given in Theorem~\ref{p2.5}, which is one of the main technical novelties of this paper. 

\subsection{Reduction to a coupling}
\label{s3.1}
In this subsection, we construct a symmetric function $F:X\times X\to\R_+$ satisfying~\eqref{lowerbound} and prove that~\eqref{squeezing} holds for it. Replacing~$S$ by the mapping $\widetilde S(u,\eta)=S(u+\hat u,\eta+\hat \eta)-\hat u$, we may assume without loss of generality that  $\hat u=0$ and $\hat\eta=0$.  

Our construction will depend on four parameters $q,b\in(0,1)$ and $R,d >0$, three of which are fixed now. Namely, let $b=\frac{a+1}{2}$, where $a\in(0,1)$ is the number in~\eqref{1.2},  let $R>0$ be such that $X\subset B_H(R)$, and let $q\in(0,\frac12]$ satisfy the inequality
\begin{equation} \label{lipschitz}
	\|S(u_1,\zeta)-S(u_2,\zeta)\|\le q^{-1}\|u_1-u_2\|
	\quad\mbox{for any $u_1,u_2\in X$, $\zeta\in\KK$};
\end{equation}
the existence of such a number~$q$ is implied by Hypothesis~\hyperlink{H1}{(H$_1$)}. We denote by $N=N(d)\ge1$ the least integer satisfying the inequality $b^{N}R\le d/2$, define $\XXX=X\times X$, and for $n\ge 0$ and  $-N\le k\le-1$ introduce the pairwise disjoint sets
\begin{align}
	\XXX_{\!\infty}&=\bigl\{(u,u')\in\XXX:u=u'\bigr\},\label{Xinfty}\\
	\XXX_{\!n}&=\bigl\{(u,u')\in\XXX:q^{n+1} d <\|u-u'\|\le q^n d\bigr\},\label{Xn}\\
	\XXX_{\!k}&=\bigl\{(u,u')\in\XXX: \|u-u'\|> d ,
	Rb^{N+k+1}<\|u\|\vee\|u'\|\le Rb^{N+k}\bigr\}.\label{Xk}
\end{align}
It is straightforward to check that~$\XXX$ is the union of the sets~$\{\XXX_{\!n}, -N\le n\le\infty\}$. Recall that $P_k(u,\Gamma)$ stands for the transition function of the Markov process defined by~\eqref{1.1}. A key observation when proving~\eqref{squeezing} is the following result.

\begin{theorem} \label{t-coupling}
	Under the hypotheses of Theorem~\ref{t1.1}, there are  $\gamma\in(0,1]$ and $C>0$ such that, for any $\nu\in(0,1)$, there is $d_0\in(0,1)$ possessing the following property: for any $d\in(0,d_0)$ we can construct a number $p\in(0,1)$, a probability space $(\Omega,\FF,\IP)$  and measurable functions $V,V':X\times X\times\Omega\to X$ (a~coupling) such that the following assertions hold.
	\begin{itemize}
		\item [\bf(a)]
		For any $(u,u')\in \XXX$, the laws of $V(u,u';\cdot)$ and $V'(u,u';\cdot)$ coincide with $P_1(u,\cdot)$ and $P_1(u',\cdot)$, respectively. Moreover, $V(u,u;\cdot)=V'(u,u;\cdot)$ almost surely for any $(u,u)\in\XXX_{\!\infty}$.
		\item [\bf(b)]
		For any $(u,u')\in\XXX_{\!n}$, we have 
		\begin{align} 
			\IP\bigl\{\bigl(V(u,u'),V'(u,u')\bigr)\in\XXX_{\!m}
			\mbox{ for some $m\ge n+1$}\bigr\}
			&\ge 1-\nu, \label{transition1}\\
			\IP\bigl\{\bigl(V(u,u'),V'(u,u')\bigr)\in\XXX_{\!m}
			\mbox{ for some $m\le n-2$}\bigr\}&\le C\,\|u-u'\|^\gamma,
			\label{transition2}
		\end{align}
		where $n\ge0$ in~\eqref{transition1} and $n\ge1$ in~\eqref{transition2}. 		
		\item [\bf(c)]
		For $-N\le k\le -1$ and $(u,u')\in\XXX_{\!k}$, we have 
		\begin{align} 
			\IP\bigl\{\bigl(V(u,u'),V'(u,u')\bigr)\in\XXX_{\!m}
			\mbox{ for some $m\ge k+1$}\bigr\}
			\ge p. \label{transition3}
		\end{align}	
		\end{itemize}
\end{theorem}

Taking this result for granted, let us complete the proof of the theorem. Let $\gamma\in(0,1]$ and $C>0$ be the numbers constructed in Theorem~\ref{t-coupling}. We fix $\nu>0$ so small that\footnote{The role of inequalities~\eqref{qnu} and~\eqref{nu-kappa2} will be clarified below.} 
\begin{align}	
q^{\gamma/2}+q^{-\gamma/2}\nu &<1, \label{qnu}\\
q^{\gamma/2}+3\nu &<1. \label{nu-kappa2}
\end{align}
Let $d_0>0$ be the constant constructed in Theorem~\ref{t-coupling} for the above choice of~$\nu$ and let $d\in(0,d_0)$ be a number that will be chosen below; once it is fixed, the integer~$N$ and the sets~$\XXX_{\!n}$ with $-N\le n\le +\infty$ are uniquely determined. We define 
\begin{equation} \label{Fuu}
	F(u,u')=\left\{
	\begin{array}{cl}
	0 & \mbox{for $(u,u')\in\XXX_{\!\infty}$},\\
	(q^{n}d)^{\gamma/2}  & \mbox{for $(u,u')\in\XXX_{\!n}$},\\
	M_k & \mbox{for $(u,u')\in\XXX_{\!k}$},	
	\end{array}
	\right.
\end{equation} 
where $n\ge0$, $-N\le k\le -1$, and $M_k\ge 2d^{\gamma/2}$ is a decreasing sequence to be chosen below. It is straightforward to see that~$F$ satisfies~\eqref{lowerbound}. We shall prove that inequality~\eqref{squeezing} holds with some~$\varkappa\in(0,1)$.  

To this end, we first reduce the proof to the particular case in which~$\mu_1,\mu_2$ are Dirac masses. Namely, denoting by $(V(u,u'),V'(u,u'))$ the random variables constructed in Theorem~\ref{t-coupling}, suppose we have proved that
\begin{equation} \label{squeezing-delta}
	\E F\bigl(V(u,u'),V'(u,u')\bigr)
	\le\varkappa \KK_F(\delta_{u},\delta_{u'})
	=\varkappa F(u,u')\quad 
	\mbox{for $(u,u')\in\XXX$}. 
\end{equation}
Let us take any measures $\mu_1,\mu_2\in\PP(X)$. For any $\theta>0$, there are $X$-valued random variables $\xi_1,\xi_2$ such that 
\begin{equation} \label{KD}
	\KK_F(\mu_1,\mu_2)\ge\E\,F(\xi_1,\xi_2)-\theta.
\end{equation}
Now note that the random variables $(V(u,u'),V'(u,u'))$  can be assumed to be independent of~$(\xi_1,\xi_2)$. In this case, the pair $(V(\xi_1,\xi_2),V'(\xi_1,\xi_2))$ is a coupling for~$(\PPPP_1^*\mu_1,\PPPP_1^*\mu_2)$. Using again the independence and relations~\eqref{squeezing-delta} and~\eqref{KD}, we derive
\begin{align*}
\E F\bigl(V(\xi_1,\xi_2),V'(\xi_1,\xi_2)\bigr)
\le \varkappa\,\E F(\xi_1,\xi_2)=\varkappa\,(\KK_F(\mu_1,\mu_2)+\theta). 	
\end{align*}
Since $\theta>0$ was arbitrary, this proves~\eqref{squeezing}. 

To establish~\eqref{squeezing-delta}, notice that there is nothing to prove when $(u,u')\in\XXX_{\!\infty}$. For $(u,u')\notin\XXX_{\!\infty}$, we abbreviate $F(u,u')=:F$ and distinguish between four cases, assuming that the parameters~$b$, $R$, $q$, $C$, and~$\gamma$ are  fixed (see~\eqref{lipschitz} and Theorem~\ref{t-coupling}).

\smallskip
{\it Case 1: $(u,u')\in\XXX_{\!n}$ with $n\ge1$, so $F=(q^nd)^{\gamma/2}$\/}. It follows from of~\eqref{transition1}, \eqref{transition2}, and~\eqref{Fuu} that
\begin{align*}
\E F\bigl(V(u,u'),V'(u,u')\bigr)
&\le q^{\gamma/2} F\,\IP(G_n^1)+M_{-N} \,\IP(G_n^2)
+q^{-\gamma/2}F\,\IP(G_n^3)\\
&\le F\bigl(q^{\gamma/2}\,\IP(G_n^1)+q^{-\gamma/2}\IP(G_n^3)+M_{-N}(q^nd)^{-\gamma/2}\IP(G_n^2)\bigr)\\
&=:F\varkappa_1,
\end{align*}
where $G_n^1$ and~$G_n^2$ denote the events on the left-hand sides   of~\eqref{transition1} and~\eqref{transition2}, respectively, $G_n^3$~is the complement of~$G_n^1\cup G_n^2$ corresponding to the event $\{(V,V')\in \XXX_{\!n}\cup \XXX_{\!n-1}\}$, and we used the fact that~$\{M_k\}$ is a decreasing sequence. The required inequality~\eqref{squeezing-delta} will be established if we prove that  $\varkappa_1<1$, uniformly in $n\ge1$. To this end, notice that  $\IP(G_n^3)\le \nu$ in view of~\eqref{transition1}, so that 
\begin{equation} \label{estimate1}
	q^{\gamma/2}\,\IP(G_n^1)+q^{-\gamma/2}\IP(G_n^3)
	\le q^{\gamma/2}+q^{-\gamma/2}\nu.
\end{equation}
Furthermore, it follows from~\eqref{transition2} that 
\begin{equation*}
	M_{-N}(q^nd)^{-\gamma/2} \IP(G_n^2)\le CM_{-N}(q^nd)^{\gamma/2}
	\le CM_{-N}d^{\gamma/2}.
\end{equation*}
Combining this with~\eqref{estimate1} and~\eqref{qnu}, we see that $\varkappa_1<1$, provided that 
\begin{equation} \label{kappa1}
	CM_{-N}d^{\gamma/2}<1-q^{\gamma/2}-q^{-\gamma/2}\nu. 
\end{equation}
Let us note that~$N$ depends on the choice of~$d$, so that the parameters~$M_{-N}$ and~$d$ are not independent. Our choice of~$M_k$ will ensure that~$M_{-N}\le 3d^{\gamma/2}$, so that~\eqref{kappa1} will be satisfied if 
\begin{equation} \label{kappa11}
	3Cd^\gamma<1-q^{\gamma/2}-q^{-\gamma/2}\nu. 
\end{equation}
In what follows, we fix $d\in(0,d_0)$ satisfying~\eqref{kappa11}. Together with~$b$, they determine~$N\ge1$ as the least positive integer satisfying $b^NR\le d/2$. 

\smallskip
{\it Case 2: $(u,u')\in\XXX_{\!0}$, so $F=d^{\gamma/2}$\/}. Let us set 
\begin{equation} \label{Mk}
	M_k=2 d^{\gamma/2} +\e(B^{N-1}-B^{N+k}),
\end{equation} 
where $\e>0$ and $B>1$ will be chosen below. Arguing as above, using~\eqref{transition1}, and assuming that $\e\le B^{1-N} d^{\gamma/2}$ (so that $M_{-N}\le 3d^{\gamma/2}$), we derive 
\begin{align*}
\E F\bigl(V(u,u'),V'(u,u')\bigr)
&\le q^{\gamma/2} F\,\IP(G_0^1)+M_{-N}\bigl(1-\IP(G_0^1)\bigr)\\
&\le F\bigl(q^{\gamma/2}+d^{-\gamma/2} M_{-N}\nu\bigr)
\le F\bigl(q^{\gamma/2}+3\nu\bigr)=:F\varkappa_2.
\end{align*}
In view of~\eqref{nu-kappa2}, we have $\varkappa_2<1$. 

\smallskip
{\it Case 3: $(u,u')\in\XXX_{\!-1}$, so $F=M_{-1}=2d^{\gamma/2}$\/}. It follows from~\eqref{transition3} with $m=-1$ that 
\begin{align*}
\E F\bigl(V(u,u'),V'(u,u')\bigr)
&\le p d^{\gamma/2} +(1-p)\bigl(2 d^{\gamma/2} +\e(B^{N-1}-1)\bigr)\\
& \le F\bigl(1-\tfrac{p}{2}+\e (2d^{\gamma/2})^{-1}(1-p)B^{N-1}\bigr)
=:F\varkappa_3.
\end{align*}
It is straightforward to check that $\varkappa_3\le 1-\frac{p}{4}<1$, provided that 
$$
\e\le p(1-p)^{-1}B^{-N} d^{\gamma/2},\quad B\ge2.
$$ 

\smallskip
{\it Case 4: $(u,u')\in\XXX_{\!k}$ with $-N\le k\le -2$, so $F=M_k$\/}. Using~\eqref{transition3} and~\eqref{Mk}, we derive
\begin{align*}
\E F\bigl(V(u,u'),V'(u,u')\bigr)
&\le p\bigl(2 d^{\gamma/2} +\e(B^{N-1}-B^{N+k+1})\bigr)\\
&\quad+(1-p)\bigl(2 d^{\gamma/2} +\e(B^{N-1}-1)\bigr)\\
&=2 d^{\gamma/2} +\e(B^{N-1}-pB^{N+k+1}-1+p).
\end{align*}
Let us set $B=2/p\ge 2$. Then the right-most term in the above inequality does not exceed $\varkappa_4 F$ with 
$$
\varkappa_4=1-\frac13\e d^{-\gamma/2}.
$$ 
Comparing the restrictions imposed on the parameters, we see that~\eqref{squeezing-delta} holds with $\varkappa=\max\{\varkappa_i,1\le i\le 4\}$, provided that 
\begin{equation} \label{eps}
\e=B^{-N} d^{\gamma/2}\min\bigl\{p(1-p)^{-1},B\bigr\}.	
\end{equation}

Relation~\eqref{squeezing-delta} implies inequality~\eqref{squeezing}, and the exponential mixing~\eqref{1.6} follows. Thus, to complete the proof of Theorem~\ref{t1.1}, it remains to establish Theorem~\ref{t-coupling}, which is done in the next subsection. 

\subsection{Proof of Theorem~\ref{t-coupling}}
\label{s3.3}
Let us fix an arbitrary $\nu\in(0,1)$. To define the mappings~$V$ and~$V'$, we first consider the case in which either $u=u'$ or $\|u-u'\|>d$, where $d>0$ is arbitrary for the moment. Let us denote by~$\zeta$ a random variable such that $\DD(\zeta)=\ell$. We set
$$
V(u,u',\cdot)=S(u,\zeta), \quad V'(u,u',\cdot)=S(u',\zeta). 
$$
Then $(V,V')$ satisfies property~(a). Recalling~\eqref{Xk}, we see that~\eqref{transition3} will be established if we prove that 
\begin{equation} \label{decayproba}
	\IP\bigl\{\|V\|\vee\|V\|\le b(\|u\|\vee\|u'\|)\bigr\}\ge p,
\end{equation}
where $p>0$ is a number depending only on~$d$ (but not on the vectors~$u,u'\in X$). To see this, note that $\|u\|\vee\|u'\|\ge r:=d/2$. By Hypothesis~\hyperlink{H2}{(H$_2$)} and the Lipschitz property of $S(u,\cdot):\KK\to H$, we have
 \begin{align}
	\|S(u,\zeta)\|  &\le \|S(u,0)\|+ C\|\zeta\|_E
	\le a\|u\|+C (r^{-1}\|u\|\vee\|u'\|)\|\zeta\|\notag\\
	&\le \bigl(a+Cr^{-1}\|\zeta\|\bigr)\|u\|\vee\|u'\|. \label{E:3.2}
\end{align} 
The right-most term of this inequality does not exceed $b(\|u\|\vee\|u'\|)$ provided that $\|\zeta\|\le C^{-1}r(b-a)$, and a similar estimate holds for $\|S(u',\zeta)\|$. It follows that the probability on the left-hand side of~\eqref{decayproba} is minorised by $\IP\{\|\zeta\|\le C^{-1}r(b-a)\}$. This quantity is positive because $\hat\eta=0$ is in the support of the law~$\ell$. 

\smallskip

Let us turn to the case $\|u-u'\|\le d$, with a sufficiently small~$d>0$ to be specified below. Given any $\delta>0$, we set $D_\delta=\{(u,u')\in X\times H:\|u-u'\|\le\delta\}$. We shall need the following auxiliary result whose proof is given in Section~\ref{ss-proof-of-P2.5}. 

\begin{proposition}\label{P:2.5} 
For any $\sigma, \theta\in(0,1)$ there are positive numbers $C$, $\beta$, and~$\delta$, a Borel-measurable   mapping $\varPhi : X\times H \times E  \to E$, and a family of Borel subsets $\{\KK^{u,\sigma,\theta}\subset \KK^u\}_{u\in X}$ such that $\varPhi^{u,u'}(\eta)=0$ if $\eta\notin \KK^{u,\sigma,\theta}$ or $u'=u$, and we have the following inequalities, in which $\varPsi^{u,u'}(\eta)=\eta+\varPhi^{u,u'}(\eta):$ 
\begin{align}
\ell(\KK^{u,\sigma,\theta})&\ge 1-\sigma,
\label{3.7}\\
\|\ell-\varPsi_*^{u,u'}( \ell)\|_{\mathrm{var}}&\le C\, \|u-u'\|^\beta,
 \label{3.8}\\
 \|S(u,\eta)- S(u',\varPsi^{u,u'}(\eta))\| &\le \theta\, \|u-u'\|,
\label{3.9}
\end{align} 
where $(u,u')\in D_\delta$ and $\eta\in \KK^{u,\sigma,\theta}$.
\end{proposition}

From now on, we assume that $d\le\delta$, where~$\delta>0$ is the number constructed in Proposition~\ref{P:2.5} with some parameters~$\sigma$ and~$\theta$ to be chosen below. We denote by~$\eta$ an $E$-valued random variable with law~$\ell$ and by $\MMMM_{u,u'}\in\PP(E\times E)$ the law of the pair $(\eta,\varPsi^{u,u'}(\eta))$, where $\varPsi^{u,u'}:E\to E$ is the mapping constructed in Proposition~\ref{P:2.5}. Let $\NNNN_{u,u'}\in\PP(E\times E)$ be a maximal coupling for the pair $(\varPsi^{u,u'}_*(\ell),\ell)$. That is, $\NNNN_{u,u'}$ is a probability measure on~$E\times E$ with marginals $\varPsi^{u,u'}_*(\ell)$ and~$\ell$ such that 
\begin{equation} \label{mc}
	\NNNN_{u,u'}\bigl(\{(\zeta,\zeta')\in E\times E:\zeta\ne\zeta'\}\bigr)
	=\|\varPsi^{u,u'}_*(\ell)-\ell\|_{\mathrm{var}}. 
\end{equation}
In view of Theorem~1.2.28 in~\cite{KS-book}, we can choose~$\NNNN_{u,u'}$ to be measurable in $(u,u')$, i.e., to be a random probability measure on~$E\times E$ with the underlying space~$D_\delta$. By construction, the projections of~$\MMMM_{u,u'}$ and~$\NNNN_{u,u'}$ to, respectively, the second and the first components coincide for any $(u,u')\in D_\delta$. Therefore, by the gluing lemma (see Theorem~\ref{t-gluing} and the remark following it), there is a probability space~$(\Omega,\FF,\IP)$ and measurable functions $\zeta,\hat\zeta,\zeta':D_\delta\times \Omega\to E$ such that, for any $(u,u')\in D_\delta$, 
\begin{equation} \label{zzz}
	\DD\bigl(\zeta(u,u',\cdot),\hat\zeta(u,u',\cdot)\bigr)=\MMMM_{u,u'},\quad
	\DD\bigl(\hat\zeta(u,u',\cdot),\zeta'(u,u',\cdot)\bigr)=\NNNN_{u,u'}. 
\end{equation}
The very definition of~$\zeta$ and~$\zeta'$ implies that their laws coincide with~$\ell$ for any $(u,u')\in D_\delta$. Therefore, defining the functions
$$
V(u,u',\omega)=S(u,\zeta(u,u',\omega)), \quad 
V'(u,u',\omega)=S(u,\zeta'(u,u',\omega)),
$$
we see that their laws coincide with $P_1(u,\cdot)$ and~$P_1(u',\cdot)$, respectively. Let us prove~\eqref{transition1} and~\eqref{transition2}. 

To this end, for any $(u,u')\in D_\delta$ we introduce the events\footnote{We shall often omit the argument~$\omega$ to simplify formulas.}
\begin{align*}
	\Omega_1^{u,u'}&:=\{\zeta\in\KK^{u,\sigma,\theta}\}\cap\{\hat\zeta(u,u')=\zeta'(u,u')\},\\
	\Omega_2^{u,u'}&:=\{\zeta\in\KK^{u,\sigma,\theta}\}\cap\{\hat\zeta(u,u')\ne\zeta'(u,u')\},\\
	\Omega_3^{u,u'}&:=\{\zeta\notin\KK^{u,\sigma,\theta}\}.
\end{align*}
These events form a partition of the probability space~$\Omega$, and it follows from~\eqref{3.7}, \eqref{3.8}, \eqref{mc}, and~\eqref{zzz} that 
\begin{align}
	\IP(\Omega_1^{u,u'})&\ge 1-\sigma-C_1\delta^\beta,\label{omega1}\\
	\IP(\Omega_2^{u,u'})&\le C_1\|u-u'\|^\beta,\label{omega2}
\end{align}
where $C_1$, $\beta$, and~$\delta$ depend on~$\sigma$ and~$\theta$. Moreover, it follows from~\eqref{3.9} and~\eqref{lipschitz} that, on the event~$\Omega_1^{u,u'}$, we have
\begin{align}
	\|V(u,u')-V'(u,u')\|&=\|S(u,\zeta(u,u'))-S(u',\hat\zeta(u,u'))\|
	\le \theta\|u-u'\|, \label{trans1}
\end{align}
whereas on $\Omega_3^{u,u'}$, 
\begin{equation} \label{trans2}
	\|V(u,u')-V'(u,u')\|=\|S(u,\zeta(u,u'))-S(u',\zeta(u,u'))\|
	\le q^{-1}\|u-u'\|. 
\end{equation}
We now specify our choice of the parameters. Let 
\begin{equation} \label{st}
	\sigma=\frac{\nu}{2}, \quad \theta=q,
\end{equation}
let $C=:C_1$, $\beta$ and~$\delta$ be the constants constructed in Proposition~\ref{P:2.5}, and let $d_0\in(0,\delta)$ be so small that $C_1d_0^\beta\le\frac{\nu}{2}$; cf.~\eqref{omega1} and~\eqref{omega2}. It follows from~\eqref{trans1}, \eqref{trans2}, and~\eqref{st} that, for $(u,u')\in\XXX_{\!n}$, we have 
\begin{align*}
	\{(V,V')\in\XXX_{\!m}\mbox{ for some $m\ge n+1$}\}&\supset \Omega_1^{u,u'},\\
	\{(V,V')\in\XXX_{\!m}\mbox{ for some $m\le n-2$}\}&\subset \Omega_2^{u,u'},
\end{align*}
where $n\ge0$ for the first inclusion and $n\ge1$ for the second. The required inequalities~\eqref{transition1} and~\eqref{transition2} (with $\gamma=\beta$ and $C=C_1$) follow now from~\eqref{omega1} and~\eqref{omega2}, respectively. This completes the proof of Theorem~\ref{t-coupling}. 

\medskip
Thus, it remains to establish Proposition~\ref{P:2.5}. To do it, we need two auxiliary results about transformation of measures under Lipschitz maps close to identity and existence of approximate right inverse for a family of operators with dense image. Since those results may have some independent interest, we formulate and prove them in an abstract setting.

\subsection{Image of measures under piecewise Lipschitz maps}
\label{s2.2}
Let~$E$ be a separable Hilbert space, let~$\KK\subset E$ be a compact subset, and let $\varPsi:\KK\to E$ be a map of the form $\varPsi(\eta)=\eta+\varPhi(\eta)$, where $\varPhi:\KK\to E$ is a Borel-measurable ``small'' map vanishing outside~$\KK$. Our goal is to study the transformation of measures on~$\KK$ under~$\varPsi$. We shall consider the situation in which the maps $\varPhi$ and $\varPsi$ are piecewise Lipschitz in the following  sense: their restrictions to a ``large'' closed subset $\KK_1\subset\KK$ is Lipschitz, while outside~$\KK_1$ they are, respectively, the zero-map and the identity-map.

\smallskip
Suppose that $E$ is the direct sum of closed subspaces $\EE$ and $\EE'$, where $\dim\EE<+\ty$, and  let~${\mathsf P}_\EE$ and~${\mathsf P}_{\EE'}$ be  the associated projections. We assume that the image of~$\varPhi$ is contained in~$\EE$. Given any subset $A\subset E$, we denote by~$A'$ its projection to~$\EE'$, and for any $w\in A'$, we write $A(w)=\{v\in \EE: v+w\in A\}$. Let $\ell \in \PP(E)$ be a measure that is supported by~$\KK$ and is representable as the tensor product of its projections~$\ell_\EE$ and~$\ell_{\EE'}$ under~${\mathsf P}_{\EE}$ and ${\mathsf P}_{\EE'}$, respectively. We assume that~$\ell_\EE$ has  a Lipschitz-continuous density~$\rho$ with respect to the Lebesgue  measure on~$\EE$.

\begin{theorem} \label{T:4.1}
Let $\varPhi:\KK\to \EE\subset E$ be a Borel-measurable map that possesses the following properties:
\begin{itemize}
\item[\bf(a)]
There is a  positive number $\varkappa$ and a closed subset $\KK_1\subset\KK$ such that $\varPhi\big|_{\KK\setminus\KK_1}=0$ and 
\begin{equation} \label{4.1}
\|\varPhi(\eta)\|\le\varkappa, \quad 
\|\varPhi(\eta)-\varPhi(\eta')\|\le\varkappa\|\eta-\eta'\|\quad
\mbox{for $\eta,\eta'\in\KK_1$}. 
\end{equation}
\item[\bf(b)]
There are positive numbers $c$ and $\gamma$ such that, for any  $w\in \KK_1'$ and  $r\in[0,1]$, we have
\begin{equation} \label{4.2}
\Leb\bigl\{v\in \EE:\dist\bigl(v,\p_w\KK_1\bigr)\le r\bigr\}\le c \,r^\gamma,
\end{equation}
where $\p_w\KK_1=\KK_1(w)\cap \overline{\KK_1(w)^c}$ and $\KK_1(w)^c=\KK(w)\setminus\KK_1(w)$.
\end{itemize}
Let $\ell\in\PP(E)$ be a measure satisfying the above hypotheses such that $\supp\ell\subset\KK$. Then there are positive numbers~$C$ and~$\beta$ depending only on~$\ell$, $c$, and~$\gamma$ such that
\begin{equation} \label{4.3}
\|\ell-\varPsi_*(\ell)\|_{\mathrm{var}}\le C \varkappa^\beta. 
\end{equation}
\end{theorem}

\begin{proof}
Note that the assertion of the theorem is trivial if~$\varkappa$ is separated from zero, so that we shall consider the case $\varkappa\le \frac12$. We wish to estimate the supremum of the absolute value of the expression $\langle f\circ\varPsi,\ell\rangle-\langle f,\ell\rangle$ over all indicator functions~$f$ of Borel sets in~$\KK$. To this end, we use the Fubini theorem to write
\begin{align}
\langle f\circ\varPsi,\ell\rangle 
&=\int_{\KK'}\ell_{\EE'}(\dd w)\int_{\KK(w)} f(v+w+\varPhi(v+w))\rho(v)\dd v,
\label{4.5}\\
\langle f,\ell\rangle 
&=\int_{\KK'}\ell_{\EE'}(\dd w)\int_{\KK(w)} f(v+w)\rho(v)\dd v.
\label{4.6}
\end{align}
Suppose we have shown that 
\begin{equation} \label{4.011}
	\biggl|\int_{\KK_1(w)} 
	\bigl(f(v+w+\varPhi(v+w))-f(v+w)\bigr)\rho(v)\dd v\biggr|
	\le C_1\varkappa^\beta,
\end{equation}
where $C_1$ and~$\beta$ do not depend on~$w$, $f$, and~$\varkappa$. In this case, taking the absolute value of the difference between~\eqref{4.5} and~\eqref{4.6}, using that $\varPhi$ is zero outside~$\KK_1$, and estimating the interior integral with the help of~\eqref{4.011}, for any indicator function~$f$ we derive 
$$
|\langle f\circ\varPsi,\ell\rangle -\langle f,\ell\rangle|\le C\varkappa^\beta. 
$$
Since~$f$ is arbitrary, we arrive at the required estimate~\eqref{4.3}. Thus, we need to establish~\eqref{4.011}. 

We first outline the main idea. Suppose that $\KK_1(w)$ coincides with the whole space~$\EE$. In this case, we can make a change of variable $v\mapsto v+\varPhi(v+w)=v'$ and rewrite the integral $\int_{\EE} f(v+w+\varPhi(v+w))\rho(v)\dd v$ in the form
$$
\int_\EE f(v'+w)
\frac{\rho(\Theta_w(v'))}{\det\bigl(I+(D\varPhi)(\Theta_w(v')+w)\bigr)}\dd v',
$$
where $\Theta_w(v')$ is the inverse of $v+\varPhi(v+w)$ with respect to~$v$. This expression is easy to compare with $\int_{\EE} f(v+w)\rho(v)\dd v$ due to inequalities~\eqref{4.1}. However, the set~$\KK_1(w)$ may have a complicated structure, and to carry out the above mentioned change of variables, we need to extend~$\varPhi$ to the whole space and to introduce some truncations not to change much the values of the integrals. 

Let us turn to the accurate proof of~\eqref{4.011}. We first extend the mapping~$\varPhi$ from~$\KK_1$ to~$E$. To this end, we use the following result, whose proof can be found in~\cite{valentine-1945} (see also Section~2.10.43 of~\cite{federer1969} for the finite-dimensional case, which is not sufficient for our purposes).

\begin{proposition}[Kirszbraun theorem]
Let $E_1$ and $E_2$ be two Hilbert spaces, let $A\subset E_1$ be a set and $\varPhi : A \to E_2$ be a Lipschitz-continuous function, with a Lipschitz constant~$\varkappa$. Then there is a function $\tilde \varPhi:E_1\to E_2$ that coincides with~$\varPhi $ on~$A$ and is Lipschitz continuous with the same constant~$\varkappa$.
\end{proposition}

Let us denote by~$\tilde\varPhi:E\to \EE$ a Lipschitz-continuous function, with Lipschitz constant~$\le\varkappa$, that coincides with~$\varPhi$ on~$\KK_1$. Since~$\KK_1$ is compact, we can multiply~$\tilde\varPhi$ by a cut-off function, so that there is no loss of generality in assuming that~$\tilde\varPhi$ has a bounded support and Lipschitz constant~$\le C\varkappa$, where~$C>0$ does not depend on~$\varkappa$. 

To introduce truncations, we define the sets
\begin{align*}
\KK_r^+(w)&=\{v\in \KK(w):\dist(v,\KK_1(w))\le r\}, 	\\
\KK_r^-(w)&=\{v\in \KK(w):\dist(v,\KK_1(w)^c)\le r\},
\end{align*}
where $r>0$ is a small parameter chosen below. Let us consider the functions
\begin{align*}
\chi_r^+(v,w)
&=\frac{\dist(v,\KK_r^+(w)^c)}{\dist(v,\KK_r^+(w)^c)+\dist(v,\KK_1(w))}, \\
\chi_r^-(v,w)
&=\frac{\dist(v,\KK_1(w)^c)}{\dist(v,\KK_1(w)^c)+\dist(v,\KK_r^-(w)^c)}. 
\end{align*}
These are non-negative functions bounded by~$1$ and Lipschitz continuous with constant~$r^{-1}$. Since~$f$ is a non-negative function and the support of~$\rho$ is equal to~$\KK(w)$, we can write
\begin{align*}
\int_{\KK_1(w)}f(v+w+\varPhi(v+w))\rho(v)\dd v
&\le \int_{\EE} f(v+w+\tilde\varPhi(v+w))\chi_r^+(v,w)\rho(v)\dd v,\\
\int_{\KK_1(w)}f(v+w)\rho(v)\dd v
& \ge \int_{\EE} f(v+w)\chi_r^-(v,w)\rho(v)\dd v\\
&\ge\int_{\EE} f(v+w)\chi_r^+(v,w)\rho(v)\dd v-c\,\|\rho\|_\infty r^\gamma,
\end{align*}
where we used the fact that $|\chi_r^-(v,w)-\chi_r^+(v,w)|$ is a function bounded by~$1$ and supported in the set entering the left-hand side of~\eqref{4.2}. Now note that,  for $\varkappa\le\frac12$, the mapping  $v\mapsto v+\tilde\varPhi(v+w)$ is a bi-Lipschitz homeomorphism of~$\EE$. Therefore, denoting by~$\delta(f)$ the expression under the absolute value in~\eqref{4.011} and using, for instance, Theorem~3.2.5 in~\cite{federer1969} to make a change of variable, we obtain
\begin{align}
\delta(f)
&\le \int_{\EE} f(v+w)\frac{\chi_r^+(\Theta_w(v),w)\rho(\Theta_w(v))}{\det\bigl(I+(D\tilde\varPhi)(\Theta_w(v)+w)\bigr)}\,\dd v\notag\\
&\qquad -\int_{\EE} f(v+w)\chi_r^+(v,w)\rho(v)\dd v
+C_1 r^\gamma\notag\\
&\le \int_B f(v+w)|\Delta(v,w)|\dd v +C_1 r^\gamma,
\label{2.020}
\end{align}
where $B\subset\EE$ is a large ball containing the supports of~$\rho$ and~$\rho\circ\Theta_w$, and we set
$$
\Delta(v,w)=\frac{\chi_r^+(\Theta_w(v),w)\rho(\Theta_w(v))}{\det\bigl(I+(D\tilde\varPhi)(\Theta_w(v)+w)\bigr)}-\chi_r^+(v,w)\rho(v). 
$$
Since $v\mapsto\Theta_w(v)$ is a $2$-Lipschitz map satisfying the inequality $|\Theta_w(v)-v|\le\varkappa$ for all $v\in\EE$, we have $|\Delta(v,w)|\le C_2\varkappa(1+r^{-1})$. Substituting this into~\eqref{2.020}, we derive 
\begin{equation*}
\delta(f)\le C_3\varkappa (1+r^{-1})+C_1r^\gamma. 
\end{equation*}
Choosing $r=\varkappa^{1/(1+\gamma)}$, we get
\begin{equation} \label{4.7}
\delta(f)\le C_4\varkappa^{\gamma/(1+\gamma)}. 
\end{equation}
A similar argument shows that $\delta(f) \ge  -C_5\varkappa^{\gamma/(1+\gamma)}$. Combining this with~\eqref{4.7}, we arrive at inequality~\eqref{4.011} with $\beta=\frac{\gamma}{1+\gamma}$. This completes the proof of the theorem. 
\end{proof}

\subsection{Approximate right inverse of linear operators with dense image}
\label{ss-right-inverse}
Let~$F$ and~$H$ be separable Hilbert spaces and let $A:F\to H$ be a continuous linear operator. Consider the equation $A\zeta=f$. In general, it does not have a solution, and when it does, the solution may not be unique. The following result shows that, under some additional conditions, one may construct an approximate solution that linearly depends on~$f$. 

\begin{proposition} \label{l2.4}
In addition to the above hypotheses, let the image of~$A$ be dense in~$H$ and let~$V$ be a Banach space compactly embedded into~$H$. Then for any $\e>0$ there is a continuous linear operator $R_\e:H\to F$ with a finite-dimensional range such that 
\begin{equation} \label{2.19}
\|AR_\e f-f\|_H\le \e\|f\|_V\quad\mbox{for any $f\in V$}. 
\end{equation}
\end{proposition}

\begin{proof}
Let us define the operator $G=AA^*:H\to H$. Since the image $\Im(A)$ is dense in~$H$, the kernel of the self-adjoint operator~$G$ is trivial, and therefore the image~$\Im(G)$ is dense in~$H$. Let us recall that the operator~$A^*G^{-1}$ defined on~$\Im(G)$ is called the {\it Moore--Penrose pseudo-inverse\/} and singles out the solution of the least norm for the equation $A\zeta=f$ (when it exists). We shall need the following lemma giving a natural approximation of the right inverse of~$A$. 

\begin{lemma} \label{l2.7}
Let $G:H\to H$ be a non-negative self-adjoint operator. Then the mapping $(0,+\infty)\ni \gamma\mapsto(G+\gamma)^{-1}$ is a well-defined smooth operator function such that $\Delta_f(\gamma):=\|G(G+\gamma)^{-1}f-f\|^2$ decreases with $\gamma$ for every $f\in H$. Moreover, the norms of the operators $G(G+\gamma)^{-1}$ and $(G+\gamma)^{-1}$ are bounded, respectively, by~$1$ and~$\gamma^{-1}$, and if~$G$ has dense image, then 
	\begin{equation} \label{2.034}
		\lim_{\gamma\to0}\|G(G+\gamma)^{-1}f-f\|=0
		\quad\mbox{for any $f\in H$}. 
	\end{equation}
\end{lemma}

We now construct~$R_\e$ by truncating~$A^*(G+\gamma)^{-1}$ to ensure that the image is finite-dimensional. Namely, choosing an orthonormal basis~$\{f_j\}$ in~$F$ and denoting by~${\mathsf P}_M$ the orthogonal projection to the vector space spanned by the first~$M$ vectors, we define $R_{\gamma,M}={\mathsf P}_MA^*(G+\gamma)^{-1}$. We now fix any $\e>0$. By Lemma~\ref{l2.7}, for any $f\in H$ there is $\gamma_\e(f)>0$ such that
$$
\|A A^*(G+\gamma)^{-1}f-f\|\le\frac{\e}{3} 
\quad\mbox{for $0<\gamma\le\gamma_\e(f)$}. 
$$
Since the norm of the operator $A A^*(G+\gamma)^{-1}$ is bounded by~$1$, for any $f\in H$ there is $\delta_\e(f)>0$ such that 
\begin{equation} \label{2.0024}
\|A A^*(G+\gamma)^{-1}g-g\|\le\frac{2\e}{3} 
\quad\mbox{for $0<\gamma\le\gamma_\e(f)$, $\|g-f\|\le\delta_\e(f)$}. 	
\end{equation}
The open balls $\{O_f:=\dot B_H(f,\delta_\e(f))\}_{f\in H}$ form a covering of the compact set $B_V(1)\subset H$. Choosing a finite sub-covering $\{O_{f_j},1\le j\le m\}$ and setting $\gamma_\e:=\min\{\gamma_\e(f_j),1\le j\le m\}$, we derive from~\eqref{2.0024} that 
$$
\|A A^*(G+\gamma_\e)^{-1}f-f\|\le\frac{2\e}{3} 
\quad\mbox{for $f\in B_V(1)$}. 	
$$
Since the sequence $\{AR_{\gamma_\e,M}\}_{M\ge1}$ converges to $A A^*(G+\gamma_\e)^{-1}$ as $M\to\infty$ in the strong operator topology and the convergence is uniform on compact subsets, we can find $M_\e\ge1$ such that 
$$
\|AR_{\gamma_\e,M_\e}f-f\|_H\le\e\quad\mbox{for $f\in B_V(1)$}.
$$
By homogeneity, this implies~\eqref{2.19} with $R_\e=R_{\gamma_\e,M_\e}$. 
\end{proof}

\begin{proof}[Proof of Lemma~\ref{l2.7}]
Since $G\ge0$, it follows that $(G+\gamma I)^{-1}$ is well defined and smooth in~$\gamma>0$. When proving that~$\Delta_f$ decreases with~$\gamma$, we can assume, by the spectral theorem, that~$G$ acts in a Lebesgue space $L^2(\XXXX,\lambda)$ as the  multiplication by a bounded non-negative function~$a(x)$. In this case, 
$$
\Delta_f(\gamma)=\|G(G+\gamma I)^{-1}f-f\|^2=\int_\XXXX \frac{\gamma^2|f(x)|^2}{(a(x)+\gamma)^2}\,\lambda(\dd x).
$$
It remains to note that the integrand is an increasing function of~$\gamma>0$. 

The above representation of the operator $G$ readily implies the assertions concerning the norms of $G(G+\gamma)^{-1}$ and $(G+\gamma)^{-1}$. To prove~\eqref{2.034} for operators with a dense image, we first note that 
	\begin{equation} \label{2.039}
	G(G+\gamma)^{-1}-I=-\gamma(G+\gamma)^{-1}.	
	\end{equation} 
	Since the norm of the operator $\gamma(G+\gamma)^{-1}$ is bounded by~$1$, it suffices to prove that the right-hand side of~\eqref{2.039} goes to zero for a dense subset of vectors~$f\in H$. Let us take any~$f$ in the (dense) image of~$G$. Then there is $h\in H$ such that $f=Gh$, so that 
$$
\|\gamma(G+\gamma)^{-1}f\|\le \gamma\|h\|+\gamma^2\|(G+\gamma)^{-1}h\|\le 2\gamma \|h\|.
$$
This implies~\eqref{2.034} and completes the proof. 
\end{proof}

In what follows, we shall need a version of Proposition~\ref{l2.4} for the case when the operator~$A$ depends on a parameter and degenerates for some of its values. Namely, let~$X$ be a compact metric space, let~$E$ be a separable Hilbert space, and let~$\ell\in\PP(E)$ be a Borel measure with a compact support~$\KK$. Consider a continuous mapping  $A:X\times E\to\LL(F,H)$ such that $A(u,\cdot):E\to \LL(F,H)$ is analytic for any $u\in X$ (we recall~\eqref{analytic} and the convention for analytic mappings with parameter). As in Proposition~\ref{l2.4}, we denote by~$V$ a Banach space compactly embedded into~$H$. Finally, we fix an orthonormal basis~$\{f_j\}$ in~$F$ and denote by~$F_N$ the vector span of $f_1,\dots,f_N$. A proof of the following result will be given in Section~\ref{s2.3}.

\begin{theorem} \label{p2.5}
In addition to the above hypotheses, let us assume that, for any $u\in X$, there is a set of full measure $\KK^u\subset\KK$ such that the image of the linear operator $A(u,\eta)$ is dense in~$H$ for any $\eta\in\KK^u$. Then,  for any $\bfe=(\e_1,\e_2)\in(0,1)^2$, there is an integer $M_\bfe\ge1$, positive numbers $\nu_{\e_2},C_\bfe$, and a non-negative continuous function $\FFFF_\bfe(u,\eta)$, defined on~$X\times E$ and analytic in~$\eta$, such that the following properties hold. 
\begin{description}
\item[Bound on the measure.]
The $\ell$-measure of the sets
\begin{equation} \label{2.23}
\KK_\bfe^{u}:=\{\eta\in\KK: \FFFF_\bfe(u,\eta)\le\nu_{\e_2}\}
\end{equation} 
satisfies the inequality
\begin{equation}
	\ell(\KK_{\bfe}^{u})\ge 1-\e_1\quad	\mbox{for $u\in X$}.
\label{2.20}
\end{equation}
\item[Right inverse.] 
Let us define the compact set
\begin{equation} \label{2.024}
\DD_\bfe=\{(u,\eta)\in X\times \KK: \FFFF_\bfe(u,\eta)\le2\nu_{\e_2}\}. 
\end{equation} 
Then there is a continuous mapping $R_\bfe:\DD_{\bfe}\to\LL(H,F)$ such that
\begin{gather}
\Im\bigl(R_\bfe(u,\eta)\bigr)\subset F_{M_\bfe}, \quad 
\|R_\bfe(u,\eta)\|_{\LL(H,F)}\le C_\bfe
\quad\mbox{for $(u,\eta)\in \DD_\bfe$},
\label{2.21}\\
\|A(u,\eta)R_\bfe(u,\eta)f-f\|_H\le \e_2\|f\|_V
\quad\mbox{for $(u,\eta)\in \DD_\bfe$, $f\in V$}. 
\label{2.22}
\end{gather}
\end{description}
\end{theorem}

\subsection{Proof of Proposition~\ref{P:2.5}}
\label{ss-proof-of-P2.5}

We first outline the main idea. We seek a vector $\zeta\in E$ depending on~$u$, $u'$, and~$\eta$ such that 
\begin{equation} \label{3.16}
\|S(u,\eta)-S(u',\eta+\zeta)\|\le \theta\|u-u'\|,
\end{equation}
provided that $\|u-u'\|$ is sufficiently small. Since~$S$ is a $C^2$-function of its arguments, we can write
\begin{equation} \label{3.17}
S(u',\eta+\zeta)=S(u,\eta)+(D_uS)(u,\eta)(u'-u)+(D_\eta S)(u,\eta)\zeta+r(u,u',\eta,\zeta),
\end{equation}
where~$r$ is a remainder term of order $\|u-u'\|^2+\|\zeta\|^2$. We see that a good choice of~$\zeta$ would be defining it as a solution of the equation
$$
(D_\eta S)(u,\eta)\zeta=-(D_uS)(u,\eta)(u'-u). 
$$
This equation is not necessarily solvable. However, by Hypothesis~\hyperlink{H3}{(H$_3$)}, the image of $(D_\eta S)(u,\eta)$ is dense in~$H$. By Theorem~\ref{p2.5}, for any $\e>0$, there exists an approximate right inverse $R_\e(u,\eta):H\to E$ such that
\begin{equation} \label{3.18}
\|(D_\eta S)(u,\eta)R_\e(u,\eta)f-f\|\le \e\|f\|_V
\quad\mbox{for any $f\in V$},
\end{equation}
where $V$ is the Banach space in~\hyperlink{H1}{(H$_1$)}. The mapping $(D_uS)(u,\eta)$ is continuous from~$H$ to~$V$, so that in~\eqref{3.18} we can take $f=-(D_uS)(u,\eta)(u'-u)$. We shall show that, for a sufficiently small $\e=\e(\theta)>0$, the mapping 
\begin{equation} \label{3.020}
\varPhi^{u,u'}(\eta)=-R_\e(u,\eta)(D_uS)(u,\eta)(u'-u)	
\end{equation}
satisfies all required properties. Let us turn to an accurate proof, which is divided into four steps. 

\medskip
{\it Step~1. Construction of~$\varPhi$}. 
Let us fix a small parameter $\e>0$ that will be chosen later. Let $\nu_\e$, $\KK_\e^{u}$, $\DD_\e$, $\FFFF_\e$, and~$R_\e$ be the objects described in Theorem~\ref{p2.5} with $\bfe=(\e,\e)$, $E=F$, and $A(u,\eta)=(D_\eta S)(u,\eta)$. We now construct the sets $\KK^{u,\sigma,\theta}\subset\KK^u$ on which~$\varPhi^{u,u'}(\cdot)$ will be defined by~\eqref{3.020}. The main point is to choose them in such a way that inequality~\eqref{4.2} is true with $\KK_1=\KK^{u,\sigma,\theta}$ (so that we can apply Theorem~\ref{T:4.1} to prove~\eqref{3.8}). 

Let us recall that $\KK_\e^{u}$, $\DD_\e$, and $R_\e$ have the form
\begin{align}
	\KK_\e^{u}&=\{\eta\in\KK:\FFFF_\e(u,\eta)\le\nu_\e\}
	\quad\mbox{for $u\in X$},\label{KE}\\
	\DD_\e&=\{(u,\eta)\in X\times \KK:\FFFF_\e(u,\eta)\le2\nu_\e\}, 
	\label{DE}\\
	R_\e(u,\eta)&={\mathsf P}_MR_{\gamma}(u,\eta)
	\quad\mbox{for $(u,\eta)\in\DD_\e$},\label{RE}
\end{align}
where the operator $R_\gamma$ is defined in~\eqref{approxinverse}, and the number $\gamma=\gamma(\e)>0$ and the integer~$M=M(\e)\ge1$ are chosen appropriately. Let us denote by~$\DD_\e'$ the projection of~$\DD_\e$ to the space $X\times E_{M}^\bot$. In other words, 
$$
\DD_\e'=\{(u,w)\in X\times E_{M}^\bot:\mbox{there is $v\in E_{M}$ such that $w+v\in\KK_\e^{u}$}\}. 
$$
Since~$\DD_\e$ is compact, so is its projection~$\DD_\e'$. We shall need the following lemma established at the end of this subsection.

\begin{lemma} \label{l3.3}
	The set $\DD_\e'$ can be represented as the disjoint union of finitely many sets $\DD_1',\dots,\DD_m'$ such that the following property holds: for any integer $l\in[1,m]$ there is $\nu_l\in(\nu_\e,3\nu_\e/2)$ such that, for any $(u,w)\in\overline{\DD_l'}$, the function $v\mapsto\FFFF_\e(u,w+v)-\nu_l$ is not identically zero. 
\end{lemma}

We now set
\begin{equation} \label{3.021}
\KK^{u,\sigma,\theta}=\bigcup_{l=1}^m
\bigl\{\eta\in \KK:\eta=v+w, (u,w)\in\DD_l', \FFFF_\e(u,w+v)\le\nu_l\bigr\}	
\end{equation}
and introduce a function $\varPhi:X\times H\times E\to E$ as follows: $\varPhi^{u,u'}(\eta)$ is defined by~\eqref{3.020} if $\eta\in\KK^{u,\sigma,\theta}$, and $\varPhi^{u,u'}(\eta)=0$ otherwise. This is a measurable mapping with range in~$E_{M}$. We claim that~\eqref{3.7}--\eqref{3.9} hold for an appropriate choice of~$\e$.

\medskip
{\it Step~2. Proof of~\eqref{3.7}}. 
Up to now, the parameter~$\e>0$ was arbitrary. Let us choose it so small that $\e\le \sigma$. In view of~\eqref{2.20}, the required inequality will be established if we prove that  $\KK^{u,\sigma,\theta}\supset	\KK_\e^u$. But this relation immediately follows from~\eqref{KE} and~\eqref{3.021} since $\nu_l>\nu_\e$. 

\medskip
{\it Step~3. Proof of~\eqref{3.9}}. We first note that~\eqref{DE} and~\eqref{3.021} imply the inclusion $\KK^{u,\sigma,\theta}\subset \{\eta\in\KK:(u,\eta)\in\DD_\e\}$. In view of relations~\eqref{3.17} and~\eqref{3.18}, in which $\zeta=\varPhi^{u,u'}(\eta)$ and $f=-(D_uS)(u,\eta)(u'-u)$, for $u\in X$, $u'\in H$, and $\eta\in\KK^{u,\sigma,\theta}$, we have
\begin{multline} \label{3.19} 
\|S(u,\eta)-S(u',\eta+\varPhi^{u,u'}(\eta))\|\\
\le\e\,\|(D_uS)(u,\eta)(u'-u)\|_V+\|r(u,u',\eta,\varPhi^{u,u'}(\eta))\|. 
\end{multline}
Hypothesis~\hyperlink{H1}{(H$_1$)} implies that 
\begin{equation} \label{3.20}
\|(D_uS)(u,\eta)(u'-u)\|_V\le C_1\|u'-u\|,
\end{equation}
where we denote by~$C_i$ positive numbers not depending on~$\e$, $u$, $u'$, and~$\eta$. Furthermore, since~$S$ is a $C^2$ function whose second derivative is bounded on bounded subsets, for any~$\rho>0$ and $\|u\|+\|u'\|+\|\eta\|_E+\|\zeta\|_E\le\rho$, we have
\begin{equation} \label{3.21}
\|r(u,u',\eta,\zeta)\|\le C_2\bigl(\|u-u'\|^2+\|\zeta\|_E^2\bigr).
\end{equation}
Recalling the definition of~$\varPhi^{u,u'}$ (see~\eqref{3.020}) and using the inequality  in~\eqref{2.21}, we derive
\begin{equation} \label{3.23}
\|\varPhi^{u,u'}(\eta)\|\le C_3(\e)\|u-u'\|\quad
\mbox{for $u\in X$, $\eta\in \KK^{u,\sigma,\theta}$, $u'\in H$}.
\end{equation}
Substituting~\eqref{3.20}--\eqref{3.23} into~\eqref{3.19}, we obtain 
$$
\|S(u,\eta)-S(u',\eta+\varPhi^{u,u'}(\eta))\|
\le \bigl(C_4\e+C_5(\e)\|u-u'\|\bigr)\|u-u'\|.
$$
For a given $\theta\in(0,1)$, we can choose first~$\e$ and then~$\delta$ so  that $(C_4\e+C_5(\e)\delta)\le\theta$. We thus obtain~\eqref{3.9} for $(u,u')\in D_\delta$ and~$\eta\in\KK^{u,\sigma,\theta}$. 

\medskip
{\it Step~4. Proof of~\eqref{3.8}}. 
We shall use Theorem~\ref{T:4.1} with $\KK_1=\KK^{u,\sigma,\theta}$. Let us fix $(u,u')\in D_\delta$. The mapping $\varPhi^{u,u'}(\cdot):\KK\to E$ is measurable, and its image is contained in~$E_{M}$. By Hypothesis~\hyperlink{H4}{(H$_4$)}, the measure~$\ell$ can be written as the direct product of its projections to~$E_{M}$ and~$E_{M}^\bot$. Inequality~\eqref{3.8} will be established if we show that Properties~(a) and~(b) of Theorem~\ref{T:4.1} are true for $\varPhi^{u,u'}(\cdot)$ with $\varkappa=C\|u-u'\|$, where~$C$ is an absolute constant. 

By construction, $\varPhi^{u,u'}$ vanishes outside~$\KK^{u,\sigma,\theta}$. In view of~\eqref{3.23}, the first inequality in~\eqref{4.1} is satisfied. To prove\footnote{Notice that we cannot apply the mean value theorem, since we do not know if~$\KK^{u,\sigma,\theta}$ is convex.} the second, let us fix any smooth function~$h(t)$ that is equal to~$1$ for $|t|\le3\nu_\e/2$ and vanishes for $|t|\ge2\nu_\e$. Then we can write  
\begin{equation} \label{3.028}
	\varPhi^{u,u'}(\eta)
	=-h\bigl(\FFFF_\e(u,\eta)\bigr)R_\e(u,\eta)(D_uS)(u,\eta)(u'-u)
	\quad\mbox{for $\eta\in\KK^{u,\sigma,\theta}$}.
\end{equation}
The right-hand side of~\eqref{3.028} is obviously  locally Lipschitz in $\eta\in E$, with a Lipschitz constant proportional to~$\|u-u'\|$. It follows that~$\varPhi^{u,u'}$ satisfies the Lipschitz condition on $\KK^{u,\sigma,\theta}$ with a constant of the form~$C \|u-u'\|$. We have thus established~(a). 

\smallskip
To prove~(b), we denote by~$\KK^l$, $l=1,\dots,m$, the sets entering the right-hand side of~\eqref{3.021}. Let us note 
$$
\p_w\KK^{u,\sigma,\theta}=\bigcup_{l=1}^m\p_w\KK^l,
$$ 
so that it suffices to establish~\eqref{4.2} for each of the sets~$\KK^l$. Denoting by~$O_M$ a sufficiently large ball in~$E_M$, it is easy to see that  
\begin{align*}
\p_w\KK^l
&=\{v\in\KK^{u,\sigma,\theta}(w):\FFFF_\e(u,w+v)=\nu_l\}\\
&\subset\{v\in O_M:\FFFF_\e(u,w+v)-\nu_l=0\}.
\end{align*}
Now note that the analytic function $O_M\ni v\mapsto \FFFF_\e(u,w+v)-\nu_l$ does not vanish identically for $(u,w)\in\overline\DD_l$. An estimate for the neighbourhood of the set of zeros for such functions will be given in Corollary~\ref{T:4.3}. Applying the latter to the Lebesgue measure on~$E_{M}$, we conclude that~\eqref{4.2} is valid for~$\p_w\KK^l$. This completes the proof of Proposition~\ref{P:2.5}. 

\begin{proof}[Proof of Lemma~\ref{l3.3}]
	For any $(u,w)\in\DD_\e'$, the function $v\mapsto \FFFF_\e(u,w+v)$ is analytic on the finite-dimensional space~$E_{M}$. Therefore, by Sard's theorem, almost every real number is a regular value for it. Hence, we can find a number $\nu_{u,w}\in(\nu_\e,3\nu_\e/2)$ such that each of the functions $v\mapsto \FFFF_\e(u,w+v)-\nu_{u,w}$ is not identically equal to zero. By continuity, $v\mapsto \FFFF_\e(u',w'+v)-\nu_{u,w}$ will not vanish identically, provided that $(u',w')$ belongs to a non-degenerate closed ball~$B_{u,w}$ centred at~$(u,w)$. Since the corresponding open balls form a covering of the compact set~$\DD_\e'$, we can find a finite sub-covering $B_{u_j,w_j}$, $j=1,\dots,m$. We can now set $\DD_l'=(\DD_\e'\cap B_{u_l,w_l})\setminus \bigl(\,\cup_{j=1}^{l-1}\DD_j'\bigr)$, where the union is empty for $l=1$. 
\end{proof}

\section{Some auxiliary results}
\label{s-auxiliary-results}
In this section, we prove the theorem about the existence of a right inverse for a family of operators with dense image and establish an estimate for the measure of a tubular neighbourhood of a nodal set for holomorphic functions depending on a parameter. The latter was used in the proof of Theorem~\ref{p2.5} and will also be needed in Section~\ref{s6.2}.

\subsection{Proof of Theorem~\ref{p2.5}}
\label{s2.3}
We essentially repeat the proof of Proposition~\ref{l2.4}, following the dependence on the parameters. Namely, we set $G(u,\eta)=A(u,\eta)A(u,\eta)^*$, and given an integer $M\ge1$ and a number $\gamma>0$, define 
\begin{equation} \label{approxinverse}
R_\gamma(u,\eta)=A(u,\eta)^*\bigl(G(u,\eta)+\gamma I\bigr)^{-1}, \quad 
R_{M,\gamma}(u,\eta)={\mathsf P}_MR_\gamma(u,\eta),	
\end{equation}
where ${\mathsf P}_M:F\to F$ denotes the orthogonal projection to~$F_M$. We shall prove that, for any given $\bfe=(\e_1,\e_2)\in(0,1)^2$ and an appropriate choice of~$M$ and~$\gamma$, the operator~$R_{M,\gamma}$ possesses all the required properties. 

The fact that the image of~$R_{M,\gamma}$ is contained in the subspace~$F_M$ follows immediately from the definition. Furthermore, since the norm of the inverse $(G+\gamma I)^{-1}$ is bounded by~$\gamma^{-1}$, we have 
$$
\|R_{M,\gamma}(u,\eta)\|_{\LL(H,F)}\le C_1\gamma^{-1}, \quad 
C_1=\sup_{(u,\eta)\in X\times\KK}\|A(u,\eta)\|_{\LL(F,H)},
$$
so that the inequality in~\eqref{2.21} holds for any fixed~$M$ and~$\gamma$. Thus, we need to construct a function~$\FFFF_\bfe$ and to prove~\eqref{2.20} and~\eqref{2.22}.

Let us show that for any $\gamma>0$ there is an integer $M_{\gamma,\bfe}\ge1$ with the following property: if for some closed subset $\DD_\bfe\subset X\times\KK$ we have
\begin{equation} \label{2.036}
	\sup_{(u,\eta)\in\DD_\bfe}\|A(u,\eta)R_\gamma(u,\eta)f-f\|<\e_2
	\quad\mbox{for $f\in B_V(1)$},
\end{equation}	
then inequality~\eqref{2.22} with $R_\bfe=R_{M,\gamma}$ is true for $M\ge M_{\gamma,\bfe}$. Indeed, let us note that 
$$
\Delta_{M,\gamma}(u,\eta)f:=AR_{\gamma}f-AR_{M,\gamma}f
=A(u,\eta)(I-{\mathsf P}_M)R_{\gamma}(u,\eta).
$$
Since the mapping $(u,\eta,f)\mapsto R_\gamma(u,\eta)f$ acting from $X\times\KK\times H$ to~$H$ is continuous, the image of the compact set $\DD_\bfe\times B_V(1)$ is compact. The convergence of ${\mathsf P}_M$ to~$I$ in the strong operator topology implies that 
$$
\sup_{(u,\eta,f)\in \DD_\bfe\times B_V(1)}\|\Delta_{M,\gamma}(u,\eta)f\|\to0
\quad\mbox{as $M\to\infty$}. 
$$
Combining this with~\eqref{2.036}, we see that~\eqref{2.22} is true for $M\ge M_\gamma$. We thus need to construct a function~$\FFFF_\bfe$ and numbers~$ \gamma,\nu_{\e_2}>0$ for which~\eqref{2.20} and~\eqref{2.036} hold.

To this end, note that, by Lemma~\ref{l2.7}, for any $u\in X$, $\eta\in\KK^u$, and $f\in H$,
	\begin{equation} \label{2.0034}
		\lim_{\gamma\to0}\|A(u,\eta)R_{\gamma}(u,\eta)f-f\|=0. 
	\end{equation}
Let $\{f_j,1\le j\le N\}\subset H$ be an $(\e_2/4)$-net for the compact set $B_V(1)\subset H$. For $\gamma>0$, we define the continuous function
\begin{equation} \label{FG}
\FFFF_\gamma(u,\eta)=\sum_{j=1}^N\|A(u,\eta)R_\gamma(u,\eta)f_j-f_j\|^2	
\end{equation}
and notice that it is analytic in~$\eta$ and decreases with~$\gamma>0$ for each~$(u,\eta)$ in view of Lemma~\ref{l2.7}. Furthermore, setting $\nu_{\e_2}=\e_2^2/32$, we see that if $\FFFF_\gamma(u,\eta)\le2\nu_{\e_2}$, then for any $f\in B_V(1)$ we have 
\begin{align} 
	\|A(u,\eta)R_\gamma(u,\eta)f-f\|
	&\le \min_{1\le j\le N}
	\bigl(\|AR_\gamma(f-f_j)\|+\|f-f_j\|\bigr)+\FFFF_\gamma(u,\eta)^{1/2} \notag\\
	&\le 3\e_2/4,\label{2.037}
\end{align}
where we used the fact that the norm of the operator $G(G+\gamma I)^{-1}$ is bounded by~$1$. Suppose we have established the following lemma.

\begin{lemma} \label{l2.8}
	Under the hypotheses of Theorem~\ref{p2.5}, for any $\nu>0$ we have 
	\begin{equation} \label{UC}
 		\ell\bigl(\{\eta\in\KK:\FFFF_\gamma(u,\eta)<\nu\}\bigr)\to1\quad
 		\mbox{uniformly in $u\in X$ as $\gamma\to0^+$}. 
	\end{equation}
\end{lemma}

Applying this result with $\nu=\nu_{\e_2}$, we can find $\gamma=\gamma(\bfe)>0$ such that~\eqref{2.20} holds for the set~$\KK_\bfe^u$ defined by relation~\eqref{2.23} with $\FFFF_\bfe=\FFFF_{\gamma(\bfe)}$. It remains to note that inequality~\eqref{2.036} follows immediately from~\eqref{2.037}. This completes the proof of Theorem~\ref{p2.5}.

\begin{proof}[Proof of Lemma~\ref{l2.8}]
By assumption, the image of~$A(u,\eta)$ is dense for any $u\in X$ and $\eta\in\KK^u$. It follows from~\eqref{2.0034} that
\begin{equation*}
\FFFF_\gamma(u,\eta)\to0\quad\mbox{as $\gamma\to0$ for each $\eta\in\KK^u$}. 
\end{equation*}
Since $\ell(\KK^u)=1$ and the almost sure convergence implies convergence in probability, we see that the family of functions $\{m_\gamma:X\to[0,1]\}_{\gamma\in(0,1)}$ defined by 
\begin{equation} \label{2.038}
m_\gamma(u)=\ell\bigl(\{\eta\in\KK:\FFFF_\gamma(u,\eta)<\nu\}\bigr)
\end{equation}
converges to~$1$ as $\gamma\to0^+$ for any $u\in X$ and $\nu>0$. We need to prove that this convergence is uniform. Suppose we have established the following two properties:
\begin{description}
	\item[Semicontinuity.] 
	For any $\gamma\in(0,1)$, the function $m_\gamma:X\to[0,1]$ is lower semicontinuous.
	\item[Monotonicity.] 
	For any $\gamma_1\le \gamma_2$ and $u\in X$, we have $m_{\gamma_1}(u)\ge m_{\gamma_2}(u)$. 
\end{description}
In this case, the required uniform convergence follows from Dini's theorem for a sequence of increasing functions, which remains true when the functions are lower semicontinuous; see~\cite[Theorem~2.4.10]{dudley2002}.

Let us prove the semicontinuity. We denote by $\mu_u\in\PP(\R)$ the image of~$\ell$ under the mapping $\FFFF_\gamma(u,\cdot)$. Thus, for any bounded continuous function $g:\R\to\R$, we have 
$$
\int_\R g(r)\mu_u(\dd r)=\int_Eg\bigl(\FFFF_\gamma(u,\eta)\bigr)\ell(\dd\eta). 
$$
It follows from the Lebesgue theorem on dominated convergence that the function $u\mapsto \mu_u$ acting from~$X$ to the space~$\PP(\R)$ endowed with the weak topology is continuous. By the portemanteau theorem (see~\cite[Theorem~11.1.1]{dudley2002}), for any open set $O\subset\R$ the function $u\to \mu_u(O)$ is lower semicontinuous. It remains to note that $m_\gamma(u)=\mu_u(O)$ with $O=(-\infty,\nu)$. 

To prove the monotonicity, it suffices to note that, by Lemma~\ref{l2.7}, the function $\Delta_f(\gamma):=\|G(G+\gamma)^{-1}f-f\|^2$ decreases with~$\gamma$, so that the same is true for~$\FFFF_\gamma(u,\eta)$. The proof of the lemma is complete. 
\end{proof}

\subsection{Measure of a tubular neighbourhood of a nodal set}
\label{s2.4}
As before, let~$X$ be a compact metric space, let~$E$ be a separable Hilbert space, and let $\ell\in\PP(E)$ be a probability measure with a compact support~$\KK$. We assume that~$\ell$ is {\it decomposable\/} in the sense that there is an orthonormal basis~$\{e_j\}$ in~$E$ such that~$\ell$ can be represented as the direct of product its projections to the one-dimensional subspaces spanned by~$e_j$. Let us fix a bounded convex open set~$\OO\supset\KK$ and consider a continuous function $\FFFF:X\times \OO\to \R $ such that $\FFFF(u,\cdot): \OO\to \R $ is analytic and not identically zero for any $u\in X$.

\begin{lemma} \label{l2.6}
In addition to the above hypotheses, let us assume that the one-dimensional projections of~$\ell$ to the spaces~$\lspan(e_j)$ possess continuous densities with respect to the Lebesgue measure. Then there are positive numbers~$C$ and~$c$ such that, for any $u\in X$ and~$r\in[0,1]$,
\begin{equation} \label{2.25}
\ell\bigl(\{\eta\in\OO:|\FFFF(u,\eta)|\le r\}\bigr)\le C\,r^c.
\end{equation}
\end{lemma}

As an immediate consequence of this lemma, we obtain the following estimate for the measure of a tubular neighbourhood of the nodal set for an analytic function. Namely, for any $u\in X$, we denote
$$
\NN(u) = \{\eta\in \OO: \FFFF(u,\eta) =0\}.
$$

\begin{corollary} \label{T:4.3}
Under the  hypotheses of Lemma~\ref{l2.6}, there are numbers $C,c>0$ such that, for any $u\in X$ and~$r\in[0,1]$,
\begin{equation}\label{4.8}
\ell\bigl(\{\eta\in\OO:\dist(\eta,\NN(u))\le r\}\bigr)\le C\,r^c.
\end{equation}
\end{corollary}

\begin{proof}
We first recall that $\|(D_\eta \FFFF)(u,\eta)\|_{\LL(E)}\le C<\infty$ for $(u,\eta)\in X\times\OO$, in view of our convention concerning analytic functions. By the convexity of~$\OO$, it follows that 
$$
\{\eta\in\OO:\dist(\eta,\NN(u))\le r\}\subset
\{\eta\in\OO:|\FFFF(u,\eta)|\le C'r\},
$$
where $C'>0$ does not depend on $u\in X$ and $r\in[0,1]$. The required result is now implied by~\eqref{2.25}. 
\end{proof}

\begin{proof}[Proof of Lemma~\ref{l2.6}]
It suffices to establish~\eqref{2.25} for $r\le r_0$, for a suitable $r_0>0$. Let us denote by $\ZZ(u,r)$ the set in the left-hand side of~\eqref{2.25} and take any  $u^0\in X$ and $\eta^0\in \KK$. If $\FFFF(u^0,\eta^0)\ne0$, then for sufficiently small balls $X\supset O_{u^0}\ni u^0$ and~$E\supset O_{\eta^0}\ni\eta^0$, we have 
\begin{equation}\label{4.10}
|\FFFF(u,\eta)|> \sigma(u^0, \eta^0)>0 \quad  
\text{for $u\in O_{u^0}$, $\eta\in O_{\eta^0}$}. 
\end{equation}

Now assume that $\FFFF(u^0,\eta^0) =0$. Since~$\FFFF(u^0,\cdot)$ is analytic and not identically zero, we can find a vector $e\in E$ belonging to the span of finitely many first vectors of the basis~$\{e_j\}$ such that  the function $t\mapsto f_{u^0,\eta^0}(t):=\FFFF(u^0,\eta^0+te)$ does not vanish identically. It follows that $f_{u^0,\eta^0}^{(m)}(0)\ne0$
for some $m=m(u^0,\eta^0)$, so that we can find a number $\delta>0$ and open balls $O_{u^0}\ni u^0$ and~$O_{\eta^0}\ni\eta^0$ such that 
$$
|f_{u,\eta}^{(m)}(t)| \ge \gamma(u^0, \eta^0)>0\quad  
\text{for $u\in O_{u^0}$, $\eta\in O_{\eta^0}$, $|t|\le\delta$}. 
$$
This implies that (see Lemma~2 in~\cite{bakhtin-1986} or  Lemma~B.1 in~\cite{eliasson-2002})
\begin{equation}\label{4.11}
\Leb\bigl(\{t\in [-\delta,\delta]: |f(t)| \le r\}\bigr) 
\le C(u^0, \eta^0)\,r^{1/m}
\end{equation}
for $u\in O_{u^0}$, $\eta\in O_{\eta^0}$, and $r\in[0,1]$. 
Since the projection of~$\ell$ to any subspace spanned by finitely many vectors of the basis~$\{e_j\}$ possesses a bounded density with respect to the Lebesgue measure, applying the Fubuni theorem, we conclude from~\eqref{4.11} that
\begin{equation}\label{4.12}
\ell\bigl(\ZZ(u,r) \cap O_{\eta^0}\bigr) \le C(u^0,\eta^0)\,r^{c(u^0,\eta^0)}.
\end{equation}

Now let us choose a finite system of domains $O_{u^j} \times O_{\eta^j}$ which covers the compact set $X\times \KK$ such that~\eqref{4.10} or~\eqref{4.12} hold with $u^0=u^j$ and $\eta^0=\eta^j$. Denoting by~$r_0$ the minimum of all involved constants~$\sigma(u^j,\eta^j)$, by~$C$ the maximum of the constants $C(u^j, \eta^j)$, and by~$c$ the minimum of all exponents $c(u^j, \eta^j)$, we see that~\eqref{2.25} follows from~\eqref{4.10} and~\eqref{4.12}. 
 \end{proof}

In Section~\ref{s6.2}, we shall need the following particular case of Lem\-ma~\ref{l2.6} when~$E$ is a finite-dimensional space and~$X$ is a singleton.

\begin{corollary} \label{r2.9}
	 Let~$\FFFF:\R^n\to\R$ be a non-zero analytic function and let $\xi_1,\dots,\xi_n$ be independent random variables whose joint law possesses a density (with respect to the Lebesgue measure) that is bounded by $\lambda\rho$, where $\lambda>0$ is a number, and~$\rho$ is a continuous function with compact support. Then 
	\begin{equation}
		\IP\{|\FFFF(\xi_1,\dots,\xi_n)|\le r\}\le C(\rho)\,r^c\lambda \quad\mbox{for any $r\in[0,1]$},
	\end{equation}
	where $C(\rho)$ and $c$ are positive numbers depending on~$\FFFF$, but not on~$\lambda$ and the random variables $\xi_1,\dots,\xi_n$. 
\end{corollary}

\section{Applications}
\label{s4} 

\subsection{Two-dimensional Navier--Stokes system}
\label{s4.1}
We consider the Navier--Stokes system on the torus $\T_a^2=\R^2/(2\pi a_1)\Z\oplus (2\pi a_2)\Z$, written in the form~\eqref{1.7} after projecting to the space~$H$ of divergence-free square-integrable vector fields on~$\T_a^2$ with zero mean value. The random forcing is assumed to be of the form~\eqref{1.08}, where $\{\eta_k\}$ is a sequence of i.i.d.\ random variables in $E:=L^2(J,\HH)$ for some Hilbert space~$\HH$ embedded into~$H$. Let $S(u_0,\zeta)$ be the time-$1$ resolving operator for~\eqref{1.7}. That is, $S$ maps $H\times E$ to~$H$ and takes $(u_0,\eta)$ to~$u(1)$, where~$u(t)$ stands for the solution of~\eqref{1.7} satisfying the initial condition $u(0)=u_0$. Let us denote by $H^s(\T_a^2,\R^2)$ the usual Sobolev space of order $s\in\Z$ and by~$H^s$ its intersection with~$H$. Standard dissipativity and regularisation results enable one to prove that, for any compact subset $\KK\subset E$, there is a compact absorbing set~$X\subset H$, bounded and closed in~$H^2$, such that $S(X\times \KK)\subset X$; see~\cite[Section~2]{KS-cmp2001} or~\cite[Section~4.1]{shirikyan-asens2015}. The restrictions~$\{u_k\}$ of a solution~$u(t)$ for~\eqref{1.7}  to integer times satisfy relation~\eqref{1.1}. Thus, if~$\{\eta_k\}$ is a sequence of i.i.d.\ random variables in~$E$ whose law~$\ell$ has a compact support~$\KK\subset E$, then~\eqref{1.7} defines a homogeneous Markov process~$(u_k,\IP_u)$ in~$H$ for which~$X$ is a closed invariant subset. Our aim is to study the large-time asymptotics of~$(u_k,\IP_u)$.

From now on, we assume that the measure $\ell=\DD(\eta_k)$ has a compact support $\KK\subset E$, which contains the origin, and satisfies\footnote{We emphasise that~\hyperlink{H4}{(H$_4$)} can be regarded as a condition on the law~$\ell$ of the random variables~$\eta_k$, since any other random variable with law~$\ell$ has the same structure as~$\eta_k$.} the decomposability  hypothesis~\hyperlink{H4}{(H$_4$)}. To formulate the main result, we need the concepts of an {\it observable measure\/}\footnote{The concept of observability is widely used in the control theory and means, roughly speaking, that if a functional of a non-zero solution of a homogeneous linear differential equation vanishes identically in time, then it must be zero. In Definition~\ref{d6.1}, we have a similar property for functions: the left-hand side of~\eqref{6.2} defines an affine function, and if it vanishes on~$\zeta$, then it must be zero.} and a {\it saturating subspace\/}. 

\begin{definition} \label{d6.1}
Let~$\HH$ be a Hilbert space with an inner product~$(\cdot,\cdot)_\HH$ and $\{\varphi_i\}_{i\in\II}$ an orthonormal basis in~$\HH$.	We say that a function~$\zeta\in L^2(J,\HH)$ is {\it Lipschitz-observable\/} with respect to~$\{\varphi_i\}_{i\in\II}$ if for any Lipschitz-continuous functions $a_i:J\to\R$, $i\in\II$ and any continuous function $b:J\to\R$ such that 
	\begin{equation} \label{6.1}
		\sum_{i\in\II}\|a_i\|_{C(J)}^2<\infty
	\end{equation}
	the equality\footnote{The convergence in~$L^1(J)$ of the series in~\eqref{6.2} (when~$\II$ is infinite) follows easily from~\eqref{6.1}.}
	\begin{equation} \label{6.2}
		\sum_{i\in\II}a_i(t)(\zeta(t),\varphi_i)_{\HH}-b(t)=0\quad\mbox{in $L^1(J)$}
	\end{equation}
	implies that $a_i$, $i\in\II$, and~$b$ vanish identically. 
	
	A probability measure~$\ell$ on $L^2(J,\HH)$ is said to be {\it Lipschitz-observable\/} with respect to~$\{\varphi_i\}$  if $\ell$-almost every trajectory $\eta\in L^2(J,\HH)$ is Lipschitz-observable with respect to~$\{\varphi_i\}$. 
\end{definition}

We now fix a finite- or infinite-dimensional Hilbert space $\HH\subset H^3$ and assume that~$\ell$ satisfies the following hypothesis, which is not very restrictive in view of the examples given in Section~\ref{s6}. In particular, as we show in Section~\ref{s6.2}, it holds for the coloured noise~\eqref{0.2}, \eqref{0.3}. 

\begin{description}
\item[Observability.]
{\sl
The measure~$\ell$ is Lipschitz-observable with respect to an orthonormal basis~$\{\varphi_i\}$ of the space~$\HH$.}  
\end{description}

\begin{remark} \label{rfd-obs}
	If $\dim\HH<\infty$, then property of Lipschitz-observability does not depend on the basis. Indeed,  if we change $\{\varphi_i\}_{i\in\II}$ to another orthogonal basis in~$\HH$, then the vector-function $a(t)=(a_1(t), \dots a_N(t))$ will be replaced by $a'(t)=Ua(t)$, where $U:\R^N\to\R^N$ is an orthogonal transformation. The function $a'(t)$ vanishes identically in~$t$ if and only if $a(t)$ does, and the  components of $a'$ are Lipschitz-continuous if and only if those of~$a$ are. This implies the required assertion. 
	
	We also note that, in the case of a finite $\II$, the Lipschitz-observability  means that the elements $\{(\zeta(t),\varphi_i)\}$ of the quotient space $ L^2(J)/C(J)$ are linearly independent over the ring of Lipschitz-continuous functions on $J$.
\end{remark}

Given a subspace $\HH\subset H^3$, we define a non-decreasing sequence of closed subspaces $\HH_k \subset H^3$, $k\ge0$, by the following rule:
\begin{itemize}
	\item $\HH_0$ coincides with the closure of~$\HH$ in~$H^3$;
	\item if~$\HH_k$ is already defined, then~$\HH_{k+1}$ is the closure in $H^3$ of the linear space, formed by all vectors of the form
\begin{equation} \label{eta1}
	\eta+\sum_{l=1}^nQ(\zeta_l,\xi_l),\quad n=1,2, \dots,
\end{equation}
where $Q(\zeta,\xi)=\Pi(\langle \zeta, \nabla\rangle \xi+\langle \xi,\nabla\rangle \zeta)$, and for $l=1,\dots,n$ the vectors $\eta,\eta_l\in\HH_k$ and $\xi_l\in\HH$ are such that $Q(\zeta_l,\xi_l)\in H^3$.
\end{itemize} 

\begin{definition}
	A subspace $\HH\subset H^3$ is said to be {\it saturating\/} if the union of~$\{\HH_k\}_{k\ge0}$ is dense in~$H$.\footnote{We emphasise that condition of saturation  depends on the parameters~$a_1$ and~$a_2$ of the torus and the quadratic term~$Q$, but not on the viscosity~$\nu$.}
\end{definition}

\begin{theorem} \label{t-mixingNS}
	Let $\HH\subset H^3$ be a saturating subspace that is a Hilbert space with an orthonormal basis~$\{\varphi_i\}_{i\in \II}$ satisfying 
	\begin{equation} \label{convphi}
		\sum_{i\in\II}\|\varphi_i\|_1^2<\infty, 
	\end{equation}
	and let~$\eta$ be a random process of the form~\eqref{1.08} whose law is decomposable and satisfies the observability hypothesis. Assume, in addition, that the support of the law~$\ell$ of~$\eta_k$ contains the origin. Then, for any $\nu>0$, the Markov process~$(u_k,\IP_u)$ has a unique stationary measure~$\mu_\nu\in\PP(X)$, which satisfies inequality~\eqref{1.6}. 
\end{theorem}

\begin{remark} \label{r-NSglobal}
	Even though the convergence to the stationary measure  is formulated in Theorem~\ref{t-mixingNS} for initial functions supported by~$X$, the exponential convergence to $\mu_\nu$ holds for any initial measure with a finite first moment. Namely, there are positive numbers~$C$ and~$\gamma$ such that, for any initial measure $\lambda\in\PP(H)$ satisfying the condition $\mmmm(\lambda):=\int_H\|u\|\lambda(\dd u)<\infty$, we have
	\begin{equation} \label{mixing-global}
		\|\PPPP_k^*\lambda-\mu_\nu\|_L^*\le C_1e^{-\gamma_1 k}\bigl(1+\mmmm(\lambda)\bigr)\quad\mbox{for $k\ge0$}.
	\end{equation}
	The validity of~\eqref{mixing-global} can be derived easily from~\eqref{1.6} using the dissipativity and regularisation properties of the Navier--Stokes dynamics (e.g., see Section~3.2.4 in~\cite{KS-book}), and we shall not dwell on it. 
	
	Furthermore, denoting by~$\mu_\nu(t)$ the law of a solution for~\eqref{1.7} issued from an initial condition~$u_0$ that is independent of~$\eta$ and is distributed as~$\mu_\nu$, we note that $\mu_\nu(t+1)=\mu_\nu(t)$ for any $t\ge0$, and for any other solution $u(t)$ of~\eqref{1.7} we have 
	$$
	\bigl\|\DD(u(t))-\mu_\nu(t)\bigr\|_L^*\le C_2e^{-\gamma_2 t}\bigl(1+\mmmm(u(0))\bigr), \quad t\ge0.
	$$ 
Both these claims follow immediately from the fact that, for any integer $k\ge0$ and any $s\in(0,1)$, the function~$u(k+s)$ with  can be written as a Lipschitz-continuous function of the independent random variables $u(k)$ and~$\eta_{k+1}$ (cf.\ Section~3 in~\cite{shirikyan-2018}). 
\end{remark}

The coloured noises~\eqref{0.3} with $c_j = Cj^{-q}$ for some $q>1$ are decomposable, and as we show below, they satisfy the observability hypothesis. Therefore the result above implies the assertion of Theorem~\ref{MT}, as well as inequality~\eqref{mixing-global}. 

\begin{proof}
	In view of Theorem~\ref{t1.1}, it suffices to check Hypotheses~\hyperlink{H1}{(H$_1$)}--\hyperlink{H4}{(H$_4$)}, where~$H$ is defined by~\eqref{spaceH}, and $E=L^2(J,\HH)$. The regularity condition is a consequence of some well-known properties of the resolving semigroup for the 2D Navier--Stokes system (see Section~6 in~\cite[Chapter~1]{BV1992} and~\cite{kuksin-1982}); for the reader's convenience, we presnt a complete proof of analyticity of~$S$ in Section~\ref{s-resolvingNS}.  Condition~\hyperlink{H2}{(H$_2$)} with $\hat\eta=0$ and $\hat u=0$ follows immediately from the dissipativity of the homogeneous Navier--Stokes system. The decomposability condition~\hyperlink{H4}{(H$_4$)} for~$\ell$ holds by assumption. Thus, we only need to check~\hyperlink{H3}{(H$_3$)}. The proof is divided into three steps. 
	
	\smallskip
	{\it Step~1. Reduction}. 
	We recall that $X\subset H^2$ stands for a closed bounded set that is invariant and absorbing under the dynamics of Eq.~\eqref{1.7} restricted to the integer lattice. Given $u\in X$ and $\eta\in E$, let $\tilde u\in L^2(J,H^3)\cap W^{1,2}(J,H^1)$ be the solution of Eq.~\eqref{1.7} with the initial condition~$u$ at time $t=0$ and let $R^{\tilde u}(t,s):H\to H$ (with $0\le s\le t\le 1$) be the two-parameter process solving the linearised problem 
	\begin{equation} \label{lNS}
		\dot v+\nu Lv+Q(\tilde u(t)) v=0, \quad v(s)=v_0,
	\end{equation}
	where $Q(\tilde u) v=Q(\tilde u,v)$ and $v_0\in H$. In other words, $R^{\tilde u}(t,s)$ takes~$v_0$ to~$v(t)$, where $v:[s,1]\to H$ is the solution of~\eqref{lNS}. Denote by~$A(u,\eta)$ the operator $D_\eta S(u,\eta)$. This is the resolving operator for the non-homogeneous equation~\eqref{lNS} (that is, zero in the right-hand side should be replaced by~$\zeta$) with $v_0=0$. It can be written as 
	$$
	A(u,\eta):E\to H, \quad \zeta\mapsto\int_0^1R^{\tilde u}(1,s)\zeta(s)\,\dd s.
	$$ 
	Denoting by~$\LLLL(u,\eta)$ its image, we need to prove that $\LLLL(u,\eta)$ is dense in~$H$ for any fixed~$u\in X$ and $\ell$-a.e.\ $\eta\in\KK$. We claim that the density of~$\LLLL(u,\eta)$ is equivalent to the triviality of the kernel for the operator\,\footnote{The operator~$G^{\tilde u}$ is called the {\it Gramian\/} of the linear control problem associated with~\eqref{lNS}. It is a central object for studying controllability properties and enables one to formulate them in terms of observability of solutions for the dual problem; see Chapter~1 in~\cite{coron2007}.} 
	\begin{equation} \label{gramian}
	G^{\tilde u}:=\int_0^1 R^{\tilde u}(1,t){\mathsf P}_\HH R^{\tilde u}(1,t)^*\dd t,		
	\end{equation}
	where $R^{\tilde u}(1,t)^*:H\to H$ is the adjoint of $R^{\tilde u}(1,t)$, and~${\mathsf P}_\HH:H\to H$ stands for the orthogonal projection to the closure of~$\HH$ in~$H$; cf.\  Theorem~2.5 in~\cite[Part~IV]{zabczyk2008}. Indeed, let us introduce an operator $A_1(u,\eta):L^2(J,H)\to H$ by the relation $A_1(u,\eta)=A(u,\eta){\mathsf P}_\HH$ and note that the image~$\LLLL(u,\eta)$ of the operator~$A(u,\eta)$ is dense in~$H$ if and only if so is the image of~$A_1(u,\eta)$. The latter property is equivalent to the triviality of the kernel of the adjoint operator $A_1(u,\eta)^*:H\to L^2(J,H)$. Since 
	$(A(u,\eta)^*g)(t)={\mathsf P}_\HH R^{\tilde u}(1,t)^*g$ for $g\in H$, we see that $A_1(u,\eta)^*g=0$ for some $g\in H$ if and only if $G^{\tilde u}g=0$.
	
\smallskip
{\it Step~2. Description of $R^{\tilde u}(1,t)^*$}. 
	Together with~\eqref{lNS}, let us consider the dual problem 
	\begin{equation} \label{lNSdual}
	\dot w-\nu Lw-Q^*(\tilde u(t)) w=0, \quad w(1)=w_0,
	\end{equation}
	where $Q^*(\tilde u)$ is the (formal) $L^2$-adjoint of~$Q(\tilde u)$ given by
	$$
	Q^*(\tilde u)w=-\Pi\bigl(\langle\tilde u, \nabla\rangle w+(\nabla \otimes w)\tilde u\bigr),
	$$
	where $\nabla \otimes w$ is a $2\times 2$ matrix with the elements $\p_iw_j$.	Problem~\eqref{lNSdual} is a linear backward parabolic equation, and for any $w_0\in H$, it possesses a unique solution $w\in L^2(J,H^1)\cap W^{1,2}(J,H^{-1})$. It is well known, and can be proved easily, that~$w$ can be written as
	\begin{equation} \label{soldual}
		w(t)=R^{\tilde u}(1,t)^*w_0. 
	\end{equation}

\smallskip
{\it Step~3. Inductive argument\/}. 
We wish to prove that $\Ker(G^{\tilde u})=\{0\}$ for any $u\in X$ and $\ell$-a.e.\ $\eta\in E$. Since the union of~$\HH_k$ is dense in~$H$, it suffices to show that, $\ell$-almost surely, any element of $\Ker(G^{\tilde u})$ is orthogonal to~$\HH_k$ for all $k\ge0$. The rest of the proof follows a rather common argument used in the Malliavin calculus to prove the almost sure non-degeneracy of Malliavin's matrix; see the proof of Theorem~2.3.2 in the book~\cite{nualart2006} and the paper~\cite[Sections~3--5]{MP-2006}. The main difference is that we use the observability of the measure~$\ell$, rather than fine properties of functionals of the Wiener process.

Let us fix any realisation $\eta\in\KK$ that is observable with respect to~$\{\varphi_i\}_{i\in\II}$, and suppose that $w_0\in \Ker(G^{\tilde u})$. In view of~\eqref{gramian}, we have 
$$
(G^{\tilde u}w_0,w_0)=\int_0^1\|{\mathsf P}_\HH R^{\tilde u}(1,t)^*w_0\|^2\dd t=0,
$$
whence we see that ${\mathsf P}_\HH R^{\tilde u}(1,t)^*w_0=0$ for any $t\in J$. Thus,  
\begin{equation} \label{xiw0}
	(\zeta,R^{\tilde u}(1,t)^*w_0)=0\quad \mbox{for $t\in J$},
\end{equation}
where $\zeta\in\HH_0$ is arbitrary. Taking $t=1$, we see that $w_0$ must be orthogonal to~$\HH_0$. Suppose we have proved that relation~\eqref{xiw0} holds for all $\zeta\in\HH_k$ with some integer $k\ge0$ (so that taking again $t=1$ we see that~$w_0$ must be orthogonal to~$\HH_k$). By induction, the verification of~\hyperlink{H3}{(H$_3$)} will be completed once we show that~\eqref{xiw0} holds for all $\zeta\in\HH_{k+1}$. 

For any given $\zeta\in\HH_k$, differentiating~\eqref{xiw0} in time and using~\eqref{lNSdual} and~\eqref{soldual}, we derive\footnote{Let us check briefly that the regularity of~$\tilde u$ and~$w$ is sufficient to justify~\eqref{xiw0der}. Indeed, as was mentioned above, the function~$w$ belongs to~$W^{1,2}(J,H^{-1})$, so that the $L^2$-scalar product $(\zeta,w(t))$ can be differentiated for any $\zeta\in H^1$. Resolving the first equation in~\eqref{lNSdual} with respect to $\dot w$ and substituting the resulting expression into the time derivative of~\eqref{xiw0}, we obtain $(\zeta,\nu Lw+Q(\tilde u)^*w)=0$, where the equality is valid for almost every $t\in J$. Since $\tilde u\in L^2(J,H^3)$ and $w\in L^2(J,H^1)$, we can integrate by parts in the last equality, which gives the validity of equality~\eqref{xiw0der} for almost every $t\in J$. Finally, recalling that $\tilde u\in C(J,H^2)$ and $w\in C(J,L^2)$, we see that the right-hand side is a continuous function of time and, hence, vanishes on~$J$.} 
\begin{equation} \label{xiw0der}
	\bigl(\nu L\zeta+Q(\tilde u(t))\zeta,w(t)\bigr)=0\quad\mbox{for $t\in J$},
\end{equation}
where $w$ is given by~\eqref{soldual}. If we set 
\begin{equation} \label{yu}
y(t)=\tilde u(t)-\int_0^t\eta(s)\dd s=\tilde u(t)-\sum_{i\in\II}\varphi_i(x) \int_0^t\eta^i(s)\,\dd s,	
\end{equation}
where $\eta^i(t)=(\eta(t),\varphi_i)_\HH$, 
then~\eqref{xiw0der} can be rewritten as 
$$
\bigl(\nu L\zeta+Q(y(t))\zeta,w(t)\bigr)
+\sum_{i\in\II} \bigl(Q(\varphi_i)\zeta,w(t)\bigr)\int_0^t\eta^i(s)\dd s=0. 
$$
Differentiating this relation in time and setting 
\begin{align*}
	a_i(t)&=\bigl(Q(\varphi_i)\zeta,w(t)\bigr),\\
	b(t)&=\frac{\dd}{\dd t}\bigl(\nu L\zeta+Q(y(t))\zeta,w(t)\bigr)+\sum_{i\in\II} \bigl(Q(\varphi_i)\zeta,\dot w(t)\bigr)\int_0^t\eta^i(s)\dd s,
\end{align*}
we derive 
\begin{equation} \label{obs}
b(t)+\sum_{i\in\II}a_i(t)\eta^i(t)=0\quad\mbox{for all $t\in J$}. 
\end{equation}
The function~$b$ is continuous, and~$\{a_i,i\in\II\}$ are differentiable. Moreover, condition~\eqref{convphi} implies that~\eqref{6.1} holds. By observability of~$\ell$, it follows that $a_i\equiv 0$ for all $i\in\II$, whence $w(1)=w_0$ is orthogonal to all vectors~\eqref{eta1} with $\eta=0$. Since by the above~$w_0$ is orthogonal also to~$\HH_k$, we conclude that it is orthogonal to the  whole space~$\HH_{k+1}$. This completes the proof of Theorem~\ref{t-mixingNS}. 
\end{proof}

In conclusion of this subsection, let us note that Theorem~\ref{t-mixingNS} remains valid in the case when the Navier--Stokes system is studied on a two-dimensional sphere\footnote{A similar remark applies to the case of a rectangular domain with Lions boundary condition; see Section~9 in~\cite{AS-2008}. However, we shall not elaborate on that point since the analysis is similar.}, and the space~$\HH$ contains sufficiently many spherical harmonics. Indeed, all the arguments of the proof are exactly the same as in the case of the torus, except the verification of the saturating property. However, the construction of the successive subspaces~$\HH_k$ coincides with that used by Agrachev and Sarychev in~\cite[Theorems~6.1]{AS-2008} to study the controllability problem. The fact that sufficiently many spherical harmonics form a saturating subspace is established in~\cite[Theorem~10.4]{AS-2008}.

\subsection{Complex Ginzburg--Landau equation}
\label{s4.2}
We now turn to Eq.~\eqref{cgl-intro}. To simplify the presentation, we confine ourselves to the more complicated case $m=2$. In addition, we shall consider only finite-dimensional random forces; the extension to infinite-dimensional forces is simpler than in the case of the Navier--Stokes system. 

We thus consider the equation
\begin{equation} \label{cgl}
	\p_t u-(\nu +i)\Delta u+\gamma u+ic|u|^4 u=\eta(t,x), \quad x\in \T_a^3. 
\end{equation}
Here $u=u(t,x)$ is an unknown complex-valued function, $\nu$, $\gamma$, and~$c$ are given positive numbers, and~$\eta$ is a random process of the form~\eqref{1.08}, where~$\{\eta_k\}$ is a sequence of i.i.d.\ random variables in $L^2(J,H^2)$, and we set $H^s:=H^s(\T_a^3,\C)$ for $s\in\Z$. Equation~\eqref{cgl} is supplemented with the initial condition
\begin{equation} \label{ic-cgl}
	u(0)=u_0\in H^1. 
\end{equation}
We shall consider~$L^2:=L^2(\T_a^3,\C)$ as a real Hilbert space with the inner product
\begin{equation} \label{inner-product}
(u,v)=\Re\int_{\T_a^3}u\bar v\,\dd x	
\end{equation}
and endow the Sobolev spaces~$H^s$ with the associated norms and inner products. As in the case of the Navier--Stokes system, if the law of~$\eta_k$ has a compact support~$\KK$ in $L^2(J,H^2)$, then discrete-time Markov process~$(u_k,\IP_u)$ associated with~\eqref{1.1} (in which $S:H^1\times L^2(J,H^2)\to H^1$ is the time-$1$ resolving operator for~\eqref{cgl}) possesses a compact absorbing set~$X\subset H^1$, that is closed and bounded in~$H^2$ (cf.\ Theorem~\ref{t-IV-CGL}). 

Let us introduce a concept of a {\it saturating subspace\/}  for~\eqref{cgl}. Given a finite-dimensional subspace $\HH\subset H^2$ invariant under complex conjugation (that is, $\bar\zeta\in\HH$ for all $\zeta\in\HH$), we define a sequence of closed subspaces $\HH_k \subset H^2$, $k\ge0$, by the following rule:
\begin{itemize}
	\item $\HH_0$ coincides with $\HH$;
	\item if $\HH_k$ is already defined, then $\HH_{k+1}$ is the vector span of~$\HH_k$ and the products $\zeta\xi$ with $\zeta\in\HH_k$ and $\xi\in\HH$. 
\end{itemize} 
It is straightforward to check that $\{\HH_k\}$ is a non-decreasing sequence of finite-dimensional subspaces in~$H^2$ that are invariant under complex conjugation. 

\begin{definition}
	The subspace $\HH\subset H^2$ is said to be {\it saturating\/} if the union of~$\{\HH_k\}_{k\ge0}$ is dense in~$L^2$.
\end{definition}

Let us write $B(u)=ic|u|^4u$ for the nonlinear term in~\eqref{cgl} and note that it is a real-analytic function in the space~$H^2$, and its derivative $Q(u):H^2\to H^2$ acts  essentially as a multiplication operator:
$$
Q(u)v=ic\bigl(3|u|^4v+2|u|^2u^2\bar v\bigr).
$$

\begin{theorem} \label{t-mixingCGL}
	Let  $\HH\subset H^2$ be a finite-dimensional subspace that is invariant under complex conjugation, contains the function identically equal to~$1$, and is saturating. Let~$\{\varphi_i\}_{i\in\II}$ be an orthonormal basis in~$\HH$ and let~$\eta$ be a random process of the form~\eqref{1.08} in which~$\{\eta_k\}$ is an i.i.d.\ sequence in $E=L^2(J,\HH)$ such that the following conditions are fulfilled. 
\begin{description}
	\item [\bf Decomposability.]
The law~$\ell\in\PP(E)$ of~$\eta_k$ has a compact support $\KK\subset E$ containing the origin and satisfies the decomposability hypothesis~\hyperlink{H4}{{\rm(H$_4$)}}. 
	\item [\bf Observability.]
There is $T\in(0,1)$ such that the law~$\ell'$ of the restriction of the random variables~$\eta_k$ to the interval $J'=[0,T]$ is Lipschitz-observable with respect to~$\{\varphi_i\}$.
\end{description}	
Then, for any $\nu>0$, the Markov process~$(u_k,\IP_u)$ has a unique stationary measure~$\mu_\nu\in\PP(X)$, which satisfies inequality~\eqref{1.6}.
\end{theorem}

\begin{remark}
	Theorem~\ref{t-mixingCGL} remains valid in a slightly more general setting: it suffices to require that~$\HH$ should contain a subspace~$\HH'$ satisfying the hypotheses imposed on~$\HH$ in the theorem. The proof of this observation is straightforward, and we skip the details. Moreover, an analogue of Remark~\ref{r-NSglobal} is valid under the hypotheses of Theorem~\ref{t-mixingCGL}: the exponential convergence to the stationary measure holds for any initial measure with a finite moment. Finally, for equations with a higher-order nonlinearity for which the well-posedness can only be proved for essentially bounded (rather than square-integrable) functions $\eta:J\to H^2$, one can apply a slight modification of Theorem~\ref{t1.1}, requiring the map~$S$ to be defined and analytic in an $L^2$-open neighbourhood of~$L^\infty(J,H^2)$. The latter can be obtained by applying the inverse mapping theorem at any point of the space~$L^\infty(J,H^2)$; cf.\ Section~\ref{s-resolvingCGL} and Section~2.3 in~\cite{KZ-2018}.
\end{remark}

\begin{proof}[Proof of Theorem~\ref{t-mixingCGL}]
	As in the case of the Navier--Stokes system, it suffices to check Hypotheses~\hyperlink{H1}{(H$_1$)}--\hyperlink{H4}{(H$_4$)}, where~$H=H^1$ and $E=L^2(J,H^2)$. The proof of regularity properties for the resolving operator is carried by well-known methods and is presented in Section~\ref{s-resolvingCGL}. The decomposability condition~\hyperlink{H4}{(H$_4$)} for~$\ell$ holds by assumption. Thus, we only need to check Conditions~\hyperlink{H2}{(H$_2$)} and~\hyperlink{H3}{(H$_3$)}. 
	
\smallskip
{\it Step~1: Dissipativity\/}. Let us recall that $\|\cdot\|$ and~$\|\cdot\|_1$ denote the standard norms in~$L^2$ and~$H^1$, respectively. Given a number $\delta>0$, we define a new norm $\barr\cdot\barr_\delta$ on~$H^1$ by the formula $\bbar u\bbar_\delta^2=\|u\|^2+\delta\|u\|_1^2$. We claim that, for any bounded subset~$B\subset H^1$, there are $\delta>0$ and $a\in(0,1)$ such that 
\begin{equation} \label{dissipation-CGL}
	\bbar S(u,0)\bbar_\delta \le a\bbar u\bbar_\delta\quad\mbox{for $u \in B$}. 
\end{equation}
This will imply the validity of~\hyperlink{H2}{(H$_2$)} with $\hat u=0$ and~$\hat \eta=0$, since we can assume from the very beginning that the space~$H^1$ is endowed with the norm~$\barr \cdot\barr_\delta$. 

\smallskip
To establish~\eqref{dissipation-CGL}, we first note that 
\begin{equation} \label{estimate-CGL}
\|S(u)\|\le q\|u\|\quad\mbox{for $u\in L^2$}, \qquad
\|S(u)\|_1\le C_B\|u\|\quad\mbox{for $u\in B$},  	
\end{equation}
where $S(u)=S(u,0)$, $q\in(0,1)$, and~$C_B>0$ depends only on~$B$. Indeed, the first inequality is a standard estimate for the $L^2$-norm of solutions, while the second can easily be derived by taking the derivative of $t\|\nabla u\|^2$; cf.\ the proof of Theorem~\ref{t-IV-CGL}. Combining inequalities~\eqref{estimate-CGL}, we derive
$$
\bbar S(u)\bbar_\delta^2 
=\|S(u)\|^2+\delta\, \|S(u)\|_1^2
\le q^2\|u\|^2+\delta\, C_B^2\|u\|^2
\le \bigl(q^2+\delta C_B^2\bigr)\,\barr u\barr_\delta^2. 
$$
Choosing $\delta>0$ so small that $q^2+\delta C_B^2<1$, we obtain inequality~\eqref{dissipation-CGL} with some $a<1$.

\smallskip
{\it Step~2: Approximate controllability\/}. We need to prove that, for any $u\in X$ and $\ell$-a.e.\ realisation of~$\eta\in E$, the image of the derivative $(DS)(u,\eta):H^1\to H^1$ is dense. To this end, we fix some functions $u\in X$ and $\eta\in E$, denote by $\tilde u\in L^2(J,H^3)\cap W^{1,2}(J,H^1)$ the solution of~\eqref{cgl} issued from~$u$, and consider the linearised problem
	\begin{equation} \label{l-cgl}
		\dot v+Lv+Q(\tilde u)v=g, \quad v(0)=v_0, 
	\end{equation}
	where $L=-(\nu+i)\Delta+\gamma$, $g\in E$, and $v_0\in H^1$. Let us denote by $v(t;v_0,g)$ the solution of~\eqref{l-cgl}. We need to prove that, for $\ell$-a.e.\ $\eta\in E$,  the vector space $\{v(1;0,g),g\in E\}$ is dense in~$H^1$. In view of parabolic regularisation, the resolving operator for~\eqref{l-cgl} with $g\equiv0$ is continuous from~$L^2$ to~$H^1$ on an arbitrary interval $J'=[T,1]$ with $0<T<1$. Since the restriction at $t=1$ of the space of solutions~$v(t)$ for the homogeneous equation on~$J'$ with $v(T)\in L^2$ is dense in~$H^1$ (see Proposition~\ref{p-lcgl}), the required property will be established if we prove that the vector space $\{v(T;0,g),g\in E\}$ is dense in~$L^2$. To this end, we repeat the scheme used in the case of the Navier--Stokes system. The difference is that the operator~$Q$ is no longer linear in~$\tilde u$, which makes the argument slightly more involved. 
	
Let us denote by $R^{\tilde u}(t,s):L^2\to L^2$ (with $0\le s\le t\le T$) the two-parameter process solving the linearised problem~\eqref{l-cgl} with $g\equiv0$ and let  
	\begin{equation} \label{gramian-cgl}
	G^{\tilde u}:=\int_0^T R^{\tilde u}(T,t){\mathsf P}_\HH R^{\tilde u}(T,t)^*\dd t,
	\end{equation}
	where $R^{\tilde u}(T,t)^*:L^2\to L^2$ is the adjoint of $R^{\tilde u}(T,t)$, and~${\mathsf P}_\HH:L^2\to L^2$ stands for the orthogonal projection to~$\HH$. We wish to prove that $\Ker(G^{\tilde u})=\{0\}$ for any $u\in X$ and $\ell$-a.e.\ $\eta\in E$. Since the union of~$\HH_k$ is dense in~$H$, it suffices to show that any element of $\Ker(G^{\tilde u})$ is orthogonal to~$\HH_k$ for all $k\ge0$. 

To this end, we fix $u\in X$ and take any $w_0\in \Ker(G^{\tilde u})$. In view of~\eqref{gramian-cgl}, we have ${\mathsf P}_\HH R^{\tilde u}(T,t)^*w_0=0$ for any $t\in J'$, so that 
\begin{equation} \label{xiw0-cgl}
	(\zeta,R^{\tilde u}(T,t)^*w_0)=0\quad \mbox{for $t\in J'$},
\end{equation}
where $\zeta\in\HH_0$ is an arbitrary vector.  Assuming that relation~\eqref{xiw0-cgl} is true for any $\zeta\in\HH_k$, we now prove its validity for $\zeta\in\HH_{k+1}$. Once this is established, we can complete the proof by taking $t=T$. 

Differentiating~\eqref{xiw0-cgl} in time and using an analogue of relation~\eqref{soldual} for~\eqref{l-cgl}, we derive 
\begin{equation} \label{xiw0der-cgl}
	\bigl(L\zeta+Q(\tilde u(t))\zeta,w(t)\bigr)=0\quad\mbox{for $t\in J'$},
\end{equation}
where $w(t)$ is given by~\eqref{soldual}.  Let us write 
\begin{align*}
\eta(t)&=\sum_{i\in\II}\eta^i(t)\varphi_i(x),
\end{align*}
where $\eta^i(t)=(\eta(t),\varphi_i)_\HH$. Differentiating~\eqref{xiw0der-cgl} in time and using~\eqref{cgl}, we obtain
\begin{equation} \label{1der}
	\bigl(L\zeta+Q(\tilde u)\zeta,\dot w\bigr)-\bigl(B_2(\tilde u;\zeta,L\tilde u+B(\tilde u)),w\bigr)
	+\sum_{i\in\II}\bigl(B_2(\tilde u;\zeta,\varphi_i),w\bigr)\eta^i(t)=0, 
\end{equation}
where $B_k(u;\cdot)$ stands for the $k^\text{th}$ derivative of~$B(u)$ (so that $Q=B_1$, and $B_k=0$ for $k\ge6$). We thus obtain relation~\eqref{obs}, in which $J=J'$,
\begin{align*}
	a_i(t)&=\bigl(B_2(\tilde u(t);\zeta,\varphi_i),w(t)\bigr),\\
	b(t)&=	\bigl(L\zeta+Q(\tilde u(t))\zeta,\dot w(t)\bigr)-\bigl(B_2(\tilde u(t);\zeta,L\tilde u(t)+B(\tilde u(t))),w(t)\bigr). 
\end{align*} 
These are Lipschitz-continuous and continuous functions, respectively, and the observability of~$\ell'$ implies that (cf.~\eqref{xiw0der-cgl})
\begin{equation*}
	\bigl(B_2(\tilde u(t);\zeta,\varphi_i),w(t)\bigr)=0
	\quad\mbox{for $i\in\II$, $t\in J'$}. 
\end{equation*}
Applying exactly the same argument three more times, we see that 
\begin{equation} \label{5der}
	\bigl(B_5(\zeta,\varphi_i,\varphi_j,\varphi_m,\varphi_n),w(t)\bigr)=0
	\quad\mbox{for $i,j,m,n\in\II$, $t\in J'$},
\end{equation}
where we used the fact that the fifth derivative of~$B(u)$ does not depend on~$u$. We see that $w(t)$ must be orthogonal to the vector space~$\VV$ spanned by $\{(B_5(\zeta,\varphi_i,\varphi_j,\varphi_m,\varphi_n)\}$. Now note that
\begin{equation} \label{B5}
B_5(\zeta,\varphi,1,1,1)
=12ic\,(3\zeta\varphi+\bar\zeta\bar\varphi +3\bar\zeta\varphi+3\zeta\bar\varphi). 
\end{equation}
Since both~$\HH$ and~$\HH_k$ are invariant under complex conjugation, it follows from relation~\eqref{B5} that~$\VV$ must contain all the products $\zeta\xi$ with $\zeta\in\HH_k$ and $\xi\in\HH$. Thus, we have $\HH_{k+1}\subset\VV$, which completes the proof. 
\end{proof}

\section{Observable processes with decomposable laws}
\label{s6}

\subsection{Observable functions}
\label{s6.1}
Let $\HH$ be a (finite or infinite-dimensional) Hilbert space with an inner product $(\cdot,\cdot)$, let~$\{\varphi_i\}_{i\in\II}$ be an orthonormal basis in~$\HH$, and let $\eta:J\to\HH$ be a function in $L^2(J,\HH)$. Recall that the concept of an observable function was introduced in Definition~\ref{d6.1}. The following two results provide sufficient conditions for a function to be observable.  

\begin{proposition} \label{p6.2}
	Let~$\eta:J\to\HH$ be a bounded measurable function such that, for any $i\in\II$, the projection $(\eta,\varphi_i)$ has left and right limits at any point of~$J$ and is discontinuous on a countable dense set~$\D_i$ and continuous on~$J\setminus\D_i$. Suppose, in addition, that $\D_i\cap\D_j=\varnothing$ for $i\ne j$. Then~$\eta$ is observable with respect to~$\{\varphi_i\}$.
\end{proposition}

\begin{proof}
We confine ourselves to the case of an infinite-dimensional space~$\HH$. Let~$a_i$ be Lipschitz-continuous functions and~$b$ a continuous function satisfying~\eqref{6.1} and~\eqref{6.2}. The boundedness of~$\eta$ implies that the series $\sum_{i}a_i(t)(\eta(t),\varphi_i)$ converges uniformly in $t\in J$. It follows that its sum~$\zeta(t)$ has left and right limits at any point of~$J$. Let us fix any $i\in\II$ and $s\in\D_i$. Since $s\notin\D_j$ for $j\ne i$ and $(\eta(t),\varphi_j)$ is continuous on~$J\setminus\D_j$, we see that
\begin{equation} \label{6.3}
\Delta\zeta(s):=\zeta(s^+)-\zeta(s^-)=a_i(s)\bigl(\eta(s^+)-\eta(s^-),\varphi_i\bigr).	
\end{equation}
Since~$b$ is continuous on~$J$, it follows from~\eqref{6.2} that the expression in~\eqref{6.3} must vanish. By the hypotheses of the proposition, the function $(\eta,\varphi_i)$ has a jump at~$s$, whence we conclude that $a_i(s)=0$. Since $s\in\D_i$ is arbitrary and~$\D_i$ is dense, the function~$a_i$ is identically zero for any $i\in\II$, so that~$b$ also vanishes. 
\end{proof}

We now consider the case of a finite set~$\II$ and write $\II=\{1,\dots,N\}$. Let us fix an orthonormal basis~$\{\varphi_i\}$ and set $\eta=(\eta^1,\dots,\eta^N)$. 

\begin{proposition} \label{p6.3}
	Let the functions $\eta^i$ have left and  right  limits at any point of~$J$, and denote by~$\Delta \eta^i(t)$ the jump of~$\eta^i$ at~$t$. Suppose that, for any $s\in J$ and $\e\in(0,1]$, there are $t_1^\e,\dots,t_N^\e\in[s-\e,s+\e]$ such that the $N\times N$ matrix~$R_\e(s)$ with the entries  $r_l^i:=\Delta\eta^i(t_l^\e)$ is invertible and satisfies the inequality
	\begin{equation} \label{6.4}
		\bigl\|R_\e^{-1}(s)\bigr\|\le C\e^{-\theta},
	\end{equation}
	where $C>0$ and $\theta\in(0,1)$ do not depend on~$\e$. Then~$\eta$ is observable. 
\end{proposition}

\begin{proof}
Let $a_1,\dots,a_N$ be real-valued Lipschitz-continuous functions and let~$b$ a continuous function such that~\eqref{6.2} holds. We fix an arbitrary $s\in J$ and, given $\e>0$, find points $t_1^\e,\dots,t_N^\e$ satisfying the properties mentioned in the statement. It follows from~\eqref{6.2} that 
$$
\sum_{i=1}^N r_l^ia_i(t_l^\e)=0, \quad l=1,\dots,N. 
$$
Setting $a=(a_1,\dots,a_N)$ and using the Lipschitz-continuity of~$a$, we derive the relation
$$
R_\e(s)a(s)=d_\e(s), 
$$
where $|d_\e(s)|\le C_1\e$. Applying the matrix $R_\e^{-1}(s)$ and recalling~\eqref{6.4}, we obtain $|a(s)|\le C_2\e^{1-\theta}$, where $C_2$ does not depend on~$\e$. Since~$s$ and~$\e$ were arbitrary, we see that $a\equiv0$ and, hence, $b\equiv0$. 
\end{proof}

\subsection{Processes satisfying the hypotheses of Section~\ref{s4.1}}
\label{s6.2}
We now construct some examples of stochastic processes on~$[0,1]$ that possess the decomposability and observability properties. The decomposability will be a simple consequence of the explicit form of the processes, whereas observability will follow from a finer analysis based on Propositions~\ref{p6.2} and~\ref{p6.3}. In all the examples, we deal with the observability on the interval $J=[0,1]$; however, exactly the same argument shows that the result remains true on the  interval~$[0,T]$ with any $T\in(0,1)$. 

\subsubsection*{Jump process}
Let~$\HH$ be a Hilbert space with an orthonormal basis~$\{\varphi_i\}_{i\in\II}$. Suppose that, for any $i\in\II$, we are given an orthonormal basis $\{\psi_l^i\}_{l\ge1}$ in~$L^2(J)$ such that the following properties hold:
\begin{itemize}
\item[\bf(a)]
the function~$\psi_l^i$ is continuous outside a finite set~$\D_l^i$ and has different left and right limits at the points of~$\D_l^i$;
\item[\bf(b)]
the union $\cup_l\D_l^i:=\D_i$ is dense in~$J$ for each $i\in\II$; 
\item[\bf(c)]
the intersection $\D_{i_1}\cap \D_{i_2}$ is empty for $i_1\ne i_2$. 
\end{itemize}
Let us define  a stochastic process by the relation 
\begin{equation} 	\label{6.7}
	\eta(t)=\sum_{i\in\II}\eta^i(t)\varphi_i, \quad 
	\eta^i(t)=\sum_{l=1}^\infty b_{il}\xi_{l}^i\psi_l^i(t),
\end{equation}
where $\xi_{l}^i$ are independent random variables whose laws possess Lipschitz-continuous densities supported by~$[-1,1]$ and $b_{il}\in\R$ are non-zero numbers such that
\begin{equation} \label{6.6}
\sum_{i\in\II}\sum_{l=1}^\infty
	\bigl(b_{il}^2+|b_{il}|\,\|\psi_l^i\|_\infty\bigr)<\infty.
\end{equation}
It is obvious that the law~$\ell$ of~$\eta$ is a decomposable measure\,\footnote{Note that the index~$j$ in~\eqref{1.3} is now replaced by the pair $(l,i)$, and vectors~$e_j$ are the products $\psi_l^i(t)\varphi_i$.} on~$L^2(J,\HH)$. We claim that~$\ell$ is observable with respect to~$\{\varphi_i\}$. Indeed, it follows from~\eqref{6.6} that the series in~\eqref{6.7} converges uniformly in~$t$ and~$\omega$, so that the trajectories of~$\eta^i$ have left and right limits at any point and are continuous outside~$\D_i$. If we prove that they are a.s.\ discontinuous at any point of~$\D_i$, then Proposition~\ref{p6.2}  will immediately imply the observability of~$\eta$ with respect to~$\{\varphi_i\}$. 

Let us fix any point $s\in\D_i$ and calculate the jump of~$\eta^i$ at~$s$. It follows from~\eqref{6.7} that 
$$
\Delta\eta^i(s)=\sum_{l=1}^\infty b_{il}\xi_l^i\bigl(\psi_l^i(s^+)-\psi_l^i(s^-)\bigr).
$$
This series converges absolutely, and at least one of the terms $\psi_l^i(s^+)-\psi_l^i(s^-)$ is non-zero. Since the random variables~$\xi_l^i$ are independent, and  their laws have densities, we conclude that so does the law of $\Delta\eta^i(s)$. Therefore, $\Delta\eta^i(s)\ne0$ with probability~$1$. 

\subsubsection*{Haar process with exponentially decaying coefficients}
Let us recall the definition of the Haar system~$\{h_0,h_{jl}\}$. We set 
\begin{align}
	h_0(t)&=\left\{
	\begin{array}{cl}
		1 & \mbox{for $0\le t<1$},\\[2pt]
		0 & \mbox{for $t<0$ or $t\ge1$},
	\end{array}
\right.
\label{6.8}
\\
	h_{jl}(t)&=\left\{
	\begin{array}{cl}
		0 & \mbox{for $t<l2^{-j}$ or $t\ge (l+1)2^{-j}$},\\[2pt]
		1 & \mbox{for $l2^{-j}\le t<\bigl(l+\tfrac12\bigr)2^{-j}$},\\[2pt]
		-1 & \mbox{for $\bigl(l+\tfrac12\bigr)2^{-j}\le t<(l+1)2^{-j}$}, 
	\end{array}
\right.
\label{6.9}
\end{align}
where $j\ge1$ and~$l\ge0$ are integers. It is well known that the restrictions of the functions $\{h_0,h_{jl},j\ge0, 0\le l\le 2^j-1\}$ to the interval~$J=[0,1]$ form an orthogonal basis in~$L^2(J)$; see Section~22 in~\cite{lamperti1996}. This implies that the system of functions $\{h_0(\cdot+k), h_{jl}(\cdot)\}$, where $j\ge1$ and $k,l\ge 0$, form an orthogonal basis of~$L^2(\R_+)$.  

Let us assume that~$\HH$ is an $N$-dimensional Euclidean space. We fix an orthonormal basis~$\{\varphi_i,1\le i\le N\}\subset\HH$ and consider the following process in~$\HH$:
\begin{equation} \label{6.10}
	\eta(t)=\sum_{i=1}^Nb_i\biggl(\xi_0^ih_0(t)
	+\sum_{j=1}^\infty\sum_{l=0}^{2^j-1}c_j\xi_{jl}^ih_{jl}(t)\biggr)\varphi_i,
	\quad t\in J,
\end{equation}
where $b_i\in\R$ are non-zero numbers, $c_j=A^{-j}$ with some $A>1$, and~$\xi_{jl}^i$ are i.i.d.\ scalar random variables such that $|\xi_{jl}^i|\le1$ almost surely, and their law possesses a continous density with respect to the Lebesgue measure. For any $j\ge1$, any point of~$J$ belongs to the support of at most two functions~$h_{jl}$, so that the series in~\eqref{6.10} absolutely converges, uniformly with respect to~$t\in J$. Since the functions~$\{h_0\varphi_i,h_{jl}\varphi_i\}$ form an orthonormal basis in~$L^2(J,\HH)$, the law $\ell\in\PP(L^2(J,\HH))$ of~$\eta$ is decomposable, and any random variable with law~$\ell$ has the form~\eqref{6.10}. Let us prove that, if $A>1$ is sufficiently close to~$1$, then a.e.\ trajectory of~$\eta$ is observable.

We shall apply Proposition~\ref{p6.3}. Let us fix a point $s\in J$ and a number $\e>0$ and consider the points
\begin{equation} \label{6.11}
	\tau_l^j:=\bigl(l+\tfrac12\bigr)2^{-j}, \quad 0\le l\le 2^j-1.
\end{equation}
It is clear that if an integer~$j$ of the form  $C_1\log\frac1\e$ is fixed, then there are~$N$ points 
$$
\tau_{l_1}^j=:\tau_1,\quad\dots, \quad \tau_{l_N}^j=:\tau_N, \quad l_r=l_0+r,
$$
that belong to the $\e$-neighbourhood of~$s$. Consider the matrix $R_\e(s)$ corresponding to the points $\tau_1,\dots,\tau_N$. Its entries are $r_m^i=\Delta\eta^i(\tau_m)=\Delta\eta^i(\tau_{l_m}^j)$, where  
\begin{align}
	\Delta\eta^i(\tau_l^j)
	&=b_i\biggl(-2c_j\xi_{jl}^i+\sum_{r=1}^\infty 
	c_{j+r}\bigl(\xi_{j+r,(2l+1)2^{r-1}}^i-\xi_{j+r,(2l+1)2^{r-1}-1}^i\bigr)\biggr) 
	\notag\\
	&= b_iA^{-j}\zeta_l^i(j). \label{6.12}
\end{align}
Here $\{\zeta_l^i(j)\}$ are i.i.d.\ random variables whose law has a continuous density with respect to the Lebesgue measure. It follows that the determinant of the matrix $R_\e(s)$ can be written as 
\begin{equation} \label{6.13}
	\det R_\e(s)=b_1\cdots b_N A^{-Nj}\Sigma(j),
\end{equation}
where $\Sigma(j)$ is the determinant of the matrix $(\zeta_{l_m}^i(j),1\le i,m\le N)$. Notice that the entries of this matrix are i.i.d.\ random variables whose law does not depend on~$j$. Since the determinant is a non-zero polynomial of the matrix entries, Corollary~\ref{r2.9} applies and gives the inequality
\begin{equation} \label{6.14}
\IP\{|\Sigma(j)|\le r\}\le Cr^c\quad\mbox{for any $r\in[0,1]$},
\end{equation}
where $C$ and $c$ are positive numbers not depending on~$j$. Taking $r=e^{-\delta j}$, where $\delta>0$ is sufficiently small and will be chosen below, we derive 
$$
	\IP\{|\Sigma(j)|\le e^{-\delta j}\}\le Ce^{-c\delta j}
	\quad\mbox{for any $j\ge0$}.
$$
By the Borel--Cantelli lemma, there is an almost surely finite random integer $j_0\ge1$ such that
$$
	|\Sigma(j)|>e^{-\delta j}\quad\mbox{for $j\ge j_0$}. 
$$
Combining this with~\eqref{6.13}, we derive
\begin{equation} \label{6.15}
|\det R_\e(s)|\ge C_2e^{-j(N\log A+\delta)},
\end{equation}
where $j\ge j_0$ is arbitrary. For $j\sim C_1\log\frac1\e$ (such a choice is possible for a.e.~$\omega$ and sufficiently small~$\e$), we obtain
$$
|\det R_\e(s)|\ge C_2\e^\theta,
$$
where $\theta=C_1(N\log A+\delta)$. Since the entries of~$R_\e(s)$ are a.s.\ bounded by a universal number, we conclude that~\eqref{6.4} holds. Taking $\delta=\log A$ and $A>1$ sufficiently close to~$1$, we see that~$\theta<1$, and Proposition~\ref{p6.3} implies the required result.  

\subsubsection*{Haar process with algebraically decaying coefficients}
We consider again process~\eqref{6.10}, in which the numbers~$c_j$ go to zero at an algebraic  rate:
\begin{equation} \label{6.16}
	c_j=Cj^{-q}\quad\mbox{for all $j\ge 1$},
\end{equation}
where $C>0$ and $q>1$ are some numbers. For the same reasons as in the previous case, the series in~\eqref{6.10} absolutely converges, uniformly in $t\in J$, and the law of~$\eta$ is a decomposable measure on~$L^2(J,\HH)$. We claim that, for any $T\in(0,1]$, a.e.\ trajectory of~$\eta$ is observable on $J'=[0,T]$. To prove it, we apply the same argument as before. The first line of~\eqref{6.12} remains true, and the jumps take the form $\Delta \eta^i(\tau^j_l)=Cb_ij^{1-q}\zeta_l^i(j)$, where 
$$
	\zeta_l^i(j)=-2j^{-1}\xi_{jl}^i+j^{-1}
	\sum_{r=1}^\infty (1+rj^{-1})^{-q}
	\bigl(\xi_{j+r,(2l+1)2^{r-1}}^i-\xi_{j+r,(2l+1)2^{r-1}-1}^i\bigr). 
$$
It follows that the determinant of~$R_\e(s)=(r_m^i)$ can be written as (cf.~\eqref{6.13})
$$
\det R_\e(s)=C^Nb_1\cdots b_Nj^{-N(q-1)}\Sigma(j),
$$
where $\Sigma(j)$ is the determinant of the matrix $(\zeta_{l_m}^i(j))$. A simple calculation shows that $\{\zeta_{l_m}^i(j)\}$ are independent random variables that are a.s.\ bounded by a universal number, and their laws do not depend on~$i$ and~$m$. Suppose we have shown that the laws of $(\zeta_l^i(j))$ possess densities~$\rho_j$ (with respect to the Lebesgue measure) satisfying the inequality 
\begin{equation}\label{rhojrho}
	\rho_j(r)\le j\rho(r)\quad\mbox{for all $j\ge1$, $r\in\R$},
\end{equation}	
where $\rho:\R\to\R$ is a continuous function with compact support. Applying Corollary~\ref{r2.9}, we conclude that inequality~\eqref{6.14} remains valid with~$C$ replaced by~$Cj^N$. Hence, by the same argument as before, we obtain (cf.~\eqref{6.15})
\begin{equation} \label{6.17}
	|\det R_\e(s)|\ge C_2e^{-\delta j-N(q-1)\log j}. 
\end{equation}
The proof can now be completed as in the case of exponentially decaying coefficients. 

It remains to prove~\eqref{rhojrho}. To this end, we write $\zeta_l^i(j)=-2j^{-1}\xi_{jl}^i+\eta_l^i(j)$, where~$\eta_l^i(j)$ is a bounded random variable independent of $\xi_{jl}^i$. By the hypotheses, the law of~$\xi_{jl}^i$ has a bounded continuous density not depending on~$i$, $j$, and~$l$. Let us denote by~$M$ its maximum. Then the density~$\tilde\rho_j$ of the law for $-2j^{-1}\xi_{jl}^i$ is bounded by~$Mj/2$. Since~$\rho_j$ is a convolution with~$\tilde\rho_j$, it is bounded by the same constant. On the other hand, we know that~$\zeta_l^i(j)$ is almost surely bounded by a universal number. It follows that~\eqref{rhojrho} holds for some continuous function~$\rho$ with compact support. 

\section{Examples of saturating subspaces}
\label{s-satset}

In this section, we discuss some algebraic conditions that ensure the saturating property of a given subspace. This type of conditions now are rather well known in the control theory and Malliavin calculus for PDEs; see~\cite{EM-2001,AS-2006,HM-2006,MP-2006}. 

\subsection*{Navier--Stokes system}
Let us endow the space~$\R^2$ with the scalar product 
$$
  \lag x,y\rag_a=\sum_{i=1}^2 a_i^{-1}x_i\,y_i
  $$ and the corresponding norm   $|x|_a=\sqrt{\lag x,x\rag_a}$. In the case $a=(1,1)$, we write $\lag \cdot,\cdot\rag$ and~$|\cdot|$. Let~$ \Z^2_*$ be the set of non-zero integer vectors $l=(l_1,l_2)\in\Z^2$. For~$l\in \Z^2_*$, we define the functions
$$
e^a_l(x)= \begin{cases} c_l^a(x)  & \text{if }l_1>0\text{ or } l_1=0,\, l_2>0, \\ s_l^a(x)   & \text{if }l_1<0\text{ or } l_1=0,\, l_2<0 \end{cases}
$$
on~$\T_a^2$, where
  $$
  c^a_l(x) =  l^{\bot_a}\cos\lag l,x\rag_a, \quad  s^a_l(x) =  l^{\bot_a} \sin\lag l,x\rag_a,  \quad 
  l^{\bot_a}=(-a_2^{-1}l_2,a_1^{-1}l_1). 
  $$
  For any  subset $\II\subset \Z^2$, let  $\Z^2_\II$ be the set of all vectors in~$\Z^2$ that can be represented as  finite linear combinations of elements of~$\II$ with integer coefficients. We shall say that   $\II $ is  a {\it generator\/} if $\Z^2_\II= \Z^2$. 

\smallskip
Let $\II\subset \Z^2_*$ be a finite symmetric  set (i.e., $-\II=\II$) and let 
\begin{equation}\label{0.03}
\HH(\II)= \textup{span} \{e_l^a: l\in \II\}.
\end{equation}
We denote by $\HH_k(\II)\subset H^3$ the spaces defined in Section~\ref{s4.1} with $\HH=\HH(\II)$.

\begin{proposition}\label{P:1} 
Let $\II\subset\Z_*^2$ be a finite symmetric set. Then the space $\HH(\II)$ is saturating  if and only if~$\II$ is a generator and contains at least two non-parallel elements~$m$ and~$n$ such that~$|m^{\bot_a}|\neq |n^{\bot_a}|$. 
\end{proposition}

This result implies that the space   $\HH(\II)$ is saturating when    
\begin{align*}
	\II&=\{(1,0),(-1,0),(1,1), (-1,-1)\}&&\mbox{for $a=(1,1)$},
	\\\II&=\{(1,0),(-1,0),(0,1), (0,-1)\}& &\mbox{for $a=(\lambda,1)$ with $\lambda\ne 1$}.
\end{align*}
\begin{proof}[Proof of Proposition \ref{P:1}]
{\it Step 1\/}. 
We first establish some relations for Leray's projection~$\Pi$ and the quadratic function~$Q$ defined in Section~\ref{s4.1}. For any $l\in \R^2_*$, let us denote by $P(l):\R^2\to \R^2$ the orthogonal projection onto $\lspan(l)$ with respect to the scalar product $\lag \cdot,\cdot\rag$, so that $P(l)b=|l|^{-2}\langle b,l\rangle l$. We claim that 
\begin{align}
	\Pi \bigl(b \cos\lag l,x \rag_a\bigl)
	&= P(l^{\bot_a})b\,\cos\lag l,x \rag_a, 
	\label{pic}\\
	\Pi \bigl(b \sin\lag l,x \rag_a\bigl)
	&= P(l^{\bot_a})b\,\sin\lag l,x \rag_a, 
	\label{sic}\\ 
	2 Q(c_l^a,c_r^a)&=
	\Pi\bigl(\langle l^{\bot_a},r\rangle_a r^{\bot_a}-\langle r^{\bot_a},l\rangle_a l^{\bot_a}\bigr)\sin\langle l-r,x\rangle_a
\notag\\
	&\quad	-\Pi\bigl(\langle l^{\bot_a},r\rangle_a r^{\bot_a}+\langle r^{\bot_a},l\rangle_a l^{\bot_a}\bigr)\sin\langle l+r,x\rangle_a,  \label{ckcl}\\
	2Q(c_l^a,s_r^a)&= \Pi\bigl(\langle l^{\bot_a},r\rangle_a r^{\bot_a}+\langle r^{\bot_a},l\rangle_a l^{\bot_a}\bigr)\cos\langle l+r,x\rangle_a
\notag\\
	&\quad +\Pi\bigl(\langle l^{\bot_a},r\rangle_a r^{\bot_a}-\langle r^{\bot_a},l\rangle_a l^{\bot_a}\bigr)\cos\langle l-r,x\rangle_a,\label{cksl}\\
	2Q(s_l^a,s_r^a)&=
	\Pi\bigl(\langle l^{\bot_a},r\rangle_a r^{\bot_a}-\langle r^{\bot_a},l\rangle_a l^{\bot_a}\bigr)\sin\langle l-r,x\rangle_a
\notag\\
	&\quad +\Pi\bigl(\langle l^{\bot_a},r\rangle_a r^{\bot_a}+\langle r^{\bot_a},l\rangle_a l^{\bot_a}\bigr)\sin\langle l+r,x\rangle_a, \label{sksl}
\end{align}
We confine ourselves to the proof of~\eqref{pic} and~\eqref{ckcl}, since the other relations can be established in a similar way. Let us recall that Leray's projection of a function~$v\in L^2(\T_a^2,\R^2)$ can be written as $\Pi v=v-\nabla (\Delta^{-1}\diver v)$, where $\Delta^{-1}$ is the inverse of the Laplacian with range in the space of functions with zero mean value. Relation~\eqref{pic} is now implied by the following simple formulas:
$$
\diver \bigl(b \cos\lag l,x \rag_a\bigl)=-\langle b,l\rangle_a \sin\lag l,x \rag_a\,, \quad 
\Delta^{-1}\sin\lag l,x \rag_a
	=-\bigl|l^{\bot_a}\bigr|^{-2}\sin\lag l,x \rag_a\,. 
$$

To prove~\eqref{ckcl}, we write 
$$
\langle c_l^a, \nabla \rangle c_r^a
=\langle l^{\bot_a},r\rangle_a r^{\bot_a} \cos\langle l,x\rangle_a \sin\langle r,x\rangle_a.
$$
Interchanging the roles of~$l$ and~$r$, using elementary formulas for the product of trigonometric functions, and recalling the definition of~$Q$, we arrive at~\eqref{ckcl}. 

\smallskip
{\it Step 2\/}. 
Let us prove that if $l,r\in \Z_*^2$ are non-parallel vectors such that $|l^{\bot_a}|\neq |r^{\bot_a}|$, $c_r^a, s_r^a\in\HH(\II)$, and~$c_l^a, s_l^a\in  \HH_k(\II)$ for some $k\ge0$, then the functions $c_{l+r}^a, c_{l-r}^a, s_{l+r}^a, s_{l-r}^a$ belong to~$\HH_{k+1}(\II)$. Indeed, a direct verification shows that 
\begin{multline}
P((l-r)^{\bot_a})\bigl(\langle l^{\bot_a},r\rangle_a r^{\bot_a}-\langle r^{\bot_a},l\rangle_a l^{\bot_a}\bigr)\\
=|(l-r)^{\bot_a}|^{-2}\langle l^{\bot_a},r\rangle_a	\bigl(|l^{\bot_a}|^2-|r^{\bot_a}|^2\bigr)(l-r)^{\bot_a}. 
\label{proj}
\end{multline}
Let $C_a(l,r)$ be the coefficient in front of $(l-r)^{\bot_a}$. Combining~\eqref{ckcl}, \eqref{sksl}, \eqref{sic}, and~\eqref{proj}, for any $\alpha,\beta\in\R$, we derive 
\begin{align}
	Q(c_l^a,2\alpha c_r^a)+Q(s_l^a,2\beta s_r^a) 
	&=C_a(l,r)(\alpha+\beta)(l-r)^{\bot_a}\sin\langle l-r,x\rangle_a
	\notag\\
	&\quad -C_a(l,-r)(\alpha-\beta)(l+r)^{\bot_a}\sin\langle l+r,x\rangle_a.	
	\label{lcomb}
\end{align}
Since $l$ and~$r$ are non-parallel vectors such that $|l^{\bot_a}|\neq |r^{\bot_a}|$, the numbers $C_a(l,r)$ and $C_a(l,-r)$ are non-zero. This and relation~\eqref{lcomb} readily imply that $s_{l+r}^a, s_{l-r}^a\in\HH_{k+1}(\II)$. A similar argument using~\eqref{pic}, \eqref{cksl}, and~\eqref{proj} shows that $c_{l+r}^a, c_{l-r}^a\in\HH_{k+1}(\II)$. 

\smallskip
{\it Step 3\/}. 
We now take any symmetric set $\II\subset\Z_*^2$ and prove the necessity of the hypotheses of Proposition~\ref{P:1}. If~$\II$ is not a generator, then there is a vector $l\in \Z^2\setminus\Z_\II^2$. It follows from~\eqref{pic}--\eqref{sksl} that the function~$c_l^a$ is  orthogonal to $\HH_\ty(\II)$ in~$H$, so~$\HH(\II)$ is not saturating. Furthermore, if any pair of elements $(m,n)\in\II$ either are parallel or satisfy the relation $|m^{\bot_a}|\neq |n^{\bot_a}|$, then \eqref{ckcl}--\eqref{proj} imply that $\HH_k(\II)=\HH(\II)$ for all $k\ge0$, so that $\HH(\II)$ is not saturating. 

To prove the sufficiency, let us assume that~$\II$ is a generator containing at least two non-parallel elements~$m$ and~$n$ such that $|m^{\bot_a}|\neq |n^{\bot_a}|$.  We define a sequence of symmetric sets $\II_k\subset \Z_*^2$ by the following rule: $\II_0=\II$, and for $k\ge1$, 
$$
\II_k=\II_{k-1}\cup\{l+r: l\in \II_{k-1}, r\in\II, |l^{\bot_a}|\neq |r^{\bot_a}|, \mbox{$l$ and $r$ are not parallel}\}.
$$
Let~$\II_\infty$ be the union of the sets~$\II_k$ with $k\ge0$ and of the point~$(0,0)$. The proof of Proposition~1 in~\cite{AS-2006} (see also Proposition~4.4 in \cite{HM-2006}) implies that~$\II_\infty$ coincides with~$\Z^2_\II=\Z^2$. It follows that the union of~$\HH(\II_k)$ with $k\ge0$ contains all divergence-free trigonometric polynomials, whence we see that~$\HH(\II)$ is saturating. 
    \end{proof}

\subsection*{Ginzburg--Landau equation}
Let us fix a vector $a=(a_1,a_2,a_3)$ with positive components  and define the functions 
$$
c^a_l(x) =  \cos\lag l,x\rag_a, \quad 
s^a_l(x) =  \sin\lag l,x\rag_a\quad\mbox{for $l \in  \Z^3_*$},
$$ 
where $\lag l,x\rag_a = \sum_ia_i^{-1}l_ix_i$. We now set $e_0^a\equiv1$ and 
$$
e_l^a(x)= 
\begin{cases} c_l^a(x)  & \text{if }l_1 > 0\text{ or } l_1=0,\, l_2 > 0 \text{ or } l_1=l_2=0,\, l_3 > 0, \\ s_l^a(x) & \text{if }l_1 < 0\text{ or } l_1=0,\, l_2 < 0\text{ or } l_1=l_2=0,\, l_3 < 0. 
\end{cases}
$$
Given a finite symmetric set~$\II\subset\Z^3$ containing the origin, introduce the space (cf.~\eqref{0.03})
$$
\HH(\II)=\lspan\{e_l^a, ie_l^a:l\in\II\}.
$$
We denote by~$\HH_k(\II)$ the sets~$\HH_k$ defined in Section~\ref{s4.2} for $\HH=\HH(\II)$.

\begin{proposition}\label{P:2}
The subspace $\HH(\II)$ is saturating  if and only if~$\II$ is a generator. In particular, the set $\II=\{(0,0,0), (\pm1,0,0),(0,\pm1,0),(0,0,\pm1)\}$ gives rise to the $14$-dimensional saturating subspace~$\HH(\II)$. 
\end{proposition}

\begin{proof}  
The necessity of the condition is straightforward (cf.\ the case of the Navier--Stokes system), so that we confine ourselves to the proof of sufficiency. It is enough to show that if $l,r\in \Z_*^3$ are two vectors such that  $c_r^a, s_r^a\in\HH(\II)$ and~$c_l^a, s_l^a\in  \HH_k(\II)$ for some $k\ge0$, then the functions $c_{l+r}^a, c_{l-r}^a, s_{l+r}^a, s_{l-r}^a$ belong to~$\HH_{k+1}(\II)$. 

To see this, let us note that 
$$
c_l^a(x)c_r^a(x)=\frac12\bigl(c_{l-r}^a(x)+c_{l+r}^a(x)\bigr), \quad 
s_l^a(x)s_r^a(x)=\frac12\bigl(c_{l-r}^a(x)-c_{l+r}^a(x)\bigr).
$$
It follows that $c_{l+r}^a, c_{l-r}^a\in \HH_{k+1}(\II)$. A similar argument applies to the functions~$s_{l+r}^a$ and $s_{l-r}^a$. This completes the proof of the proposition.
\end{proof}

\section{Appendix}
\label{s5} 

\subsection{Measurable version of gluing lemma}
\label{s5.4}

Let $X$ be a Polish space and let $(W,\WW)$ be a measurable space. Recall that a family $\{\mu_w,w\in W\}\subset\PP(X)$ is called a {\it random probability measure\/} (RPM) on~$X$ with the underlying space~$(W,\WW)$ if for any $\Gamma\in\BB(X)$, the mapping $w\mapsto\mu_w(\Gamma)$ is measurable from~$W$ to~$\R$. A proof of the following lemma can be found in~\cite[Corollary~3.6]{BM-2019} (see also Lemma~7.6 in~\cite{villani2003} and Corollary~4 in~\cite{BM-doklady2019}). 

\begin{theorem} \label{t-gluing}
	Let $X$, $Y$, and~$Z$ be Polish spaces and let~$\{\MMMM_w\}$ and $\{\NNNN_w\}$ be RPM on~$X\times Y$ and~$Y\times Z$, respectively, with the same underlying space~$(W,\WW)$ such that the marginals of~$\MMMM_w$ and~$\NNNN_w$ on~$Y$ are the same for any $w\in W$. Then there is a RPM $\{\LLLL_w\}$ on~$X\times Y\times Z$ such that, for any $w\in W$, the marginals of~$\LLLL_w$ on~$X\times Y$ and~$Y\times Z$ coincide with~$\MMMM_w$ and~$\NNNN_w$, respectively. 
\end{theorem}

This result can be reformulated as follows: under the hypotheses of the theorem, there is a probability space $(\Omega,\FF,\IP)$ and measurable functions $\xi_w(\omega)$, $\eta_w(\omega)$, and~$\zeta_w(\omega)$ defined on $W \times\,\Omega$ and taking values in~$X$, $Y$, and~$Z$, respectively, such that 
$$
\DD(\xi_w,\eta_w)=\MMMM_w, \quad \DD(\eta_w,\zeta_w)=\NNNN_w
\quad\mbox{for any $w\in W$}. 
$$

\subsection{Density of sets of solutions for a linear parabolic PDE}
\label{linearGL}
Retaining the notation of Section~\ref{s4.2}, let us consider the following linear parabolic PDE on $J=[0,1]$: 
\begin{equation} \label{lppde}
	\dot v-\mu \Delta v+b(t,x)v+c(t,x)\bar v=0, \quad x\in \T_a^3, 
\end{equation}
where $b, c\in H^1(J\times \T_a^3,\C)$ are bounded functions, and $\mu\in\C$ is a parameter with  a positive real part. We denote by~$\UU$ the space of all $v\in L^2(J,H^2)\cap W^{1,2}(J,L^2)$ satisfying~\eqref{lppde} and, for $s\in J$, set $\UU_s=\{u(s,\cdot): u\in\UU\}\subset H^1$. The following result is a variation of the well-known property of approximate controllability of~\eqref{lppde} by a starting control. 

\begin{proposition} \label{p-lcgl}
	For any $s\in J$, the subspace~$\UU_s$ is dense in~$H^1$. 
\end{proposition}

\begin{proof}
	The claim is trivial for $s=0$, so without loss of generality, let us assume that $s=1$. We only outline the proof, which is based on a classical argument.

\smallskip 
For $\tau\in[0,1)$, let us denote by $\{R(t,\tau), \tau\le t\le 1\}$ the resolving process for Eq.~\eqref{lppde} with an initial condition specified at $t=\tau$. This means that, for any $v_0\in H^1$, the function $v(t)=R(t,\tau)v_0$ is the unique solution of~\eqref{lppde} in the space $L^2([\tau,1],H^2)\cap W^{1,2}([\tau,1],L^2)$ such that $v(\tau)=v_0$. We  denote by~$R(1,t)^*$ the adjoint of~$R(1,t)$ with respect to the inner product in~$H^1$. Setting $L= I-\Delta$, it is straightforward to check that, for any $w_1\in H^1$, the function $w:t\mapsto R(1,t)^*w_1$ is the unique solution in~$L^2(J,H^2)\cap W^{1,2}(J,L^2)$  for the dual problem
\begin{equation} \label{dual-lppde}
	\dot w+\bar\mu \Delta w-L^{-1}(\bar b Lw)-L^{-1}(cL\bar w)=0, \quad w(1)=w_1,
\end{equation}
and the inner product $(v(t),w(t))_1$ does not depend on~$t$.

Suppose now that~$\UU_1$ is not dense in~$H^1$. Then there is a non-zero $w_1\in H^{1}$ such that 
\begin{equation} \label{orthogon}
	0=\bigl(R(1,0)v_0,w_1\bigr)_1=\bigl(v_0,R(1,0)^*w_1\bigr)_1
	\quad\mbox{for any $v_0\in H^1$},
\end{equation}
whence we conclude that $R(1,0)^*w_1=0$. Thus, the solution $w(t)$ of problem~\eqref{dual-lppde} vanishes at $t=0$.  By the backward uniqueness for~\eqref{dual-lppde} (e.g., see Section~II.8 in~\cite{BV1992}), we conclude that $w_1=0$. This completes the proof of the proposition. 
\end{proof}

\subsection{Resolving operator for the Navier--Stokes system}
\label{s-resolvingNS}
Let us consider the Navier--Stokes equations~\eqref{1.7} on a 2D torus~$\T_a^2$, where $a=(a_1,a_2)$ is an arbitrary vector with positive coordinates. Given any $T>0$, we define the Hilbert spaces 
$$
\XX_T=L^2(J_T,H_\sigma^1)\cap W^{1,2}(J_T,H_\sigma^{-1}),\quad \YY_T=L^2(J_T,H_\sigma^{-1}),
$$
where $J_T=[0,T]$, and $H_\sigma^s$ stands for the space of divergence-free vector functions on~$\T_a^2$ whose components belong to the Sobolev space of order $s\in\Z$. It is a well-known fact that, for any $u_0\in L_\sigma^2$ (where $L_\sigma^2=H_\sigma^0$) and $\eta\in\YY_T$, Eq.~\eqref{1.7} has a unique solution $u\in\XX_T$ satisfying the initial condition
\begin{equation} \label{NS-IC}
	u(0,x)=u_0(x). 
\end{equation}
We shall denote the solution of~\eqref{1.7}, \eqref{NS-IC} by $\RR(u_0,\eta)$, so that~$\RR$ is a map from $L_\sigma^2\times \YY_T$ to~$\XX_T$. The following result, which is essentially established in~\cite{kuksin-1982} (see also~\cite[Chapter~1]{VF1988}), implies, in particular, that for any fixed $u_0\in L_\sigma^2$, the map taking~$\eta$ to~$u(1)$ is analytic from $\YY_T$ to~$L_\sigma^2$. 

\begin{theorem} \label{t-NSanalytic}
	The map $\RR:L_\sigma^2\times \YY_T\to\XX_T$ is analytic. 
\end{theorem}

\begin{proof}
We shall apply an analytic implicit function theorem (see Section~3.3B in~\cite{berger1977}). Let us denote by $z\in\XX_T$ the solution of the (linear) Stokes problem 
\begin{equation} \label{stokes}
\p_t z+\nu Lz=\eta, \quad z(0)=u_0.	
\end{equation}
A solution of~\eqref{1.7}, \eqref{NS-IC} is sought in the form $u=z+A\zeta$, where $\zeta\in \YY_T$ is an unknown function and $A$ stands for the resolving operator of problem~\eqref{stokes} with $u_0\equiv0$. Substituting this expression into~\eqref{1.7}, we derive the following functional equation for~$\zeta$ in the space~$\YY_T$: 
\begin{equation} \label{functional-equation}
	F(\zeta,z):=\zeta+B(z+A\zeta)=0. 
\end{equation}
This is a functional equation in~$\YY_T$ defined by the analytic function~$F$. Due to the unique solvability of the initial value problem for the Navier--Stokes system, we know that~\eqref{functional-equation} has a unique solution $\zeta\in\YY_T$, and we only need to prove that~$\zeta$ is an analytic function of~$z\in\XX_T$.  By the analytic implicit function theorem, this will be established if we prove that the linear operator $(D_\zeta F)(z,\zeta):\YY_T\to\YY_T$ is invertible for any $(z,\zeta)\in\XX_T\times\YY_T$. 

To this end, we first note that 
$$
(D_\zeta F)(z,\zeta)\xi=\xi+B(z+A\zeta,A\xi)+B(A\xi,z+A\zeta), \quad \xi\in\YY_T,
$$
where $B(\cdot,\cdot)$ is the bilinear term in the Navier--Stokes system. By Banach's inverse mapping theorem, the required invertibility of~$D_\zeta F$ will be established if we prove that, for any $f\in\YY_T$, the equation 
\begin{equation} \label{resolving-inverse}
	\xi+B(z+A\zeta,A\xi)+B(A\xi,z+A\zeta)=f
\end{equation}
has a unique solution $\xi\in\YY_T$. Setting $w=A\xi$, we obtain the following problem for~$w$:
\begin{equation} \label{problem-w}
	\p_tw+\nu Lw+B(z+A\zeta,w)+B(w,z+A\zeta)=f, \quad w(0)=0. 
\end{equation}
This is a linear Navier--Stokes-type problem, and its unique solvability in~$\XX_T$ can be established by standard arguments. It remains to note that $\xi=(\p_t+\nu L)w$ is an element of~$\YY_T$, which satisfies~\eqref{resolving-inverse}. 
\end{proof}

\subsection{Resolving operator for the complex Ginzburg--Landau equation}
\label{s-resolvingCGL}
Let us consider the Ginzburg--Landau equation~\eqref{cgl}, supplemented with the initial condition~\eqref{ic-cgl}. For a given $T>0$, we define the spaces 
$$
\XX_T=L^2(J_T,H^2)\cap W^{1,2}(J_T,L^2),\quad \YY_T=L^2(J_T,H^1),
$$
where $H^s$ is the Sobolev space of order~$s$ on~$\T_a^3$. Recall that, even though all the functions of this subsection are complex-valued, we regard~$H^s$ as a real Hilbert space, so that the nonlinear term in~\eqref{cgl} defines a real-analytic map in~$H^2$. The following theorem establishes the well-posedness of the Cauchy problem for the Ginzburg--Landau equation and the analyticity of its resolving operator. 

\begin{theorem} \label{t-IV-CGL}
	For $u_0\in H^1$ and~$\eta\in \YY_T$, problem~\eqref{cgl}, \eqref{ic-cgl} has a unique solution $u\in\XX_T$, and the map $\RR(u_0,\eta)$ acting from $H^1\times\YY_T$ to~$\XX_T$ is analytic. Moreover, for any $\rho>0$ there is $R>0$ such that, if $\eta:\R_+\to H^2$ is a locally square-integrable function, and 
	\begin{equation} \label{CGL-boundsu0eta}
	\|\eta\|_{L^2([k-1,k],H^2)}\le \rho
	\quad \mbox{for all $k\ge1$}, 
	\end{equation}
then the corresponding solution $u(t)$ belongs to~$H^2$ for $t>0$ and satisfies the inequality
\begin{equation} \label{CGL-bound-solution}
	\|u(t)\|_{H^2}\le R\quad \mbox{for $t\ge T_0$}, 
\end{equation}
where $T_0>0$ is a number depending only on the $H^1$-norm of~$u_0$. 
\end{theorem}

\begin{proof}
	We shall confine ourselves to the derivation of some apriori estimates for solutions and to the proof of uniqueness of a solution and its analytic dependence on data. 
	
	\smallskip
	{\it Step~1: Bound for the energy\/}.
	Let us define the Hamiltonian
	$$
	\scH(u)=\int_{\T^3}\bigl(\tfrac12|\nabla u|^2+\tfrac{c}{6}|u|^6\bigr)\dd x
	$$
	and note that it is a Fr\'echet differentiable function on~$H^1$ with the derivative
	$$
	(D\scH)(u;\xi)=(-\Delta u+c|u|^4u, \xi),
	$$
	where the inner product $(\cdot,\cdot)$ is defined by~\eqref{inner-product}. It follows that, for any solution $u\in\XX_T$ of~\eqref{cgl}, we have 
	\begin{align}
		\frac{\dd}{\dd t}\scH(u(t))
		&=(-\Delta u+c|u|^4u,\p_tu)\notag\\
		&=(-\Delta u+c|u|^4u,(\nu+i)\Delta u-\gamma u-ic|u|^4u+\eta)\notag\\
		&=-\nu\|\Delta u\|^2+\nu c\,(|u|^4u,\Delta u)-\gamma\|\nabla u\|^2-\gamma c\|u\|_{L^6}^6\notag\\
		&\qquad+(\nabla u,\nabla \eta)+c(|u|^4u,\eta).\label{derivativeH}
	\end{align}
	Simple calculation combined with the continuity of the embedding $H^1\subset L^6$ implies that
	\begin{align*}
	(|u|^4u,\Delta u)&=-(3|u|^4\nabla u, \nabla u)-(2|u|^2 u^2\nabla \bar u,\nabla u)\le -\bigl(|u|^4,|\nabla u|^2\bigr),\\
	(\nabla u,\nabla \eta)&\le \|\nabla u\|\,\|\nabla \eta\|\le \frac\gamma2\|\nabla u\|^2+\frac{1}{2\gamma}\|\eta\|_1^2,\\
	(|u|^4u,\eta)&\le \|u\|_{L^6}^5\|\eta\|_{L^6}\le C\scH(u)^{5/6}\|\eta\|_1. 
	\end{align*}
Substituting these inequalities in~\eqref{derivativeH} and carrying out some transformations, we derive
\begin{equation} \label{derivativeH-final}
	\frac{\dd}{\dd t}\scH(u)
	\le -\delta \bigl(\scH(u)+\|\Delta u\|^2+(|u|^4,|\nabla u|^2)\bigr)+K\bigl(\scH(u)^{5/6}\|\eta\|_1+\|\eta\|_1^2\bigr), 
\end{equation}
where $\delta$ and~$K$ are some positive numbers. We now need the following version of Gronwall lemma established in the end of this subsection. 

\begin{lemma} \label{l-gronwall}
	Let $\varphi:J_T\to\R$ be an absolutely continuous function minorised by~$1$ such that 
	\begin{equation} \label{gronwall-hypothesis}
		\dot\varphi\le -a\varphi+b+C(t)\varphi^r \quad \mbox{almost everywhere on $J_T$},
	\end{equation}
	where $a$, $b$, and $r<1$ are some positive numbers, and $C:J_T\to\R_+$ is an integrable function. Then there is a number $c_r$ depending only on~$r$ such that 
	\begin{equation}\label{gronwall-conclusion}
		\varphi(t)\le c_r e^{-at}\varphi(0)+c_r\biggl(\frac{b}{a}+(1-r)\int_0^te^{-a(1-r)(t-s)}C(s)\dd s\biggr)^{(1-r)^{-1}}. 
	\end{equation}
\end{lemma}

Let us set $\varphi(t)=\scH(u(t))+1$. It follows from~\eqref{derivativeH-final} that~$\varphi$ satisfies inequality~\eqref{gronwall-hypothesis}, in which $a=b=\delta$, $r=5/6$, and $C(t)=K(\|\eta\|_1+\|\eta\|_1^2)$. Applying Lemma~\ref{l-gronwall}, we derive 
\begin{equation} \label{energy-explicitbound}
	\scH(u(t))\le C_1 e^{-\delta t}\bigl(\scH(u_0)+1\bigr)+C_1\biggl(1+\int_0^te^{-\delta(t-s)/6}\bigl(\|\eta\|_1+\|\eta\|_1^2\bigr)\dd s\biggr)^6, 
\end{equation}
where $t\ge 0$ is arbitrary, and $C_1>0$ is a number not depending on the solution. In particular, if~\eqref{CGL-boundsu0eta} is satisfied, then the energy~$\scH(u(t))$ of any solution is bounded on~$\R_+$, and there is a number~$R_*>0$ such that 
\begin{equation} \label{bound-energy}
	\scH(u(t))\le R_*\quad\mbox{for $t\ge T_*$},
\end{equation}
	where $T_*>0$ depends only on~$\|u_0\|_1$. Moreover, integrating~\eqref{derivativeH-final} in time and using~\eqref{energy-explicitbound}, we can find $C_2>0$ depending only on~$\rho$ such that 
	\begin{equation} \label{boundintegralH2}
	\int_0^t\bigl(\|\Delta u\|^2+(|u|^4,|\nabla u|^2)\bigr)\,\dd s\le C_2(1+t)\scH(u_0)\quad \mbox{for $t\ge0$}.  
	\end{equation}
	
	\smallskip
	{\it Step~2: Uniqueness\/}.
	If $u_1,u_2\in\XX_T$ are two solutions corresponding to the same data $(u_0,\eta)\in H^1\times\YY_T$, then their difference $u=u_1-u_2$ satisfies the equations 
	\begin{equation*} 
		\p_t u-(\nu +i)\Delta u+\gamma u+ic\bigl(|u_1|^4 u_1-|u_2|^4 u_2\bigr)=0, \quad u(0)=0. 
	\end{equation*}
	Taking the inner product in~$L^2$ of the first equation with~$2u$, noting that 
	$$
	\bigl||u_1|^4 u_1-|u_2|^4 u_2\bigr|
	\le C_3\bigl(|u_1|^4+|u_2|^4\bigr)|u|,
	$$
	and carrying out some simple transformations, we derive
	\begin{equation} \label{uniquenss-diff}
		\p_t\|u\|^2\le -\gamma\|u\|^2-\nu\|\nabla u\|^2+C_4\bigl(\|u_1\|_{L^\infty}^4+\|u_2\|_{L^\infty}^4\bigr)\|u\|^2. 
	\end{equation}
	Since $\|u_i\|_{L^\infty}^2\le C_5\|u_i\|_1\|u_i\|_2$ and 
	$$
	\|u_i\|_{L^1(J_T,L^\infty)}^4\le C_5\|u_i\|_{C(J_T,H^1)}^2\|u_i\|_{L^2(J_T,H^2)}^2, 
	$$
	we see that the coefficient in front of~$\|u\|^2$ is an integrable function of time. Applying Gronwall's inequality to~\eqref{uniquenss-diff}, we conclude that $u\equiv0$. 
	
	\smallskip
	{\it Step~3: Analyticity\/}. 
	To prove the analyticity of the resolving operator for~\eqref{cgl}, \eqref{ic-cgl}, we repeat the scheme used in the case of the Navier--Stokes system. Namely, let us denote by $z\in \XX_T$ the solution of the linear problem
	\begin{equation} \label{cgl-linear}
		\p_t z+\gamma z-(\nu+i)\Delta z=\eta, \quad z(0)=u_0, 
	\end{equation}
	and denote by $A:\YY_T\to\XX_T$ its resolving operator when $u_0=0$. We seek a solution of~\eqref{cgl}, \eqref{ic-cgl} in the form $u=z+A\zeta$, where $\zeta\in\YY_T$ is an unknown function. Substituting this expression into~\eqref{cgl} and setting $B(u)=ic|u|^4u$, we see that~$\zeta$ must be a solution of \eqref{functional-equation}. We know that~\eqref{functional-equation} has a unique solution $\zeta\in\YY_T$ for any $z\in\XX_T$; cf.\ Steps~1 and~2. If we prove that~$F$ is an analytic map in~$\YY_T$ whose derivative with respect to~$\zeta$ is an invertible linear operator in~$\YY_T$, then the required result will follow from the analytic implicit function theorem (see Section~3.3B in~\cite{berger1977}). The analyticity is a straightforward consequence of the facts that~$B$ is a continuous multilinear map from~$\XX_T$ to~$\YY_T$ and~$A$ is a continuous linear operator from~$\YY_T$ to~$\XX_T$. We thus need to prove the invertibility of the derivative~$D_\zeta F$.  

Let us note that
$$
(D_\zeta F)(\zeta,z)\xi=\xi+3|z+A\zeta|^4A\xi+2|z+A\zeta|^2(z+A\zeta)^2\overline{A\xi}, \quad \xi\in\YY_T. 
$$
We now fix an arbitrary $f\in\YY_T$ and consider the equation
\begin{equation*}
	\xi+3|z+A\zeta|^4A\xi+2|z+A\zeta|^2(z+A\zeta)^2\overline{A\xi}=f. 
\end{equation*}
Setting $g=z+A\zeta\in\XX_T$ and $w=A\xi$, so that $w$ vanishes at $t=0$, we rewrite this equation in the form
\begin{equation} \label{cgl-derivative}
	\p_tw+\gamma w-(\nu+i)\Delta w+a(t,x)w+b(t,x)\bar w=f, \quad w(0)=0, 
\end{equation}
where $a=3|g|^4$ and $b=2|g|^2g^2$ are elements of the space $L^2(J_T,L^\infty)$.  It remains to note that~\eqref{cgl-derivative} is a zero-order perturbation of the linear problem~\eqref{cgl-linear}, and its unique solvability can easily be established by standard arguments.

	\smallskip
	{\it Step~4: Bound for the $H^2$ norm\/}. We now prove~\eqref{CGL-bound-solution}. In view of~\eqref{bound-energy}, it suffices to prove that the $H^2$-norm of a solution of~\eqref{cgl}, \eqref{ic-cgl} can be bounded at time $t=1$ in terms of the $H^1$-norm of~$u_0$. To this end, we calculate the derivative of the function $h(t)=t\|\Delta u(t)\|^2$. A simple calculation shows that 
\begin{equation} \label{H2norm-derivation}
	\p_t h\le \|\Delta u\|^2+t\|\Delta u\|\bigl(\|\Delta B(u)\|+\|\eta\|_2\bigr), 
\end{equation}
where $B(u)$ was defined in Step~3. Now note that
$$
\|\Delta B(u)\|\le C_6\bigl(\|u\|_{L^\infty}^4+1\bigr)\|\Delta u\|
\le C_7\bigl(U(t)+1\bigr)\|\Delta u\|,
$$
where $U(t)=\|u\|_1^2\|u\|_2^2$ is a locally integrable function of time. Substituting this inequality into~\eqref{H2norm-derivation}, we obtain
$$
\p_t h\le C_8\bigl(U(t)+1\bigr)h+\|\Delta u\|^2+\|\eta\|_2^2. 
$$
Applying Gronwall's inequality and using~\eqref{boundintegralH2}, we arrive at the required estimate. This completes the proof of the theorem. 
\end{proof}

\begin{proof}[Proof of Lemma~\ref{l-gronwall}]
Let us set $\alpha=(1-r)^{-1}$ and $\psi=\varphi^{1/\alpha}$. Then $\psi$ is an absolutely continuous function on~$J_T$. Substituting $\varphi=\psi^\alpha$ into~\eqref{gronwall-hypothesis}, we derive 
	$$
	\alpha\psi^{\alpha-1}\dot\psi\le -a\psi^\alpha+b+C(t)\psi^{\alpha r}. 
	$$
	Dividing this inequality by~$\alpha\psi^{\alpha-1}$ and using the relations  $\alpha r-\alpha+1=0$ and $\psi\ge1$, we obtain
	$$
	\dot \psi\le -a(1-r)\psi+\bigl(b+C(t)\bigr)(1-r). 
	$$
	Application of the usual Gronwall inequality results in
	$$
	\psi(t)\le \bigl(e^{-at}\varphi(0)\bigr)^{1-r}+\frac{b}{a}+(1-r)\int_0^t e^{-a(1-r)(t-s)}C(s)\,\dd s. 
	$$
	This implies the required inequality~\eqref{gronwall-conclusion} with $c_r=2^{r/(1-r)}$. 
\end{proof}

\addcontentsline{toc}{section}{Bibliography}
\def\cprime{$'$} \def\cprime{$'$}
  \def\polhk#1{\setbox0=\hbox{#1}{\ooalign{\hidewidth
  \lower1.5ex\hbox{`}\hidewidth\crcr\unhbox0}}}
  \def\polhk#1{\setbox0=\hbox{#1}{\ooalign{\hidewidth
  \lower1.5ex\hbox{`}\hidewidth\crcr\unhbox0}}}
  \def\polhk#1{\setbox0=\hbox{#1}{\ooalign{\hidewidth
  \lower1.5ex\hbox{`}\hidewidth\crcr\unhbox0}}} \def\cprime{$'$}
  \def\polhk#1{\setbox0=\hbox{#1}{\ooalign{\hidewidth
  \lower1.5ex\hbox{`}\hidewidth\crcr\unhbox0}}} \def\cprime{$'$}
  \def\cprime{$'$} \def\cprime{$'$} \def\cprime{$'$}
\providecommand{\bysame}{\leavevmode\hbox to3em{\hrulefill}\thinspace}
\providecommand{\MR}{\relax\ifhmode\unskip\space\fi MR }
\providecommand{\MRhref}[2]{%
  \href{http://www.ams.org/mathscinet-getitem?mr=#1}{#2}
}
\providecommand{\href}[2]{#2}

\end{document}